\documentclass[twoside,centertags                 % ,draft
]{amsart}

\def\ord#1{[#1]}

\usepackage[cmtip,color,all,
]{xy}

\usepackage{float}

\entrymodifiers={+!!<0pt,\fontdimen22\textfont2>}

\usepackage{bbm}
\usepackage{amssymb}
\usepackage{eucal}
\usepackage{stmaryrd}
\usepackage{graphicx}
\usepackage{mathrsfs}
\usepackage{wasysym}

\usepackage[usenames,dvipsnames]{pstricks}
\usepackage{epsfig}
\usepackage{pst-grad} % For gradients
\usepackage{pst-plot} % For axes

\usepackage{rotating}

\newtheorem{thm}[equation]{Theorem}
\newtheorem{cor}[equation]{Corollary}

\newtheorem{prop}[equation]{Proposition}
\theoremstyle{definition}

\newtheorem{rem}[equation]{Remark}

\newcommand{\norm}[1]{\left\Vert#1\right\Vert}
\newcommand{\abs}[1]{\left\vert#1\right\vert}

\newcommand{\eps}{\varepsilon}
\newcommand{\To}{\longrightarrow}

\newcommand{\id}[1]{\operatorname{id}_{#1}}

\def\r{\rightarrow} % flecha -->
\def\into{\rightarrowtail}

\def\st{\stackrel} % abreviatura de \stackrel

\numberwithin{equation}{section}
\newcommand{\ass}[4]{\text{ass.}}
\newcommand{\com}[3]{\text{sym.}}

\newcommand{\hpp}[3]{\text{mult.}}
\newcommand{\up}[1]{\text{unit.}}

\newdir{ >}{{}*!/-7pt/@{>}}
\newdir{> }{{}*!/10pt/@{>}}
\newdir{>> }{{}*!/10pt/@{>>}}

\newcommand{\card}{\operatorname{card}}

\newcommand{\level}[1]{\operatorname{level}(#1)}

\begin{document}

%\newxyColor{cola}{0.1 0.6 0.1}{rgb}{}
%\newxyColor{colb}{0.8 0.1 0.1}{rgb}{}
%\newxyColor{colc}{0.1 0.1 0.8}{rgb}{}
%\newxyColor{cold}{0.4 0.1 0.4}{rgb}{}

%\newxyColor{cole}{.8 .2 1}{rgb}{}
%\newxyColor{colf}{0.8 0.6 0.0}{rgb}{}

\title%[Determinant functors \dots]
{Unital associahedra}%
\author{Fernando Muro}%
\address{Universidad de Sevilla,
Facultad de Matem\'aticas,
Departamento de \'Algebra,
Avda. Reina Mercedes s/n,
41012 Sevilla, Spain}
\email{fmuro@us.es}
\author{Andrew Tonks}%
\address{Faculty of Computing, London Metropolitan University, 166--220
Holloway Road, London N7 8DB, United Kingdom}
\email{a.tonks@londonmet.ac.uk}

\thanks{The authors were partially supported
by the Spanish Ministry of Education and
Science under the MEC-FEDER grant  MTM2010-15831 and by the Government of Catalonia under the grant SGR-119-2009. The first author was also partially supported by the Andalusian Ministry of Economy, Innovation and Science under the grant FQM-5713.}

% ----------------------------------------------------------------
\begin{abstract}
We construct a topological cellular operad such that the algebras over its cellular chains are the homotopy unital $A_{\infty}$-algebras of Fukaya--Oh--Ohta--Ono. 
\end{abstract}

\maketitle
\tableofcontents

% ----------------------------------------------------------------

%\numberwithin{equation}{subsection}

\section*{Introduction}

The associahedra are spaces introduced by Stasheff to parametrize natural multivariate operations on loop spaces \cite{hahs}. These natural operations characterize connected loop spaces. This algebraic setting originated the theory of operads and their algebras \cite{tgils}.  

As Milnor first discovered (unpublished), the associahedra can be realized as polytopes in such a way that the operad laws are inclusions of faces.  Algebras over the differential graded (DG) operad obtained by taking cellular chains on associahedra are $A_{\infty}$-algebras, i.e.\ strongly homotopy associative algebras. These algebras are a good replacement for non-unital DG-algebras. 

Bringing coherent units into the picture turned out to be a more complicated task. The homotopy coherent notion of unital DG-algebra was recently introduced by Fukaya--Oh--Ohta--Ono in their work on symplectic geometry in a purely algebraic way \cite{fooo1,fooo2}. 
The same notion arises in Koszul duality theory for DG-operads 
with quadratic, linear and constant relations over a field of characteristic zero 
\cite{ckdt}. 

In this paper we construct a topological cellular operad such that the algebras over its cellular chains over any commutative ring are the homotopy unital $A_{\infty}$-algebras of Fukaya--Oh--Ohta--Ono. The cell complexes forming this operad are therefore called \emph{unital associahedra}. These cell complexes are contractible, but unlike in the non-unital case, unital associahedra are not finite-dimensional and their skeleta cannot be realized as polytopes, see for instance Figure \ref{discochino}.

\section{Operads}

A \emph{topological operad} $\mathtt{O}$ is a sequence of topological spaces $\mathtt{O}(n)$, $n\geq 0$, together with composition laws 
$$\circ_i\colon \mathtt{O}(m)\times \mathtt{O}(n)\To \mathtt{O}(n+m-1),\qquad 1\leq i\leq m,\quad n\geq 0,$$
and an element $u\in \mathtt{O}(1)$ satisfying the following relations:
\begin{enumerate}
\item $(a\circ_i b)\circ_j c=(a\circ_j c)\circ_{i+n-1}b$ if $1\leq j<i$ and $c\in\mathtt{O}(n)$.
\item $(a\circ_i b)\circ_j c=a\circ_i( b\circ_{j-i+1}c)$ if $b\in\mathtt{O}(m)$ and $i\leq j <m+i$.
\item $u\circ_1 a=a$.
\item $a\circ_i u=a$.
\end{enumerate}
%The operad $\mathtt{O}$ is \emph{filtered} if the spaces $\mathtt{O}(n)$ are endowed with increasing filtrations
%$$\mathtt{O}(n)_{0}\subset \cdots\subset \mathtt{O}(n)_{m}\subset \mathtt{O}(n)_{m+1}\subset\cdots\subset \mathtt{O}(n),\qquad \mathtt{O}(n)=\bigcup_{m=0}^{\infty}\mathtt{O}(n)_{m},$$
%such that $$\mathtt{O}(p)_{s}\circ_{i} \mathtt{O}(q)_{t}\subset \mathtt{O}(p+q-1)_{s+t},\qquad 1\leq i\leq p.$$

Let $\Bbbk$ be a commutative ring. A \emph{graded operad} is similarly defined, replacing spaces by graded $\Bbbk$-modules and the cartesian product $\times$ by the usual $\Bbbk$-linear tensor product of graded $\Bbbk$-modules $\otimes_\Bbbk$, which follows the Koszul sign rule. Therefore the first equation must be replaced with
\begin{itemize}
\item[$(1')$] $(a\circ_i b)\circ_j c=(-1)^{|b||c|}(a\circ_j c)\circ_{i+n-1}b$ if $1\leq j<i$ and $c\in\mathtt{O}(n)$.
\end{itemize}
The other three equations do not change. The unit must have degree $0$.  A \emph{differential graded operad} is defined in the same way, endowing graded $\Bbbk$-modules with differentials. The unit must be a cycle. %One can also consider filtered operads in this context.

Any differential graded operad has an underlying graded operad. If $\mathtt{O}$ is a topological operad made up from cell complexes in such a way that $u\in\mathtt{O}(1)$ is a vertex and the composition laws are cellular maps, then the cellular chain complexes with coefficients in $\Bbbk$, $C_*(\mathtt{O}(n),\Bbbk)$, form a differential graded operad $C_*(\mathtt{O},\Bbbk)$ in the obvious way. %If $\mathtt{O}$ is filtered by subcomplexes then $C_*(\mathtt{O},\Bbbk)$ is also filtered.

These structures often appear in the literature under the name  \emph{non-symmetric} operad. We omit the adjective since we will not use any other kind of operad in this paper. 

\section{Trees}

A \emph{planted tree with leaves}
is a contractible finite $1$-dimensional simplicial complex $T$ 
with a set of vertices $V(T)$, 
%equipped with a total order $\preceq$, 
a non-empty set of edges $E(T)$, 
a distinguished vertex $r(T)\in V(T)$ called \emph{root}, 
and a set of distinguished vertices $L(T)\subset V(T)\setminus\{r(T)\}$ called \emph{leaves}. The root and the leaves must have degree $1$. Recall that the \emph{degree} of $v\in V(T)$ is the number of edges containing $v$. %Nevertheless, we will mostly use the following number,
%$$\val{T}{v}=(\text{degree of }v)-1.$$
%When $T$ is understood we abbreviate $\val{}{v}=\val{T}{v}$.
The other degree one vertices are called \emph{corks}. 

The \emph{level} of a vertex $v\in V(T)$ is the distance to the root, $\level{v}=d(v,r(T))$, with respect to the usual metric $d$ such that the distance between two adjacent vertices $\{u,v\}\in E(T)$ is $d(u,v)=1$. The \emph{height} $\operatorname{ht}(T)$ of a planted  tree with leaves $T$ is $$\operatorname{ht}(T)=\max_{v\in V(T)}\level{v}.$$

A \emph{planted planar  tree} with leaves is a planted tree with leaves $T$ together with a total order   
$\preceq$ on $V(T)$, called \emph{path order}, satisfying the following conditions. Given two vertices $v,w\in V(T)$:
\begin{itemize}
 \item If $v$ lies on the (shortest) path from $r(T)$ to $w$ then $v\prec w$.

% NO! v3 v8 v6 en fig 1 contradice estas defns!
 %given vertices $u,v,w\in V(T)$ such that $\{u,v\}\in E(T)$ and $\level{u}<\level{v}$ then $u\prec v$, and if in addition $\{u,w\}\notin E(T)$ and $u\prec w$ then $v\prec w$.
%
%\item If $w\in V(T)$ and $u\prec  w\prec v$ then $\{u,w\}\in E(T)$.
\item  Otherwise, assume that the path from $r(T)$ to $v$ coincides with the path from $r(T)$ to $w$ up to level $n$, and let $v '$ and $w'$ be the level $n+1$ vertices on these paths. If $v'\prec w'$ then $v\prec w$.
\end{itemize}
The set $E(T)$ is ordered according to the top vertex of each edge.
Our definition of planar structure is equivalent to giving a local ordering on the set of descendent edges at each vertex. However the global path order is essential in what follows.

\smallskip

All trees in this paper will be planted planar  trees with leaves. 
The geometric realization $|T|$ of a tree is depicted drawing vertices with the same level on the same horizontal line, following from left to right the order induced by the path order (see Figure~\ref{ppt}). 
For most purposes, the heuristic picture we should have in mind of a tree $T$ corresponds to the space  $\norm{T}=|T|\setminus\left(\{r(T)\}\sqcup L(T)\right)$.

\begin{figure}[H]
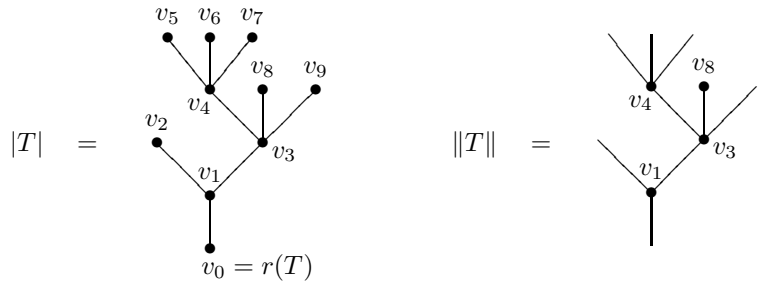

$$|T|\quad=\quad\begin{array}{c}
\xy/r2pt/:
(0,0)*-{\bullet}="a",
(9,-3.5)*{v_{0}=r(T)},
(0,10)*-<2.8pt>{\bullet}="b",
(0,14)*{v_1},
(-10,20)*-{\bullet}="c",
(-10,24)*{v_2},
(10,20)*-<3pt>{\bullet}="d",
(14,18)*{v_3},
(0,30)*-<3pt>{\bullet}="e",
(-2,27)*{v_4},
(10,30)*-{\bullet}="f",
(10,34)*{v_8},
(20,30)*-{\bullet}="g",
(20,34)*{v_9},
(-8,40)*-{\bullet}="h",
(-8,44)*{v_5},
(0,40)*-{\bullet}="i",
(0,44)*{v_6},
(8,40)*-{\bullet}="j",
(8,44)*{v_7},
\ar@{-}"a";"b",
\ar@{-}"b";"c",
\ar@{-}"b";"d",
\ar@{-}"d";"e",
\ar@{-}"d";"f",
\ar@{-}"d";"g",
\ar@{-}"e";"h",
\ar@{-}"e";"i",
\ar@{-}"e";"j",
\endxy
\end{array}
\qquad
\qquad
\norm{T}\quad=\quad\begin{array}{c}
\xy/r2pt/:
(0,0)*{}="a",
%(0,-3)*{a},
(0,10)*-<3pt>{\bullet}="b",
(0,14)*{v_1},
(-10,20)*{}="c",
%(-10,23)*{c},
(10,20)*-<3pt>{\bullet}="d",
(14,18)*{v_3},
(0,30)*-<3pt>{\bullet}="e",
(-2,27)*{v_4},
(10,30)*-{\bullet}="f",
(10,34)*{v_8},
(20,30)*{}="g",
%(20,33)*{g},
(-8,40)*{}="h",
%(-10,43)*{h},
(0,40)*{}="i",
%(0,43)*{i},
(8,40)*{}="j",
%(10,43)*{j},
\ar@{-}"a";"b",
\ar@{-}"b";"c",
\ar@{-}"b";"d",
\ar@{-}"d";"e",
\ar@{-}"d";"f",
\ar@{-}"d";"g",
\ar@{-}"e";"h",
\ar@{-}"e";"i",
\ar@{-}"e";"j",
\endxy
\end{array}
$$
\caption{On the left, the geometric realization of a tree $T$ with vertices ordered by the subscript. The set of leaves is $L(T)=\{v_2,v_5,v_6,v_7,v_9\}$. On the right, the space $\norm{T}$, where we can see the only cork $v_8$.}
\label{ppt}
\end{figure}

%Given $e=\{u\prec v\}\in E(T)$ we say that $e$ is \emph{an incoming edge} of $u$ and \emph{the outgoing edge} of $v$ (there is only one if $w\neq r(T)$ and none otherwise). 
An \emph{inner vertex} is a vertex which is neither a leaf nor the root, that is,  it has degree $>1$ or is a cork. An \emph{inner edge} is an edge whose vertices are inner. Denote $I(T)\subset E(T)$ the set of inner edges.  Abusing terminology, we say that a non-inner edge is the \emph{root} or a \emph{leaf} if it contains the root or a leaf vertex, respectively.
If $e\in I(T)$, the quotient $T/e$ has the same root and leaves, and $\norm{T/e}$ is depicted contracting $e$ %to $v$
 and moving the vertices over $e%v
$ one level downwards (see Figure~\ref{pptl2}). The path order is the quotient order.

\begin{figure}[H]
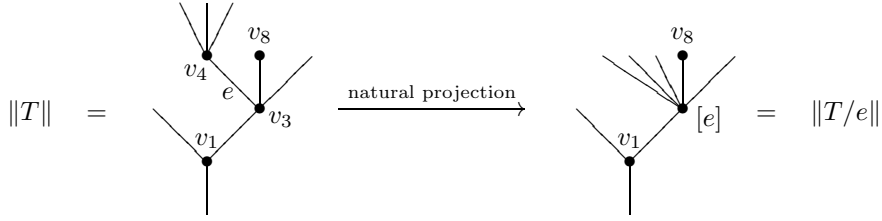

$$\norm{T}\quad=\quad\begin{array}{c}
\xy/r2pt/:
(0,0)*{}="a",
%(0,-3)*{a},
(0,10)*-<3pt>{\bullet}="b",
(0,14)*{v_1},
(-10,20)*{}="c",
%(-10,23)*{c},
(10,20)*-<3pt>{\bullet}="d",
(14,18)*{v_3},
(0,30)*-<3pt>{\bullet}="e",
(4,23)*{e},
(-2,27)*{v_4},
(10,30)*-<3pt>{\bullet}="f",
(10,34)*{v_8},
(20,30)*{}="g",
%(20,33)*{g},
(-5,40)*{}="h",
%(-10,43)*{h},
(0,40)*{}="i",
%(0,43)*{i},
(5,40)*{}="j",
%(10,43)*{j},
(80,0)*{}="aa",
%(0,-3)*{a},
(80,10)*-<3pt>{\bullet}="bb",
(80,14)*{v_1},
(70,20)*{}="cc",
%(-10,23)*{c},
(90,20)*-<3pt>{\bullet}="dd",
(95,18)*{[e]},
%(50,30)*-<3pt>{\bullet}="ee",
%(48,28)*{v_4},
(90,30)*-<3pt>{\bullet}="ff",
(90,34)*{v_8},
(100,30)*{}="gg",
%(20,33)*{g},
(75,30)*{}="hh",
%(-10,43)*{h},
(80,30)*{}="ii",
%(0,43)*{i},
(85,30)*{}="jj",
%(10,43)*{j},
\ar@{-}"a";"b",
\ar@{-}"b";"c",
\ar@{-}"b";"d",
\ar@{-}"d";"e",
\ar@{-}"d";"f",
\ar@{-}"d";"g",
\ar@{-}"e";"h",
\ar@{-}"e";"i",
\ar@{-}"e";"j",
\ar@{-}"aa";"bb",
\ar@{-}"bb";"cc",
\ar@{-}"bb";"dd",
%\ar@{-}"dd";"ee",
\ar@{-}"dd";"ff",
\ar@{-}"dd";"gg",
\ar@{-}"dd";"hh",
\ar@{-}"dd";"ii",
\ar@{-}"dd";"jj",
\ar(25,20);(60,20)^{\text{natural projection}}
\endxy
\end{array}=\quad\norm{T/e}$$
\caption{The  natural projection $T\r T/e$ contracting the inner edge $e=\{v_{3},v_{4}\}$.}
\label{pptl2}
\end{figure}

Morphisms in the category of trees are generated by these natural projections $T\r T/e$ and by the simplicial isomorphisms preserving the root, the leaves, and the path order. Notice that there are no non-trivial automorphisms in this category. From now on we consider a skeletal subcategory by choosing exactly one tree in each isomorphism class.
Since natural projections at different inner edges commute, this is a poset.

Given trees $T$ and $T'$ with $p$ and $q$ leaves, respectively, and $1\leq i\leq p$, the tree $T\circ_i T'$ is obtained by \emph{grafting} the root edge of $\norm{T'}$ onto the $i^{\text{th}}$ leaf edge of $\norm{T}$. 
The path order in $V(T\circ_i T')$ is obtained by inserting $V(T')\setminus\{r(T')\}$ into $V(T)$ in place of the $i^{\text{th}}$ leaf vertex.
%If $v\in L(T)$ is the $i^{\text{th}}$ leaf, the path order in $V(T\circ_i T')$ is obtained by inserting $V(T')\setminus\{r(T')\}$ in $V(T)\setminus\{v\}$ in the place where $v$ was.

\begin{figure}[H]
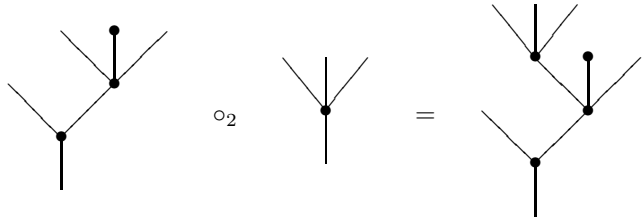

$$
\begin{array}{c}
\xy/r2pt/:
(0,0)*{}="a",
%(0,-3)*{a},
(0,10)*-<3pt>{\bullet}="b",
(-10,20)*{}="c",
%(-10,23)*{c},
(10,20)*-<3pt>{\bullet}="d",
(0,30)*{}="e",
(10,30)*-{\bullet}="f",
(20,30)*{}="g",
%(20,33)*{g},
%(-10,43)*{h},
%(0,43)*{i},
%(10,43)*{j},
\ar@{-}"a";"b",
\ar@{-}"b";"c",
\ar@{-}"b";"d",
\ar@{-}"d";"e",
\ar@{-}"d";"f",
\ar@{-}"d";"g",
\endxy
\end{array}
\quad\circ_2\quad
\begin{array}{c}
\xy/r2pt/:
(0,0)*{}="d",
(0,10)*-<3pt>{\bullet}="e",
%(20,33)*{g},
(-8,20)*{}="h",
%(-10,43)*{h},
(0,20)*{}="i",
%(0,43)*{i},
(8,20)*{}="j",
%(10,43)*{j},
\ar@{-}"d";"e",
\ar@{-}"e";"h",
\ar@{-}"e";"i",
\ar@{-}"e";"j",
\endxy
\end{array}
\quad=\quad
\begin{array}{c}
\xy/r2pt/:
(0,0)*{}="a",
%(0,-3)*{a},
(0,10)*-<3pt>{\bullet}="b",
(-10,20)*{}="c",
%(-10,23)*{c},
(10,20)*-<3pt>{\bullet}="d",
(0,30)*-<3pt>{\bullet}="e",
(10,30)*-{\bullet}="f",
(20,30)*{}="g",
%(20,33)*{g},
(-8,40)*{}="h",
%(-10,43)*{h},
(0,40)*{}="i",
%(0,43)*{i},
(8,40)*{}="j",
%(10,43)*{j},
\ar@{-}"a";"b",
\ar@{-}"b";"c",
\ar@{-}"b";"d",
\ar@{-}"d";"e",
\ar@{-}"d";"f",
\ar@{-}"d";"g",
\ar@{-}"e";"h",
\ar@{-}"e";"i",
\ar@{-}"e";"j",
\endxy
\end{array}
$$
\caption{An example of grafting.}
\label{agrafting}
\end{figure}

The sets $\mathtt T(n)$ of trees with $n$ leaves, with the composition given by grafting and unit $|\in \mathtt T(1)$, form a discrete operad $\mathtt T$.
In fact $\mathtt T$ is the free operad generated by 
the tree of height 1 and 1 cork, 
together with
the trees of height 2 with no corks,
also known as \emph{corollas}.

\begin{figure}[H]
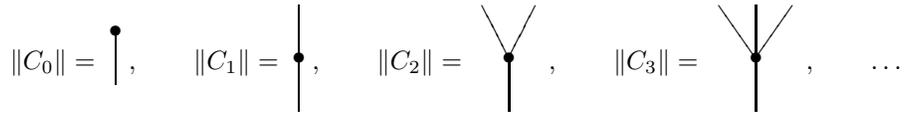

$$\norm{C_{0}}={}\begin{array}{c}
\xy/r2pt/:
(0,0)*{}="a",
%(0,-3)*{a},
(0,10)*-{\bullet}="b",
\ar@{-}"a";"b",
\endxy
\end{array},\qquad
\norm{C_{1}}={}\begin{array}{c}
\xy/r2pt/:
(0,0)*{}="a",
%(0,-3)*{a},
(0,10)*-{\bullet}="b",
(0,20)*{}="c",
\ar@{-}"a";"b",
\ar@{-}"b";"c",
\endxy
\end{array},\qquad
\norm{C_{2}}={}\begin{array}{c}
\xy/r2pt/:
(0,0)*{}="a",
%(0,-3)*{a},
(0,10)*-<3pt>{\bullet}="b",
(-5,20)*{}="c",
(5,20)*{}="d",
\ar@{-}"a";"b",
\ar@{-}"b";"c",
\ar@{-}"b";"d",
\endxy
\end{array},\qquad
\norm{C_{3}}={}\begin{array}{c}
\xy/r2pt/:
(0,0)*{}="a",
%(0,-3)*{a},
(0,10)*-<3pt>{\bullet}="b",
(0,20)*{}="c",
(-7,20)*{}="d",
(7,20)*{}="e",
\ar@{-}"a";"b",
\ar@{-}"b";"c",
\ar@{-}"b";"d",
\ar@{-}"b";"e",
\endxy
\end{array},\qquad\dots$$
\caption{The corollas~$C_{n}$, $n\geq 0$.}
\label{corollas}
\end{figure}
In particular, let $S=\{j_1<\dots<j_m\}\subseteq [n+m]$, where we write $[k]=\{1,\dots,k\}$. 
Then grafting $C_0$ successively at the positions given by $S$ defines a function
\begin{align*}
\mathtt T(n+m)&\longrightarrow \mathtt T(n)\\
T&\longmapsto T^{\bullet S}
:=((T\circ_{j_m}C_0) \circ_{j_{m-1}}C_0) \cdots \circ_{j_1}C_0. 
\end{align*}
We say that $T^{\bullet S}$ is obtained by adding corks to $T$ in the leaves at places $S$.

%Suppose that $T$ and $T'$ have $m$ and $n$ inner edges, respectively. Denote $\tau_{T,T',i}\in S_{n+m}$ of the permutation taking the restriction of the path order $T\circ_i T'$ to $I(T\circ_i T')\setminus \{r(T')\}=I(T)\sqcup Y(T')$ to the order induced by first taking the inner edges in $T$ and then the inner edges in $T'$ with their path order.

Let $T$ be a \emph{binary tree}, that is, a tree in which all vertices have degree $1$ or $3$. If %$T$ has at least three leaves and 
$e=\{u\prec v\}\in E(T)$ is an edge such that $u$ has degree 3 and $v$ has degree $1$, we define the binary tree $T\backslash e$ in such a way that $\norm{T\backslash e}$ is obtained by removing $v$ and $e$ and deleting $u$ in order to join the other two edges $e'$ and $e''$ incident in $u$% (see Figure \ref{tminuse})
.

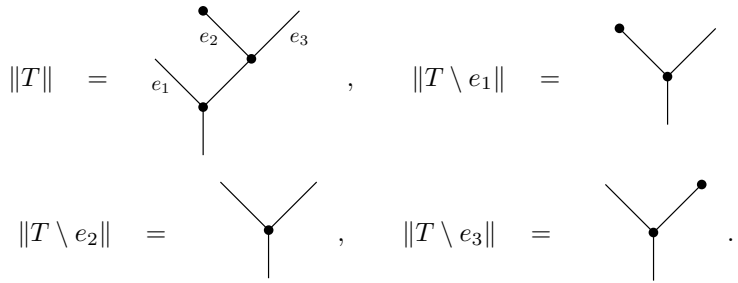
\begin{figure}[H]
$$\norm{T}\quad=\quad\begin{array}{c}           
\scalebox{.8} % Change this value to rescale the drawing.
{\begin{pspicture}(0,-1.3435937)(2.84875,1.3735938)
\psline[linewidth=0.02cm](0.0,0.26640624)(0.8,-0.5335938)
\psline[linewidth=0.02cm](0.8,-0.5335938)(0.8,-1.3335937)
\psline[linewidth=0.02cm](0.8,-0.5335938)(1.6,0.26640624)
\psline[linewidth=0.02cm](1.6,0.26640624)(0.8,1.0664062)
\psline[linewidth=0.02cm](1.6,0.26640624)(2.4,1.0664062)
\psdots[dotsize=0.16](0.8,-0.5335938)
\psdots[dotsize=0.16](1.6,0.26640624)
\psdots[dotsize=0.16](0.8,1.0664062)
\usefont{T1}{ptm}{m}{n}
\rput(0.1,-0.2){$e_1$}
\usefont{T1}{ptm}{m}{n}
\rput(.9,.6){$e_2$}
\usefont{T1}{ptm}{m}{n}
\rput(2.4,0.6){$e_3$}
\end{pspicture} 
}\end{array},
\qquad\norm{T\setminus e_1}\quad=\quad
\begin{array}{c}
\scalebox{.8} % Change this value to rescale the drawing.
{
\begin{pspicture}(0,-0.85)(1.69,0.85)
\psline[linewidth=0.02cm](0.88,-0.84)(0.88,-0.04)
\psline[linewidth=0.02cm](0.88,-0.04)(1.68,0.76)
\psline[linewidth=0.02cm](0.88,-0.04)(0.08,0.76)
\psdots[dotsize=0.16](0.88,-0.04)
\psdots[dotsize=0.16](0.08,0.76)
\end{pspicture} 
} 
\end{array},$$
$$\norm{T\setminus e_2}\quad=\quad
\begin{array}{c}
\scalebox{.8} % Change this value to rescale the drawing.
{
\begin{pspicture}(0,-0.81)(1.61,0.81)
\psline[linewidth=0.02cm](0.8,-0.8)(0.8,0.0)
\psline[linewidth=0.02cm](0.8,0.0)(1.6,0.8)
\psline[linewidth=0.02cm](0.8,0.0)(0.0,0.8)
\psdots[dotsize=0.16](0.8,0.0)
\end{pspicture} 
} 
\end{array}
,\qquad
\norm{T\setminus e_3}\quad=\quad
\begin{array}{c}
\scalebox{.8} % Change this value to rescale the drawing.
{
\begin{pspicture}(0,-0.85)(1.69,0.85)
\psline[linewidth=0.02cm](0.8,-0.84)(0.8,-0.04)
\psline[linewidth=0.02cm](0.8,-0.04)(1.6,0.76)
\psline[linewidth=0.02cm](0.8,-0.04)(0.0,0.76)
\psdots[dotsize=0.16](0.8,-0.04)
\psdots[dotsize=0.16](1.6,0.76)
\end{pspicture} 
}
\end{array}
.$$
\caption{For the tree $T$, we illustrate the operation $T\setminus e_i$ for all edges $e_i$ whose top vertex has degree $1$.}
\label{tminuse}
\end{figure}

The number of binary trees with $n$ leaves and no corks is the Catalan number $C_{n-1}=\frac{1}{n}\binom{2n-2}{n-1}$. Such a tree has $n-2$ inner edges. More generally, there are $\binom{n+m}{m}C_{n+m-1}$ binary trees with $n$ leaves and $m$ corks, and such a tree has $2m+n-2$ inner edges.

%%%%%%%%%%%%%%%%%%%%%%%%%%%%%%%%%%%%%%%%%%%%%%%%%%%%%%%%%%%%%%%%%%%%%%%%
%\ignore{  %%%%%%%%%%%%%%%%%%%%%%%%%%%%%%%%%%%%%%%%%%%%%%%%%%%%%%%%%%%%%
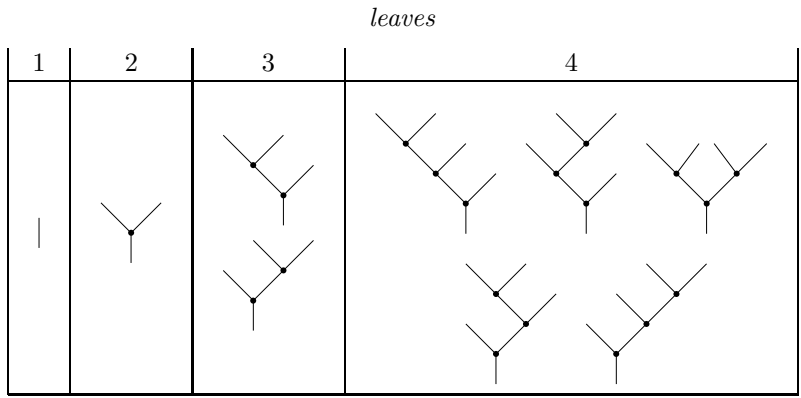
\begin{figure}[H]
$$
\begin{array}{c}
\text{\emph{leaves}}\\[5pt]
\begin{array}{|c|c|c|c|}
1&2&3&4\\
\hline
\begin{array}{c}
\scalebox{.5} % Change this value to rescale the drawing.
{
\begin{pspicture}(0,-0.41)(0.01,0.41)
\psline[linewidth=0.02cm](0.0,-0.4)(0.0,0.4)
\end{pspicture} 
}
\end{array}

&
\begin{array}{c}
\scalebox{.5} % Change this value to rescale the drawing.
{
\begin{pspicture}(0,-0.81)(1.61,0.81)
\psline[linewidth=0.02cm](0.8,-0.8)(0.8,0.0)
\psline[linewidth=0.02cm](0.8,0.0)(1.6,0.8)
\psline[linewidth=0.02cm](0.8,0.0)(0.0,0.8)
\psdots[dotsize=0.16](0.8,0.0)
\end{pspicture} 
}

\end{array}

&
\begin{array}{c}
\scalebox{.5} % Change this value to rescale the drawing.
{
\begin{pspicture}(0,-2.61)(2.41,2.61)
\psline[linewidth=0.02cm](1.6,0.2)(1.6,1.0)
\psline[linewidth=0.02cm](1.6,1.0)(2.4,1.8)
\psline[linewidth=0.02cm](1.6,1.0)(0.8,1.8)
\psline[linewidth=0.02cm](0.8,1.8)(1.6,2.6)
\psline[linewidth=0.02cm](0.8,1.8)(0.0,2.6)
\psline[linewidth=0.02cm](0.0,-1.0)(0.8,-1.8)
\psline[linewidth=0.02cm](0.8,-1.8)(0.8,-2.6)
\psline[linewidth=0.02cm](0.8,-1.8)(1.6,-1.0)
\psline[linewidth=0.02cm](1.6,-1.0)(0.8,-0.2)
\psline[linewidth=0.02cm](1.6,-1.0)(2.4,-0.2)
\psdots[dotsize=0.16](0.8,1.8)
\psdots[dotsize=0.16](1.6,1.0)
\psdots[dotsize=0.16](0.8,-1.8)
\psdots[dotsize=0.16](1.6,-1.0)
\end{pspicture} 
}

\end{array}
&
\begin{array}{c}\\
\scalebox{.5} % Change this value to rescale the drawing.
{
\begin{pspicture}(0,-3.61)(10.41,3.61)
\psline[linewidth=0.02cm](0.0,3.6)(0.8,2.8)
\psline[linewidth=0.02cm](0.8,2.8)(1.6,3.6)
\psline[linewidth=0.02cm](0.8,2.8)(1.6,2.0)
\psline[linewidth=0.02cm](1.6,2.0)(2.4,2.8)
\psline[linewidth=0.02cm](1.6,2.0)(2.4,1.2)
\psline[linewidth=0.02cm](2.4,1.2)(3.2,2.0)
\psline[linewidth=0.02cm](2.4,1.2)(2.4,0.4)
\psline[linewidth=0.02cm](4.0,2.8)(4.8,2.0)
\psline[linewidth=0.02cm](4.8,2.0)(5.6,2.8)
\psline[linewidth=0.02cm](4.8,2.0)(5.6,1.2)
\psline[linewidth=0.02cm](5.6,1.2)(5.6,0.4)
\psline[linewidth=0.02cm](5.6,2.8)(4.8,3.6)
\psline[linewidth=0.02cm](5.6,2.8)(6.4,3.6)
\psline[linewidth=0.02cm](5.6,1.2)(6.4,2.0)
\psline[linewidth=0.02cm](7.2,2.8)(8.0,2.0)
\psline[linewidth=0.02cm](8.0,2.0)(8.6,2.8)
\psline[linewidth=0.02cm](8.0,2.0)(8.8,1.2)
\psline[linewidth=0.02cm](8.8,1.2)(8.8,0.4)
\psline[linewidth=0.02cm](8.8,1.2)(9.6,2.0)
\psline[linewidth=0.02cm](9.6,2.0)(9.0,2.8)
\psline[linewidth=0.02cm](9.6,2.0)(10.4,2.8)
\psline[linewidth=0.02cm](2.4,-2.0)(3.2,-2.8)
\psline[linewidth=0.02cm](3.2,-2.8)(3.2,-3.6)
\psline[linewidth=0.02cm](3.2,-2.8)(4.0,-2.0)
\psline[linewidth=0.02cm](4.0,-2.0)(3.2,-1.2)
\psline[linewidth=0.02cm](4.0,-2.0)(4.8,-1.2)
\psline[linewidth=0.02cm](3.2,-1.2)(4.0,-0.4)
\psline[linewidth=0.02cm](3.2,-1.2)(2.4,-0.4)
\psline[linewidth=0.02cm](5.6,-2.0)(6.4,-2.8)
\psline[linewidth=0.02cm](6.4,-2.8)(6.4,-3.6)
\psline[linewidth=0.02cm](6.4,-2.8)(7.2,-2.0)
\psline[linewidth=0.02cm](7.2,-2.0)(6.4,-1.2)
\psline[linewidth=0.02cm](7.2,-2.0)(8.0,-1.2)
\psline[linewidth=0.02cm](8.0,-1.2)(7.2,-0.4)
\psline[linewidth=0.02cm](8.0,-1.2)(8.8,-0.4)
\psdots[dotsize=0.16](2.4,1.2)
\psdots[dotsize=0.16](1.6,2.0)
\psdots[dotsize=0.16](0.8,2.8)
\psdots[dotsize=0.16](4.8,2.0)
\psdots[dotsize=0.16](5.6,2.8)
\psdots[dotsize=0.16](5.6,1.2)
\psdots[dotsize=0.16](8.0,2.0)
\psdots[dotsize=0.16](8.8,1.2)
\psdots[dotsize=0.16](9.6,2.0)
\psdots[dotsize=0.16](3.2,-1.2)
\psdots[dotsize=0.16](4.0,-2.0)
\psdots[dotsize=0.16](3.2,-2.8)
\psdots[dotsize=0.16](6.4,-2.8)
\psdots[dotsize=0.16](7.2,-2.0)
\psdots[dotsize=0.16](8.0,-1.2)
\end{pspicture} 
}

\end{array}\\ \hline
\end{array}
\end{array}$$
\medskip
\caption{Binary trees with up to four leaves and no corks.}
\label{234}
\end{figure}
%%}%%%%%%%%%%%%%%%%%%%%%%%%%%%%%%%%%%%%%%%%%%%%%%%%%%%%%%%%%%%%%%%%%%%%

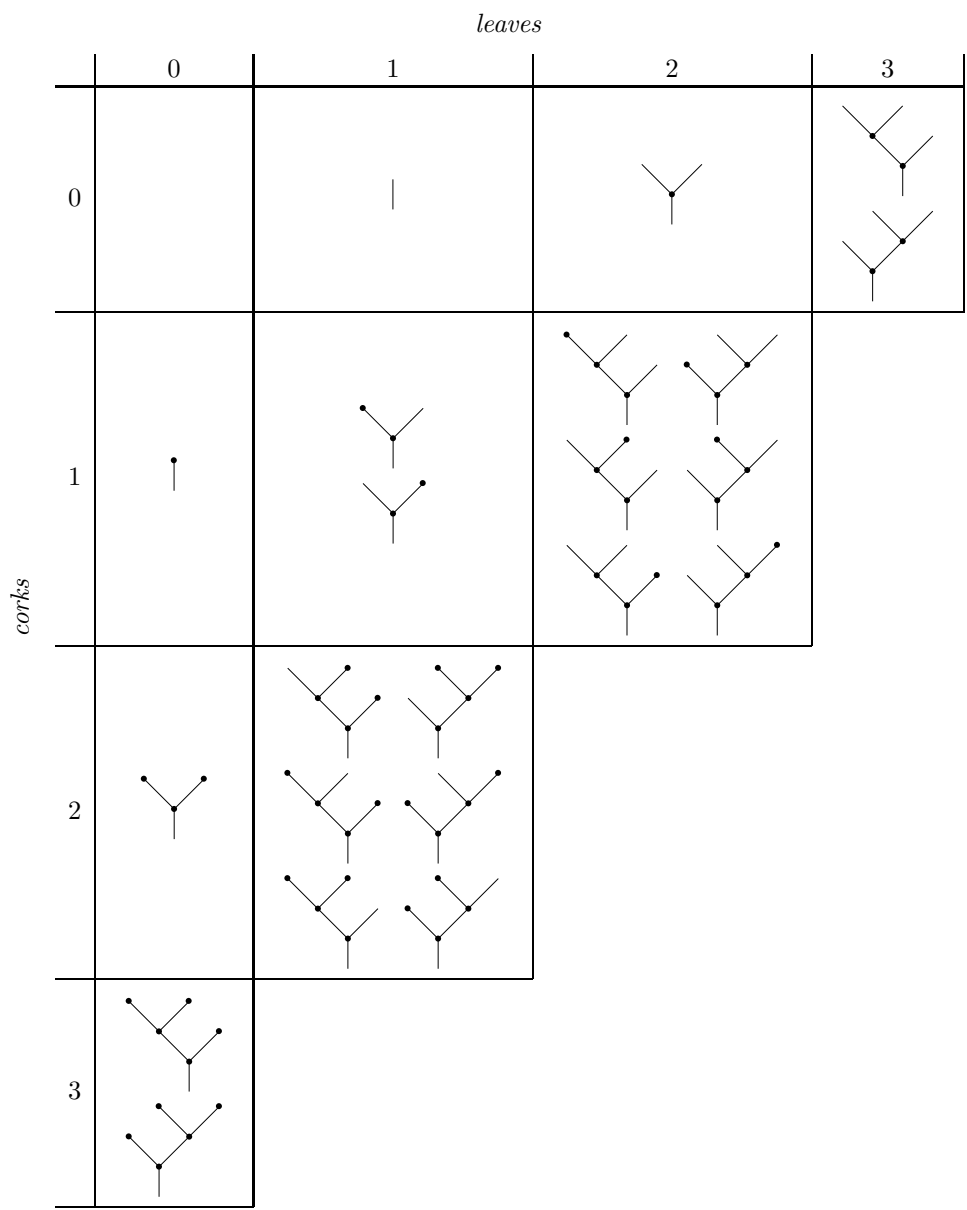
\begin{figure}[H]

$$
\begin{array}{cc}
&\text{\emph{leaves}}\\[5pt]
\begin{sideways}\text{\emph{corks}}\end{sideways}&
\begin{array}{c|c|c|c|c|}
&0&1&2&3\\
\hline 0&&
\begin{array}{c}
\scalebox{.5} % Change this value to rescale the drawing.
{
\begin{pspicture}(0,-0.41)(0.01,0.41)
\psline[linewidth=0.02cm](0.0,-0.4)(0.0,0.4)
\end{pspicture} 
}
\end{array}

&
\begin{array}{c}
\scalebox{.5} % Change this value to rescale the drawing.
{
\begin{pspicture}(0,-0.81)(1.61,0.81)
\psline[linewidth=0.02cm](0.8,-0.8)(0.8,0.0)
\psline[linewidth=0.02cm](0.8,0.0)(1.6,0.8)
\psline[linewidth=0.02cm](0.8,0.0)(0.0,0.8)
\psdots[dotsize=0.16](0.8,0.0)
\end{pspicture} 
}

\end{array}

&
\begin{array}{c}\\[-5pt]
\scalebox{.5} % Change this value to rescale the drawing.
{
\begin{pspicture}(0,-2.61)(2.41,2.61)
\psline[linewidth=0.02cm](1.6,0.2)(1.6,1.0)
\psline[linewidth=0.02cm](1.6,1.0)(2.4,1.8)
\psline[linewidth=0.02cm](1.6,1.0)(0.8,1.8)
\psline[linewidth=0.02cm](0.8,1.8)(1.6,2.6)
\psline[linewidth=0.02cm](0.8,1.8)(0.0,2.6)
\psline[linewidth=0.02cm](0.0,-1.0)(0.8,-1.8)
\psline[linewidth=0.02cm](0.8,-1.8)(0.8,-2.6)
\psline[linewidth=0.02cm](0.8,-1.8)(1.6,-1.0)
\psline[linewidth=0.02cm](1.6,-1.0)(0.8,-0.2)
\psline[linewidth=0.02cm](1.6,-1.0)(2.4,-0.2)
\psdots[dotsize=0.16](0.8,1.8)
\psdots[dotsize=0.16](1.6,1.0)
\psdots[dotsize=0.16](0.8,-1.8)
\psdots[dotsize=0.16](1.6,-1.0)
\end{pspicture} 
}

\end{array}
\\
\hline 1&
\begin{array}{c}
\scalebox{.5} % Change this value to rescale the drawing.
{
\begin{pspicture}(0,-0.445)(0.14,0.455)
\psline[linewidth=0.02cm](0.07,-0.435)(0.07,0.365)
\psdots[dotsize=0.16](0.06,0.375)
\end{pspicture} 
}
\end{array}
&
\begin{array}{c}
\scalebox{.5} % Change this value to rescale the drawing.
{
\begin{pspicture}(0,-1.855)(1.78,1.865)
\psline[linewidth=0.02cm](0.89,0.155)(0.89,0.955)
\psline[linewidth=0.02cm](0.89,0.955)(1.69,1.755)
\psline[linewidth=0.02cm](0.89,0.955)(0.09,1.755)
\psdots[dotsize=0.16](0.89,0.955)
\psline[linewidth=0.02cm](0.89,-1.845)(0.89,-1.045)
\psline[linewidth=0.02cm](0.89,-1.045)(1.69,-0.245)
\psline[linewidth=0.02cm](0.89,-1.045)(0.09,-0.245)
\psdots[dotsize=0.16](0.89,-1.045)
\psdots[dotsize=0.16](0.08,1.765)
\psdots[dotsize=0.16](1.68,-0.235)
\end{pspicture} 
}
\end{array}
&
\begin{array}{c}\\[-5pt]
\scalebox{.5} % Change this value to rescale the drawing.
{
\begin{pspicture}(0,-4.055)(5.78,4.065)
\psline[linewidth=0.02cm](1.69,1.555)(1.69,2.355)
\psline[linewidth=0.02cm](1.69,2.355)(2.49,3.155)
\psline[linewidth=0.02cm](1.69,2.355)(0.89,3.155)
\psline[linewidth=0.02cm](0.89,3.155)(1.69,3.955)
\psline[linewidth=0.02cm](0.89,3.155)(0.09,3.955)
\psline[linewidth=0.02cm](3.29,3.155)(4.09,2.355)
\psline[linewidth=0.02cm](4.09,2.355)(4.09,1.555)
\psline[linewidth=0.02cm](4.09,2.355)(4.89,3.155)
\psline[linewidth=0.02cm](4.89,3.155)(4.09,3.955)
\psline[linewidth=0.02cm](4.89,3.155)(5.69,3.955)
\psdots[dotsize=0.16](0.89,3.155)
\psdots[dotsize=0.16](1.69,2.355)
\psdots[dotsize=0.16](4.09,2.355)
\psdots[dotsize=0.16](4.89,3.155)
\psdots[dotsize=0.16](0.08,3.965)
\psdots[dotsize=0.16](3.28,3.165)
\psline[linewidth=0.02cm](1.69,-1.245)(1.69,-0.445)
\psline[linewidth=0.02cm](1.69,-0.445)(2.49,0.355)
\psline[linewidth=0.02cm](1.69,-0.445)(0.89,0.355)
\psline[linewidth=0.02cm](0.89,0.355)(1.69,1.155)
\psline[linewidth=0.02cm](0.89,0.355)(0.09,1.155)
\psline[linewidth=0.02cm](3.29,0.355)(4.09,-0.445)
\psline[linewidth=0.02cm](4.09,-0.445)(4.09,-1.245)
\psline[linewidth=0.02cm](4.09,-0.445)(4.89,0.355)
\psline[linewidth=0.02cm](4.89,0.355)(4.09,1.155)
\psline[linewidth=0.02cm](4.89,0.355)(5.69,1.155)
\psdots[dotsize=0.16](0.89,0.355)
\psdots[dotsize=0.16](1.69,-0.445)
\psdots[dotsize=0.16](4.09,-0.445)
\psdots[dotsize=0.16](4.89,0.355)
\psdots[dotsize=0.16](1.68,1.165)
\psdots[dotsize=0.16](4.08,1.165)
\psline[linewidth=0.02cm](1.69,-4.045)(1.69,-3.245)
\psline[linewidth=0.02cm](1.69,-3.245)(2.49,-2.445)
\psline[linewidth=0.02cm](1.69,-3.245)(0.89,-2.445)
\psline[linewidth=0.02cm](0.89,-2.445)(1.69,-1.645)
\psline[linewidth=0.02cm](0.89,-2.445)(0.09,-1.645)
\psline[linewidth=0.02cm](3.29,-2.445)(4.09,-3.245)
\psline[linewidth=0.02cm](4.09,-3.245)(4.09,-4.045)
\psline[linewidth=0.02cm](4.09,-3.245)(4.89,-2.445)
\psline[linewidth=0.02cm](4.89,-2.445)(4.09,-1.645)
\psline[linewidth=0.02cm](4.89,-2.445)(5.69,-1.645)
\psdots[dotsize=0.16](0.89,-2.445)
\psdots[dotsize=0.16](1.69,-3.245)
\psdots[dotsize=0.16](4.09,-3.245)
\psdots[dotsize=0.16](4.89,-2.445)
\psdots[dotsize=0.16](2.48,-2.435)
\psdots[dotsize=0.16](5.68,-1.635)
\end{pspicture} 
}
\end{array}\\
\cline{1-4} 2&
\begin{array}{c}
\scalebox{.5} % Change this value to rescale the drawing.
{
\begin{pspicture}(0,-0.855)(1.78,0.865)
\psline[linewidth=0.02cm](0.89,-0.845)(0.89,-0.045)
\psline[linewidth=0.02cm](0.89,-0.045)(1.69,0.755)
\psline[linewidth=0.02cm](0.89,-0.045)(0.09,0.755)
\psdots[dotsize=0.16](0.89,-0.045)
\psdots[dotsize=0.16](0.08,0.765)
\psdots[dotsize=0.16](1.68,0.765)
\end{pspicture} 
}
\end{array}

&

\begin{array}{c}\\[-5pt]
\scalebox{.5} % Change this value to rescale the drawing.
{
\begin{pspicture}(0,-4.055)(5.78,4.065)
\psline[linewidth=0.02cm](1.69,1.555)(1.69,2.355)
\psline[linewidth=0.02cm](1.69,2.355)(2.49,3.155)
\psline[linewidth=0.02cm](1.69,2.355)(0.89,3.155)
\psline[linewidth=0.02cm](0.89,3.155)(1.69,3.955)
\psline[linewidth=0.02cm](0.89,3.155)(0.09,3.955)
\psline[linewidth=0.02cm](3.29,3.155)(4.09,2.355)
\psline[linewidth=0.02cm](4.09,2.355)(4.09,1.555)
\psline[linewidth=0.02cm](4.09,2.355)(4.89,3.155)
\psline[linewidth=0.02cm](4.89,3.155)(4.09,3.955)
\psline[linewidth=0.02cm](4.89,3.155)(5.69,3.955)
\psdots[dotsize=0.16](0.89,3.155)
\psdots[dotsize=0.16](1.69,2.355)
\psdots[dotsize=0.16](4.09,2.355)
\psdots[dotsize=0.16](4.89,3.155)
\psdots[dotsize=0.16](1.68,3.965)
\psdots[dotsize=0.16](2.48,3.165)
\psdots[dotsize=0.16](5.68,3.965)
\psdots[dotsize=0.16](4.08,3.965)
\psline[linewidth=0.02cm](1.69,-1.245)(1.69,-0.445)
\psline[linewidth=0.02cm](1.69,-0.445)(2.49,0.355)
\psline[linewidth=0.02cm](1.69,-0.445)(0.89,0.355)
\psline[linewidth=0.02cm](0.89,0.355)(1.69,1.155)
\psline[linewidth=0.02cm](0.89,0.355)(0.09,1.155)
\psline[linewidth=0.02cm](3.29,0.355)(4.09,-0.445)
\psline[linewidth=0.02cm](4.09,-0.445)(4.09,-1.245)
\psline[linewidth=0.02cm](4.09,-0.445)(4.89,0.355)
\psline[linewidth=0.02cm](4.89,0.355)(4.09,1.155)
\psline[linewidth=0.02cm](4.89,0.355)(5.69,1.155)
\psdots[dotsize=0.16](0.89,0.355)
\psdots[dotsize=0.16](1.69,-0.445)
\psdots[dotsize=0.16](4.09,-0.445)
\psdots[dotsize=0.16](4.89,0.355)
\psdots[dotsize=0.16](0.08,1.165)
\psdots[dotsize=0.16](2.48,0.365)
\psdots[dotsize=0.16](5.68,1.165)
\psdots[dotsize=0.16](3.28,0.365)
\psline[linewidth=0.02cm](1.69,-4.045)(1.69,-3.245)
\psline[linewidth=0.02cm](1.69,-3.245)(2.49,-2.445)
\psline[linewidth=0.02cm](1.69,-3.245)(0.89,-2.445)
\psline[linewidth=0.02cm](0.89,-2.445)(1.69,-1.645)
\psline[linewidth=0.02cm](0.89,-2.445)(0.09,-1.645)
\psline[linewidth=0.02cm](3.29,-2.445)(4.09,-3.245)
\psline[linewidth=0.02cm](4.09,-3.245)(4.09,-4.045)
\psline[linewidth=0.02cm](4.09,-3.245)(4.89,-2.445)
\psline[linewidth=0.02cm](4.89,-2.445)(4.09,-1.645)
\psline[linewidth=0.02cm](4.89,-2.445)(5.69,-1.645)
\psdots[dotsize=0.16](0.89,-2.445)
\psdots[dotsize=0.16](1.69,-3.245)
\psdots[dotsize=0.16](4.09,-3.245)
\psdots[dotsize=0.16](4.89,-2.445)
\psdots[dotsize=0.16](0.08,-1.635)
\psdots[dotsize=0.16](1.68,-1.635)
\psdots[dotsize=0.16](4.08,-1.635)
\psdots[dotsize=0.16](3.28,-2.435)
\end{pspicture} 
}

\end{array}
\\
\cline{1-3} 3&
\begin{array}{c}\\[-5pt]
\scalebox{.5} % Change this value to rescale the drawing.
{
\begin{pspicture}(0,-2.655)(2.58,2.665)
\psline[linewidth=0.02cm](1.69,0.155)(1.69,0.955)
\psline[linewidth=0.02cm](1.69,0.955)(2.49,1.755)
\psline[linewidth=0.02cm](1.69,0.955)(0.89,1.755)
\psline[linewidth=0.02cm](0.89,1.755)(1.69,2.555)
\psline[linewidth=0.02cm](0.89,1.755)(0.09,2.555)
\psline[linewidth=0.02cm](0.09,-1.045)(0.89,-1.845)
\psline[linewidth=0.02cm](0.89,-1.845)(0.89,-2.645)
\psline[linewidth=0.02cm](0.89,-1.845)(1.69,-1.045)
\psline[linewidth=0.02cm](1.69,-1.045)(0.89,-0.245)
\psline[linewidth=0.02cm](1.69,-1.045)(2.49,-0.245)
\psdots[dotsize=0.16](0.89,1.755)
\psdots[dotsize=0.16](1.69,0.955)
\psdots[dotsize=0.16](0.89,-1.845)
\psdots[dotsize=0.16](1.69,-1.045)
\psdots[dotsize=0.16](0.08,2.565)
\psdots[dotsize=0.16](1.68,2.565)
\psdots[dotsize=0.16](2.48,1.765)
\psdots[dotsize=0.16](2.48,-0.235)
\psdots[dotsize=0.16](0.88,-0.235)
\psdots[dotsize=0.16](0.08,-1.035)
\end{pspicture} 
}

\end{array}\\
\cline{1-2} 

\end{array}
\end{array}$$

\bigskip

\caption{Binary trees such that the number of leaves plus the number of corks does not exceed three.}
\label{23}
\end{figure}

%$e'=\{u,v\},e=\{v,w_1\},e''=\{v,w_2\}\in E(T)$ and $w_1\in L(T)$ is a leaf, we denote $T\backslash e$ the tree with $V(T\backslash e)=V(T)\setminus\{v,w_1\}$ and
%$E(T\backslash e)=(E(T)\setminus\{e,e'\})\cup\{\{u,w_2\}\}$. i.e. . . %If both $e$ and $e'$ are inned edges we denote $\rho_{T,e}\in S_{n-3}$ the permutation sending the restriction of the path order in $T$ to $I(T)\setminus\{e',e''\}=I(T\backslash e)\setminus\{e'\cup e''\}$ to the restriction of the path order in $T\backslash e$. %If only $e'$ is internal

\section{Associahedra}

We present here the classical associahedra, following first the Boardman--Vogt \emph{cubical} construction \cite{hiasts}, and then describing the polytope cellular structure.

Recall that the positive and negative \emph{faces} of the hypercube $[0,1]^n$ are, $1\leq i\leq n$,
\begin{align*}
\partial_i^-\colon [0,1]^{n-1}&\To[0,1]^n,\\
(x_1,\dots,x_{n-1})&\;\mapsto\;(x_1,\dots,x_{i-1},0,x_{i},\dots,x_{n-1}),\\
\partial_i^+\colon [0,1]^{n-1}&\To[0,1]^n,\\
(x_1,\dots,x_{n-1})&\;\mapsto\;(x_1,\dots,x_{i-1},1,x_{i},\dots,x_{n-1}).
\end{align*}
The \emph{degeneracies} are, $1\leq i\leq n$,
\begin{align*}
\pi_i\colon [0,1]^{n}&\To[0,1]^{n-1},\\
(x_1,\dots,x_{n})&\;\mapsto\;(x_1,\dots,x_{i-1},x_{i+1},\dots,x_{n}).
\end{align*}
The \emph{connections} are, $1\leq i< n$,
\begin{align*}
\gamma_i\colon [0,1]^{n}&\To[0,1]^{n-1},\\
(x_1,\dots,x_{n})&\;\mapsto\;(x_1,\dots,x_{i-1},\max(x_i,x_{i+1}),x_{i+2},\dots,x_{n}).
\end{align*}
Given $s,t\geq 0$ and $1\leq i\leq s+1$ we denote:
\begin{align}
\label{sigma} \sigma_{
%s,t,
i}\colon [0,1]^s\times[0,1]^t&\st{\cong}\To [0,1]^{s+t},\\\nonumber
(x_1,\dots,x_s,y_1,\dots, y_t)&\;\mapsto\;(x_1,\dots,x_{i-1},y_1,\dots, y_t,x_i,\dots,x_s).
\end{align}

Now for any binary tree $T$ with $n$ leaves and no corks, consider the $(n-2)$-cube $$H_T=[0,1]^{n-2}.$$ 
We regard $(x_1,\dots, x_{n-2})\in H_T$ as a labeling of $\norm{T}$ with the $i^{\text{th}}$ inner edge labeled by $x_i$:
$$\norm{T}\quad=\quad\begin{array}{c}
\xy/r1.3pt/:
(0,0)*{}="a",
%(0,-3)*{a},
(0,10)*-<3pt>{\bullet}="b",
(-10,20)*{}="c",
%(-10,23)*{c},
(10,20)*-<3pt>{\bullet}="d",
(0,30)*-<3pt>{\bullet}="e",
(20,30)*{}="g",
%(20,33)*{g},
(-8,40)*{}="h",
%(-10,43)*{h},
%(0,43)*{i},
(8,40)*{}="j",
%(10,43)*{j},
\ar@{-}"a";"b",
\ar@{-}"b";"c",
\ar@{-}"b";"d",
\ar@{-}"d";"e",
\ar@{-}"d";"g",
\ar@{-}"e";"h",
\ar@{-}"e";"j",
\endxy
\end{array},
\qquad\qquad\qquad
\begin{array}{c}
\xy/r1.3pt/:
(0,0)*{}="a",
%(0,-3)*{a},
(0,10)*-<3pt>{\bullet}="b",
(-10,20)*{}="c",
%(-10,23)*{c},
(10,20)*-<3pt>{\bullet}="d",
(8,12)*{\scriptstyle x_1},
(0,30)*-<3pt>{\bullet}="e",
(2,23)*{\scriptstyle x_2},
(20,30)*{}="g",
%(20,33)*{g},
(-8,40)*{}="h",
%(-10,43)*{h},
%(0,43)*{i},
(8,40)*{}="j",
%(10,43)*{j},
\ar@{-}"a";"b",
\ar@{-}"b";"c",
\ar@{-}"b";"d",
\ar@{-}"d";"e",
\ar@{-}"d";"g",
\ar@{-}"e";"h",
\ar@{-}"e";"j",
\endxy
\end{array}= \quad(x_1,x_2)\in H_T.
$$
The \emph{associahedron} $K_n$ is then defined as the cell complex 
$$
\left.
\coprod_{\begin{array}{c}\\[-5mm]\text{\scriptsize binary trees $T$}\\[-1.5mm]\text{\scriptsize with $n$ leaves}
\\[-1.5mm]\text{\scriptsize and no corks}\end{array}}\!\!\!\!\!\!
H_T
\;
\right/
{\vphantom y}_{{\vphantom y}_{\displaystyle\sim\quad.}}
$$
Here $\sim$ is defined as follows. Given two  binary trees $T$ and $T'$ with $n$ leaves and no corks, if  $T/e_i=T'/e_j'$, where $e_{i}$ (resp.\ $e_{i}'$) denotes the $i^{\text{th}}$ inner edge of $T$ (resp.~$T'$), then 
identify the $i^{\text{th}}$ negative face of $H_T$ with the $j^{\text{th}}$ negative face of $H_{T'}$,
%according to the following maps
%$$\xymatrix@C=40pt{H_{T/e_i}=H_{T'/e'_j}\ar[r]^-{\partial^-_j\sigma_{T',e_j'}}\ar@<-5.5ex>[d]_{\partial^-_i\sigma_{T,e_i}}&H_{T'}\\\hspace{-45pt}H_T}$$

\begin{equation}\label{asociativo}
\begin{array}{c}
\xy/r1.3pt/:
(0,0)*{}="a",
(0,10)*-<3pt>{\bullet}="b",
(10,20)*{}="c",
(10,30)*{}="cc",
(-10,20)*-<3pt>{\bullet}="d",
(-7,13)*{\scriptstyle 0},
(0,30)*{}="e",
(0,40)*{}="ee",
(-20,30)*{}="g",
(-20,40)*{}="gg",
\ar@{.}"a";"b",
\ar@{.}"c";"cc",
\ar@{.}"e";"ee",
\ar@{.}"g";"gg",
\ar@{-}"b";"c",
\ar@{-}"b";"d",
\ar@{-}"d";"e",
\ar@{-}"d";"g",
\endxy
\end{array}
\qquad\sim\qquad\begin{array}{c}
\xy/r1.3pt/:
(0,0)*{}="a",
(0,10)*-<3pt>{\bullet}="b",
(-10,20)*{}="c",
(-10,30)*{}="cc",
(10,20)*-<3pt>{\bullet}="d",
(7,13)*{\scriptstyle 0},
(0,30)*{}="e",
(0,40)*{}="ee",
(20,30)*{}="g",
(20,40)*{}="gg",
\ar@{.}"a";"b",
\ar@{.}"c";"cc",
\ar@{.}"e";"ee",
\ar@{.}"g";"gg",
\ar@{-}"b";"c",
\ar@{-}"b";"d",
\ar@{-}"d";"e",
\ar@{-}"d";"g",
\endxy
\end{array}.
\end{equation}
Heuristically, the label indicates the length of the edge, so a length zero edge should be collapsed to a point, hence the identification. %We adopt the conventions that $K_1=*$ and $K_0=\varnothing$.

\bigskip

\begin{figure}[H]
$$K_0=\varnothing\;,\qquad
\begin{array}{cc}\\[-24pt]&|\\ K_{1}=&{\bullet}\end{array}, \qquad
\begin{array}{cc}\\[-29pt]&\includegraphics[scale=.1]{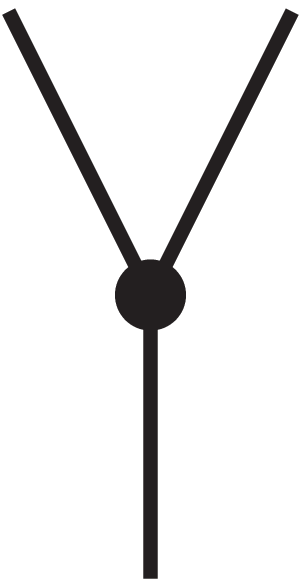}\\[-3pt] K_{2}=&{\bullet}\end{array}, $$

$$
K_{3}=\begin{array}{c}\\[-35pt]
\begin{array}{ccc}
\includegraphics[scale=.1]{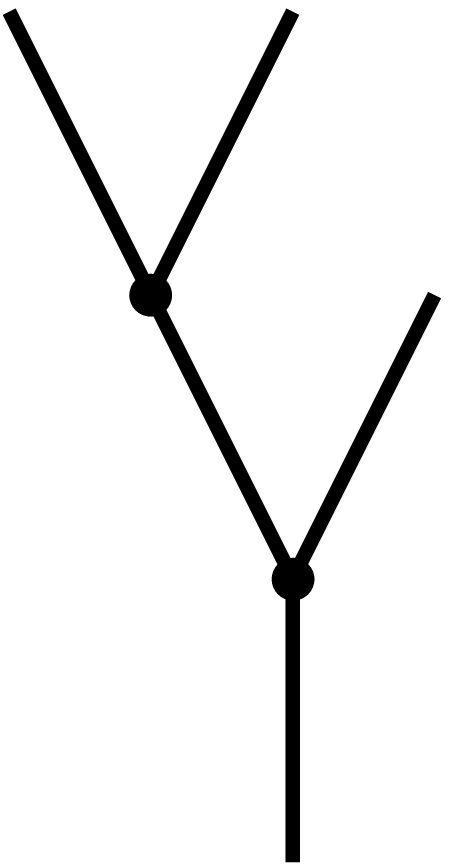}\hspace{0pt}&
\includegraphics[scale=.08]{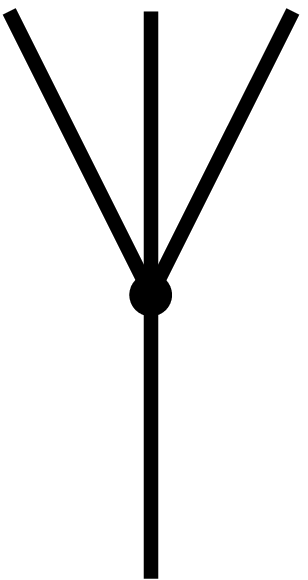}\hspace{0pt}&
\includegraphics[scale=.1]{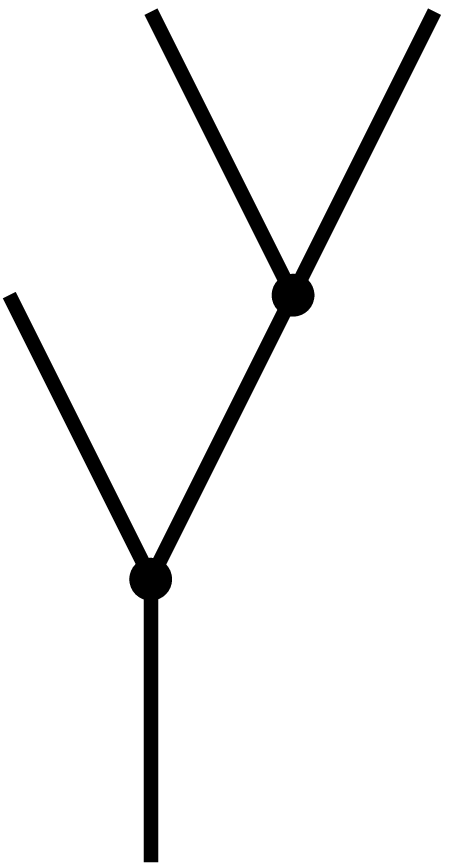}
\end{array}\\[-5pt]
\includegraphics[scale=.2]{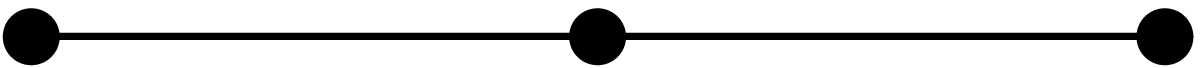}
\end{array}\;,\qquad
\qquad K_{4}=\begin{array}{c}\includegraphics[scale=.5]{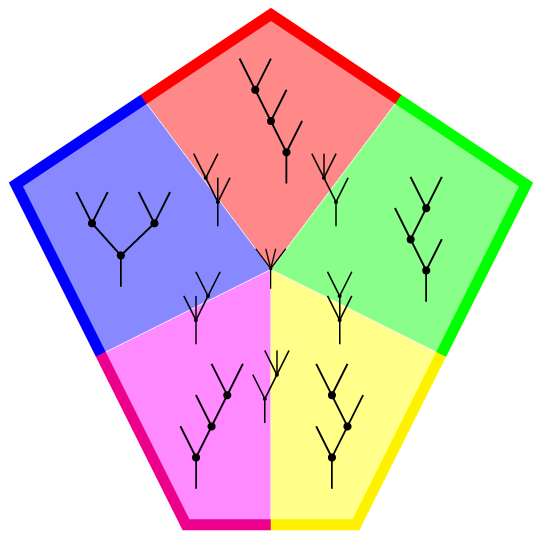} % para B&W poner aqu' k5grey.eps
\end{array}.$$
\caption{The first five associahedra with the cellular structure given by binary trees with no corks. On common faces we depict the non-binary tree which yields the identification.}
\end{figure}

The  \emph{topological $A$-infinity operad} $\mathtt{A}_\infty$ is given by the associahedra $\mathtt{A}_\infty(n)=K_n$, $n\geq 0$. One defines the composition law
$$\circ_i\colon K_p\times K_q\To K_{p+q-1},\quad 1\leq i\leq p,$$
on cells  $H_T\times H_{T'}$ (where $T$ and $T'$ have $p$ and $q$ leaves, respectively, and no corks)  
by
$$
\xymatrix@C=22pt{H_T\times H_{T'}=[0,1]^{p-2}\times[0,1]^{q-2}\ar[r]^-{\sigma_{
%p-2,q-2,
j}}& [0,1]^{p+q-4}\ar[r]^-{\partial^+_{j}}& [0,1]^{p+q-3}=H_{T\circ_iT'}}
$$
if %$e_{j}\in I(T\circ_i T')$ is 
the identification of the root edge of $T'$ with the $i^{\text{th}}$ leaf edge of $T$
 is the $j^{\text{th}}$ inner edge of $T\circ_{i}T'$.

We can illustrate the composition graphically:
$$
\begin{array}{c}
\xy/r1.3pt/:
(0,0)*{}="a",
(0,10)*-<3pt>{\bullet}="b",
(-10,20)*-<3pt>{\bullet}="c",
(10,20)*-<3pt>{\bullet}="d",
(-8,12)*{\scriptstyle x_1},
(8,12)*{\scriptstyle x_2},
(4,30)*{}="e",
(16,30)*{}="g",
(-16,30)*{}="ee",
(-4,30)*{}="gg",
\ar@{-}"a";"b",
\ar@{-}"b";"c",
\ar@{-}"b";"d",
\ar@{-}"d";"e",
\ar@{-}"d";"g",
\ar@{-}"c";"ee",
\ar@{-}"c";"gg",
\endxy
\end{array}
\circ_2
\begin{array}{c}
\xy/r1.3pt/:
(0,0)*{}="a",
(0,10)*-<3pt>{\bullet}="b",
(-6,20)*{}="c",
(6,20)*-<3pt>{\bullet}="d",
(8,14)*{\scriptstyle y_1},
(0,30)*-<3pt>{\bullet}="e",
(0,23)*{\scriptstyle y_2},
(12,30)*{}="g",
(-6,40)*{}="h",
(6,40)*{}="j",
\ar@{-}"a";"b",
\ar@{-}"b";"c",
\ar@{-}"b";"d",
\ar@{-}"d";"e",
\ar@{-}"d";"g",
\ar@{-}"e";"h",
\ar@{-}"e";"j",
\endxy
\end{array}
\quad=\quad
\begin{array}{c}
\xy/r1.3pt/:
(0,0)*{}="a",
(0,10)*-<3pt>{\bullet}="b",
(-10,20)*-<3pt>{\bullet}="c",
(10,20)*-<3pt>{\bullet}="d",
(-8,12)*{\scriptstyle x_1},
(8,12)*{\scriptstyle x_2},
(4,30)*{}="e",
(16,30)*{}="g",
(-16,30)*{}="ee",
(-4,30)*{}="gg",
(-4,30)*-<3pt>{\bullet}="bbb",
(-10,40)*{}="ccc",
(2,40)*-<3pt>{\bullet}="ddd",
(4,34)*{\scriptstyle y_1},
(-4,24)*{\scriptstyle 1},
(-4,50)*-<3pt>{\bullet}="eee",
(-4,43)*{\scriptstyle y_2},
(8,50)*{}="ggg",
(-10,60)*{}="hhh",
(2,60)*{}="jjj",
\ar@{-}"bbb";"ccc",
\ar@{-}"bbb";"ddd",
\ar@{-}"ddd";"eee",
\ar@{-}"ddd";"ggg",
\ar@{-}"eee";"hhh",
\ar@{-}"eee";"jjj",
\ar@{-}"a";"b",
\ar@{-}"b";"c",
\ar@{-}"b";"d",
\ar@{-}"d";"e",
\ar@{-}"d";"g",
\ar@{-}"c";"ee",
\ar@{-}"c";"gg",
\endxy
\end{array}\quad=\quad(x_1,1,y_1,y_2,x_2).
$$
Thus composition is grafting of trees, taking labels into account, and labeling by 1 the newly created inner edge.

If one forgets the cubical subdivision, the associahedron $K_n$ can be represented as an $(n-2)$-dimensional convex polytope% with $(n-3)$-faces $K_p\circ_i K_q$, $p+q-1=n$, $1\leq i\leq p$
.
The poset of faces of $K_n$ under inclusion may be identified with the category of planted planar trees with $n$ leaves, no corks and no degree 2 vertices. Inclusion of faces corresponds to collapse maps $T\to T/e$. We may write $K_T$ for the codimension $r$ face of $K_n$ corresponding to a tree $T$ with $n$ leaves and $r$ inner edges. 
The unique $(n-2)$-cell of $K_n$ is the corolla $C_n$, the codimension 1 faces $K_p\circ_i K_q$ are the trees with a unique internal edge $e=\{v_1,v_{i+1}\}$, and more generally
$$
%\partial(K_T)=\bigcup_{T_2\circ_i T_1\to T}
K_{T}\circ_i K_{T'}=
%\bigcup_{T_2\circ_i T_1\to T} 
K_{T\circ_i T'}. $$
%$$
%K_2\circ_2K_3\qquad=\quad\begin{array}{c}\xy/r2pt/:
%(0,0)*{}="a",
%%(0,-3)*{a},
%(0,10)*-<3pt>{\bullet}="b",
%(0,14)*{v_1},
%(-10,20)*{}="c",
%%(-10,23)*{c},
%(10,20)*-<3pt>{\bullet}="d",
%(14,18)*{v_3},
%(0,30)*{}="e",
%%(-2,28)*{v_4},
%(10,30)*{}="f",
%%(10,34)*{v_8},
%(20,30)*{}="g",
%%(20,33)*{g},
%%(-8,40)*{}="h",
%%(-10,43)*{h},
%%(0,40)*{}="i",
%%(0,43)*{i},
%%(8,40)*{}="j",
%%(10,43)*{j},
%\ar@{-}"a";"b",
%\ar@{-}"b";"c",
%\ar@{-}"b";"d",
%\ar@{-}"d";"e",
%\ar@{-}"d";"f",
%\ar@{-}"d";"g",
%%\ar@{-}"e";"h",
%%\ar@{-}"e";"i",
%%\ar@{-}"e";"j",
%\endxy\end{array}\quad\subset \qquad K_4
%$$

\bigskip

\begin{figure}[H]
$$K_0=\varnothing\;,\qquad
\begin{array}{cc}\\[-24pt]&|\\ K_{1}=&{\bullet}\end{array}, \qquad
\begin{array}{cc}\\[-29pt]&\includegraphics[scale=.1]{2sin.eps}\\[-3pt] K_{2}=&{\bullet}\end{array}, $$

$$
K_{3}=\begin{array}{c}\\[-35pt]
\begin{array}{ccc}
\includegraphics[scale=.07]{ito0sin.eps}\hspace{14pt}&
\includegraphics[scale=.13]{ito12.eps}\hspace{14pt}&
\includegraphics[scale=.07]{ito1sin.eps}
\end{array}\\[-5pt]
\includegraphics[scale=.2]{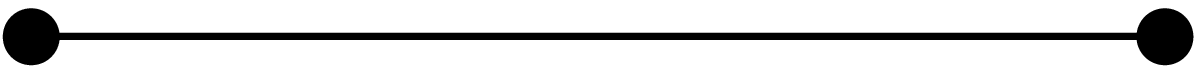}
\end{array},\qquad
\qquad K_{4}=\begin{array}{c}\includegraphics[scale=.5]{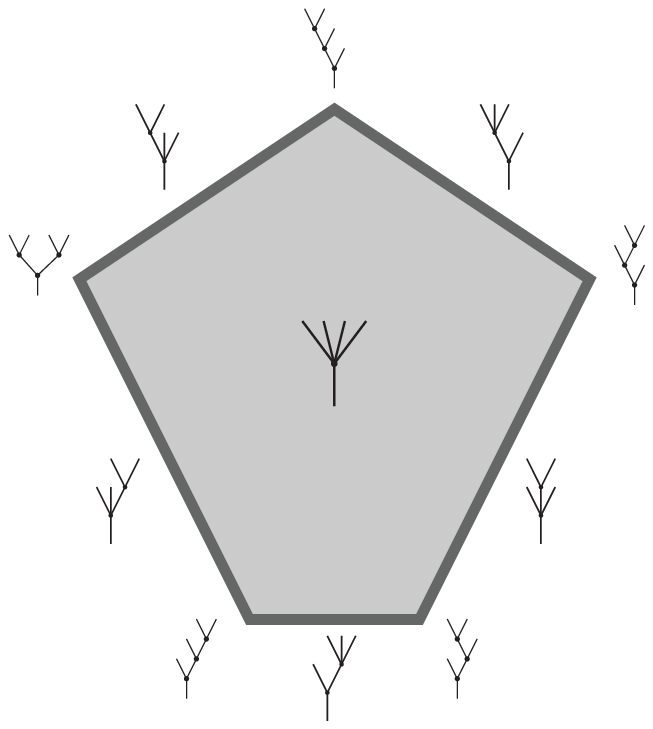}
\end{array}.$$
\caption{The first five associahedra $K_n$ as polytopes. We depict the faces $K_T$ simply by $T$. }
\end{figure}

The \emph{topological strictly unital $A$-infinity operad} $\mathtt{suA}_\infty$ is given by  $\mathtt{suA}_\infty(0)=\bullet$ and $\mathtt{suA}_\infty(n)=\mathtt{A}_\infty(n)=K_n$ if $n\geq 1$. The composition law
$$\circ_i\colon\mathtt{suA}_\infty(p)\times \mathtt{suA}_\infty(q)\To \mathtt{suA}_\infty(p+q-1)$$
is defined as in $\mathtt{A}_\infty$ except when $q=0$. If $p> 2$ and $q=0$, the map,
\begin{equation}\label{degeneracy}
\circ_i\colon K_p\To K_{p-1},\quad 1\leq i\leq p,
\end{equation}
also called associahedral \emph{degeneracy}, 
is defined on cells as follows. Suppose $T$ is a binary tree with $p$ leaves and no corks, $e$ is its $i^{\text{th}}$ leaf edge, and $e'\prec e''$ are the two edges adjacent to $e$.
%, $e''$ bigger than $e'$. One of them must be an inner edge. Assume that $e'$ is the $j^{\text{th}}$ inner edge. 
As illustrated in Figure \ref{circideg}, we distinguish three cases:
\begin{itemize}
\item If $e'$ is the $j^{\text{th}}$ inner edge and $e''$ is also inner, then $\circ_i$ is defined on the cell $H_T$ via the connection
$$H_T\st{\gamma_j}\To H_{T\backslash e}.$$
\item In the following two cases:
\begin{itemize}
\item $e'$ is the root and $e''$ is $j^{\text{th}}$ inner edge, $j=1$,
\item $e''$ is a leaf and $e'$ is $j^{\text{th}}$ inner edge, 
\end{itemize}
the map $\circ_i$ is defined on $H_T$ via the degeneracy
$$H_T\st{\pi_j}\To H_{T\backslash e}.$$
\end{itemize}

\begin{figure}[H]
$$
\scalebox{.8} % Change this value to rescale the drawing.
{
\begin{pspicture}(0,-1.61)(9.009063,1.61)
\psline[linewidth=0.02cm](1.17,-0.8)(1.17,0.0)
\psline[linewidth=0.02cm](1.17,0.0)(1.97,0.8)
\psline[linewidth=0.02cm](1.17,0.0)(0.37,0.8)
\psdots[dotsize=0.16](1.17,0.0)
\psdots[dotsize=0.16](1.97,0.8)
\psdots[dotsize=0.16](1.17,-0.8)
\psline[linewidth=0.02cm,linestyle=dotted,dotsep=0.08cm](1.97,0.8)(1.97,1.6)
\psline[linewidth=0.02cm,linestyle=dotted,dotsep=0.08cm](1.17,-0.8)(1.17,-1.6)
\usefont{T1}{ptm}{m}{n}
\rput(0.7,.2){$e$}
\usefont{T1}{ptm}{m}{n}
\rput(1.7,0.2){$e''$}
\usefont{T1}{ptm}{m}{n}
\rput(1.4,-0.3){$e'$}
\psline[linewidth=0.02cm](4.37,-0.8)(4.37,0.0)
\psline[linewidth=0.02cm](4.37,0.0)(5.17,0.8)
\psline[linewidth=0.02cm](4.37,0.0)(3.57,0.8)
\psdots[dotsize=0.16](4.37,0.0)
\psdots[dotsize=0.16](5.17,0.8)
\psline[linewidth=0.02cm,linestyle=dotted,dotsep=0.08cm](5.17,0.8)(5.17,1.6)
\usefont{T1}{ptm}{m}{n}
\rput(3.9,.2){$e$}
\usefont{T1}{ptm}{m}{n}
\rput(4.9,.2){$e''$}
\usefont{T1}{ptm}{m}{n}
\rput(4.6,-0.3){$e'$}
\psline[linewidth=0.02cm](7.57,-0.8)(7.57,0.0)
\psline[linewidth=0.02cm](7.57,0.0)(8.37,0.8)
\psline[linewidth=0.02cm](7.57,0.0)(6.77,0.8)
\psdots[dotsize=0.16](7.57,0.0)
\psdots[dotsize=0.16](7.57,-0.8)
\psline[linewidth=0.02cm,linestyle=dotted,dotsep=0.08cm](7.57,-0.8)(7.57,-1.6)
\usefont{T1}{ptm}{m}{n}
\rput(7.0,.2){$e$}
\usefont{T1}{ptm}{m}{n}
\rput(8.1,.2){$e''$}
\usefont{T1}{ptm}{m}{n}
\rput(7.8,-0.3){$e'$}
\end{pspicture} 
}$$
\caption{The three possible cases in the definition of the degeneracy \eqref{degeneracy} on cells, up to symmetries.}
\label{circideg}
\end{figure}
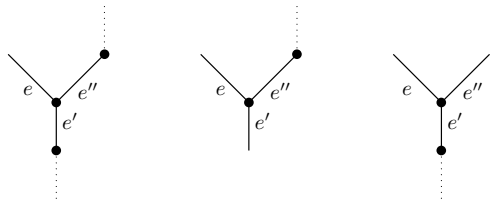

\begin{figure}[H]
$$
\begin{array}{c}
\includegraphics[scale=.3]{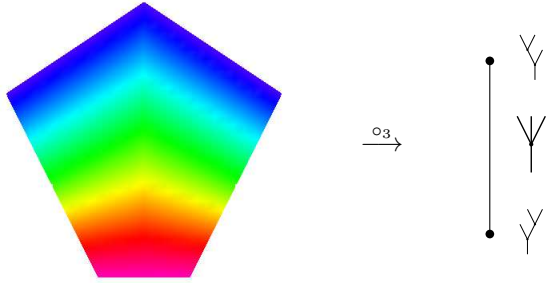}   % Para B&W poner aqu' degene4grey.eps
\end{array}\qquad\st{\circ_{3}}\To\qquad
\begin{array}{cc}\\[0pt]
\includegraphics[scale=.2,angle=90]{interval4.eps}\hspace{-10pt}&
\begin{array}{c}\\[-75pt]
\includegraphics[scale=.07]{ito0sin.eps}\\[10pt]
\includegraphics[scale=.13]{ito12.eps}\\[10pt]
\includegraphics[scale=.07]{ito1sin.eps}
\end{array}
\end{array}$$
\caption{The associahedral degeneracy $\circ_{3}\colon K_{4}\rightarrow K_{3}$. The colour gradient represents the fibers.}
\label{circideg2}
\end{figure}

\section{Associahedra with units}

We now introduce the unital associahedra, starting with a cubical structure and then with a smaller number of cells generalizing the polytope structure of the classical non-unital associahedra.

The \emph{unital associahedron} $$K^u_n$$ is the cell complex defined as follows. For any binary tree $T$ with $n$ leaves and $m$ corks take the $(2m+n-2)$-cube
$$H_T=[0,1]^{2m+n-2}$$
with the convention $H_{T}=\bullet$ if the exponent is negative%$m=0$ and $n=1$
, i.e.\ if $\norm{T}=|$. 

We regard $(x_1,\dots, x_{2m+n-2})\in H_T$ as a labeling of $\norm{T}$ with the $i^{\text{th}}$ inner edge labeled by $x_i$,

$$\norm{T}\quad=\quad\begin{array}{c}
\xy/r1.3pt/:
(0,0)*{}="a",
%(0,-3)*{a},
(0,10)*-<3pt>{\bullet}="b",
(-10,20)*{}="c",
%(-10,23)*{c},
(10,20)*-<3pt>{\bullet}="d",
(0,30)*-<3pt>{\bullet}="e",
(20,30)*{}="g",
%(20,33)*{g},
(-8,40)*{}="h",
%(-10,43)*{h},
%(0,43)*{i},
(8,40)*-<3pt>{\bullet}="j",
%(10,43)*{j},
\ar@{-}"a";"b",
\ar@{-}"b";"c",
\ar@{-}"b";"d",
\ar@{-}"d";"e",
\ar@{-}"d";"g",
\ar@{-}"e";"h",
\ar@{-}"e";"j",
\endxy
\end{array},
\qquad\qquad\qquad
\begin{array}{c}
\xy/r1.3pt/:
(0,0)*{}="a",
%(0,-3)*{a},
(0,10)*-<3pt>{\bullet}="b",
(-10,20)*{}="c",
%(-10,23)*{c},
(10,20)*-<3pt>{\bullet}="d",
(8,12)*{\scriptstyle x_1},
(8,32)*{\scriptstyle x_3},
(0,30)*-<3pt>{\bullet}="e",
(2,23)*{\scriptstyle x_2},
(20,30)*{}="g",
%(20,33)*{g},
(-8,40)*{}="h",
%(-10,43)*{h},
%(0,43)*{i},
(8,40)*-<3pt>{\bullet}="j",
%(10,43)*{j},
\ar@{-}"a";"b",
\ar@{-}"b";"c",
\ar@{-}"b";"d",
\ar@{-}"d";"e",
\ar@{-}"d";"g",
\ar@{-}"e";"h",
\ar@{-}"e";"j",
\endxy
\end{array}= (x_1,x_2,x_3)\in H_T.
$$

The cell complex $K^u_n$ is the quotient of the disjoint union 
$$\coprod_{\begin{array}{c}\\[-5mm]\text{\scriptsize binary trees $T$}\\[-1.5mm]\text{\scriptsize with $n$ leaves}\end{array}}H_T$$
%indexed by the binary trees with $n$ leaves
by the equivalence relation $\sim$ generated by the following relations: 
Consider two binary trees $T$ and $T'$ with $n$ leaves and $m$ corks. Denote
$e_{i}$ (resp.\ $e_{i}'$) the $i^{\text{th}}$ inner edge of $T$ (resp.~$T'$).

\begin{itemize}
 \item[(R1)] If 
$e_i$ and $e'_j$ do not contain corks and $T/e_i=T'/e_j'$, then $\sim$ 
identifies the $i^{\text{th}}$ negative face of $H_T$ with the $j^{\text{th}}$ negative face of $H_{T'}$, i.e.\ 
$\partial^-_i(x)\sim \partial^-_j(x)$ 
for all $x\in [0,1]^{2m+n-3}$. See \eqref{asociativo} for an illustration.
  
\item[(R2)] If $e_i$ contains a cork, then $\sim$ 
identifies each point in the $i^{\text{th}}$ negative face of $H_T$ with a point in $H_{T\backslash e_i}$,
$$\xymatrix{H_T=[0,1]^{2m+n-2}&
\ar[l]_-{\partial_i^-}[0,1]^{2m+n-3}\ar@{->>}[r]^-\eps
&H_{T\backslash e_i}},$$
i.e. $\partial^-_i(x)\sim \eps(x)$ for all $x\in[0,1]^{2m+n-3}$,
where $\eps$ is defined according to the local structure of $T$ around $e_i$. %See Figure~\ref{k21} for an illustration. 
There are eight possible cases:
$$\begin{array}{|c|c||c|c|}
  \norm{T} &\sim&\norm{T} &\sim\\[2pt]
\hline \begin{array}{c}\\[-2mm]
\xy/r1.3pt/:
(0,-10)*{}="o",
(0,0)*-{\bullet}="a",
(-8,4)*{\scriptstyle e_{i-1}},
(-8,12)*{\scriptstyle e_{i}},
(12,12)*{\scriptstyle e_{i+1}},
(0,10)*-<2.8pt>{\bullet}="b",
(-10,20)*-{\bullet}="c",
(10,20)*-<3pt>{\bullet}="d",
(10,30)*{}="f",
\ar@{.}"o";"a",
\ar@{-}"a";"b",
\ar@{-}"b";"c",
\ar@{-}"b";"d",
\ar@{.}"d";"f",
\endxy\\[-2mm]{}
\end{array}
& 
\partial^-_i(x)\st{\text{R2a}}\sim \gamma_{i-1}(x)
&
\begin{array}{c}
\\[-2mm]
\xy/r1.3pt/:
(0,-10)*{}="o",
(0,0)*-{\bullet}="a",
(-5,4)*{\scriptstyle e_{j}},
(-12,12)*{\scriptstyle e_{j+1}},
(9,12)*{\scriptstyle e_{i}},
(0,10)*-<2.8pt>{\bullet}="b",
(-10,20)*-{\bullet}="c",
(10,20)*-<3pt>{\bullet}="d",
(-10,30)*{}="f",
\ar@{.}"o";"a",
\ar@{-}"a";"b",
\ar@{-}"b";"c",
\ar@{-}"b";"d",
\ar@{.}"c";"f",
\endxy\\[-2mm]{}
\end{array}
&
\partial^-_i(x)\st{\text{R2b}}\sim \gamma_{j}(x)\\
\hline\begin{array}{c}
\\[-2mm]
\xy/r1.3pt/:
(0,-10)*{}="o",
(0,0)*-{\bullet}="a",
(-8,4)*{\scriptstyle e_{i-1}},
(-8,12)*{\scriptstyle e_{i}},
(0,10)*-<2.8pt>{\bullet}="b",
(-10,20)*-{\bullet}="c",
(10,20)*{}="d",
\ar@{.}"o";"a",
\ar@{-}"a";"b",
\ar@{-}"b";"c",
\ar@{-}"b";"d",
\endxy\\[-2mm]{}
\end{array}
&\partial^-_i(x)\st{\text{R2c}}\sim \pi_{i-1}(x)
&
\begin{array}{c}
\\[-2mm]
\xy/r1.3pt/:
(0,-10)*{}="o",
(0,0)*-{\bullet}="a",
(-8,4)*{\scriptstyle e_{i-1}},
(8,12)*{\scriptstyle e_{i}},
(0,10)*-<2.8pt>{\bullet}="b",
(-10,20)*{}="c",
(10,20)*-<3pt>{\bullet}="d",
\ar@{.}"o";"a",
\ar@{-}"a";"b",
\ar@{-}"b";"c",
\ar@{-}"b";"d",
\endxy\\[-2mm]{}
\end{array}
&\partial^-_i(x)\st{\text{R2d}}\sim \pi_{i-1}(x)\\
\hline\begin{array}{c}
\\[-2mm]
\xy/r1.3pt/:
(0,0)*{}="a",
(-13,12)*{\scriptstyle e_i=e_{1}},
(8,12)*{\scriptstyle e_{2}},
(0,10)*-<2.8pt>{\bullet}="b",
(-10,20)*-{\bullet}="c",
(10,20)*-<3pt>{\bullet}="d",
(10,30)*{}="f",
\ar@{-}"a";"b",
\ar@{-}"b";"c",
\ar@{-}"b";"d",
\ar@{.}"d";"f",
\endxy\\[-2mm]{}
\end{array}
&
\partial^-_1(x)\st{\text{R2e}}\sim \pi_{1}(x)
&
\begin{array}{c}
\\[-2mm]
\xy/r1.3pt/:
(0,0)*{}="a",
(-8,12)*{\scriptstyle e_{1}},
(24,12)*{\scriptstyle e_{i}=e_{2m+n-2}},
(0,10)*-<2.8pt>{\bullet}="b",
(-10,20)*-{\bullet}="c",
(10,20)*-<3pt>{\bullet}="d",
(-10,30)*{}="f",
\ar@{-}"a";"b",
\ar@{-}"b";"c",
\ar@{-}"b";"d",
\ar@{.}"c";"f",
\endxy\\[-2mm]{}
\end{array}
&
\partial^-_{2m+n-2}(x)\st{\text{R2f}}\sim \pi_{1}(x)\\
\hline\begin{array}{c}
\\[-2mm]
\xy/r1.3pt/:
(0,0)*{}="a",
(-13,12)*{\scriptstyle e_i=e_{1}},
(0,10)*-<2.8pt>{\bullet}="b",
(-10,20)*-{\bullet}="c",
(10,20)*{}="d",
\ar@{-}"a";"b",
\ar@{-}"b";"c",
\ar@{-}"b";"d",
\endxy\\[-2mm]{}
\end{array}
&
\partial^-_1(x)\st{\text{R2g}}\sim \bullet\in H_|
&
\begin{array}{c}
\\[-2mm]
\xy/r1.3pt/:
(0,0)*{}="a",
(13,12)*{\scriptstyle e_i=e_{1}},
(0,10)*-<2.8pt>{\bullet}="b",
(-10,20)*{}="c",
(10,20)*-<3pt>{\bullet}="d",
\ar@{-}"a";"b",
\ar@{-}"b";"c",
\ar@{-}"b";"d",
\endxy\\[-2mm]{}
\end{array}
&
\partial^-_1(x)\st{\text{R2h}}\sim \bullet\in H_|\\\hline
  \end{array}
$$

\bigskip

\noindent These relations can be illustrated as follows:

$$\begin{array}{|c|}
\hline\\[-8pt]
  \qquad\;\begin{array}{c}\\[-2mm]
\xy/r1.3pt/:
(0,-10)*{}="o",
(0,0)*-{\bullet}="a",
(-4,4)*{\scriptstyle a},
(-8,12)*{\scriptstyle 0},
(8,12)*{\scriptstyle b},
(0,10)*-<2.8pt>{\bullet}="b",
(-10,20)*-{\bullet}="c",
(10,20)*-<3pt>{\bullet}="d",
(10,30)*{}="f",
\ar@{.}"o";"a",
\ar@{-}"a";"b",
\ar@{-}"b";"c",
\ar@{-}"b";"d",
\ar@{.}"d";"f",
\endxy\\[-2mm]{}
\end{array}
\quad\st{\text{R2a}}\sim\quad
\begin{array}{c}\\[-2mm]
\xy/r1.3pt/:
(0,-10)*{}="o",
(0,0)*-{\bullet}="a",
(15,10)*{\scriptstyle \max(a,b)},
(0,20)*-<3pt>{\bullet}="d",
(0,30)*{}="f",
\ar@{.}"o";"a",
\ar@{-}"a";"d",
\ar@{.}"d";"f",
\endxy\\[-2mm]{}
\end{array}
\st{\text{R2b}}\sim\quad
\begin{array}{c}
\\[-2mm]
\xy/r1.3pt/:
(0,-10)*{}="o",
(0,0)*-{\bullet}="a",
(-4,4)*{\scriptstyle a},
(-8,12)*{\scriptstyle b},
(8,12)*{\scriptstyle 0},
(0,10)*-<2.8pt>{\bullet}="b",
(-10,20)*-{\bullet}="c",
(10,20)*-<3pt>{\bullet}="d",
(-10,30)*{}="f",
\ar@{.}"o";"a",
\ar@{-}"a";"b",
\ar@{-}"b";"c",
\ar@{-}"b";"d",
\ar@{.}"c";"f",
\endxy\\[-2mm]{}
\end{array}
\\
\hline\begin{array}{c}
\\[-2mm]
\xy/r1.3pt/:
(0,-10)*{}="o",
(0,0)*-{\bullet}="a",
(-4,4)*{\scriptstyle a},
(-8,12)*{\scriptstyle 0},
(0,10)*-<2.8pt>{\bullet}="b",
(-10,20)*-{\bullet}="c",
(10,20)*{}="d",
\ar@{.}"o";"a",
\ar@{-}"a";"b",
\ar@{-}"b";"c",
\ar@{-}"b";"d",
\endxy\\[-2mm]{}
\end{array}
\quad\st{\text{R2c}}\sim\quad
\begin{array}{c}\\[-2mm]
\xy/r1.3pt/:
(0,-10)*{}="o",
(0,0)*-{\bullet}="a",
(0,20)*{}="d",
\ar@{.}"o";"a",
\ar@{-}"a";"d",
\endxy\\[-2mm]{}
\end{array}
\quad\st{\text{R2d}}\sim\quad
\begin{array}{c}
\\[-2mm]
\xy/r1.3pt/:
(0,-10)*{}="o",
(0,0)*-{\bullet}="a",
(-4,4)*{\scriptstyle a},
(8,12)*{\scriptstyle 0},
(0,10)*-<2.8pt>{\bullet}="b",
(-10,20)*{}="c",
(10,20)*-<3pt>{\bullet}="d",
\ar@{.}"o";"a",
\ar@{-}"a";"b",
\ar@{-}"b";"c",
\ar@{-}"b";"d",
\endxy\\[-2mm]{}
\end{array}
\\
\hline\begin{array}{c}
\\[-2mm]
\xy/r1.3pt/:
(0,0)*{}="a",
(-8,12)*{\scriptstyle 0},
(8,12)*{\scriptstyle a},
(0,10)*-<2.8pt>{\bullet}="b",
(-10,20)*-{\bullet}="c",
(10,20)*-<3pt>{\bullet}="d",
(10,30)*{}="f",
\ar@{-}"a";"b",
\ar@{-}"b";"c",
\ar@{-}"b";"d",
\ar@{.}"d";"f",
\endxy\\[-2mm]{}
\end{array}
\quad\st{\text{R2e}}\sim\quad
\begin{array}{c}\\[-2mm]
\xy/r1.3pt/:
(0,0)*{}="a",
(0,20)*-<3pt>{\bullet}="d",
(0,30)*{}="f",
\ar@{-}"a";"d",
\ar@{.}"d";"f",
\endxy\\[-2mm]{}
\end{array}
\quad\st{\text{R2f}}\sim\quad
\begin{array}{c}
\\[-2mm]
\xy/r1.3pt/:
(0,0)*{}="a",
(-8,12)*{\scriptstyle a},
(7,12)*{\scriptstyle 0},
(0,10)*-<2.8pt>{\bullet}="b",
(-10,20)*-{\bullet}="c",
(10,20)*-<3pt>{\bullet}="d",
(-10,30)*{}="f",
\ar@{-}"a";"b",
\ar@{-}"b";"c",
\ar@{-}"b";"d",
\ar@{.}"c";"f",
\endxy\\[-2mm]{}
\end{array}
\\
\hline\begin{array}{c}
\\[-2mm]
\xy/r1.3pt/:
(0,0)*{}="a",
(-8,12)*{\scriptstyle 0},
(0,10)*-<2.8pt>{\bullet}="b",
(-10,20)*-{\bullet}="c",
(10,20)*{}="d",
\ar@{-}"a";"b",
\ar@{-}"b";"c",
\ar@{-}"b";"d",
\endxy\\[-2mm]{}
\end{array}
\quad\st{\text{R2g}}\sim\quad
\begin{array}{c}\\[-2mm]
\xy/r1.3pt/:
(0,0)*{}="a",
(0,20)*{}="d",
\ar@{-}"a";"d",
\endxy\\[-2mm]{}
\end{array}
\quad\st{\text{R2h}}\sim\quad
\begin{array}{c}
\\[-2mm]
\xy/r1.3pt/:
(0,0)*{}="a",
(8,12)*{\scriptstyle 0},
(0,10)*-<2.8pt>{\bullet}="b",
(-10,20)*{}="c",
(10,20)*-<3pt>{\bullet}="d",
\ar@{-}"a";"b",
\ar@{-}"b";"c",
\ar@{-}"b";"d",
\endxy\\[-2mm]{}
\end{array}\\\hline
  \end{array}
$$

\end{itemize}

\bigskip

\begin{prop}\label{circi}
For $p\geq 1$ and $q\geq 0$ there are maps
$$\circ_j\colon K^u_p\times K^u_q\To K^u_{p+q}$$
defined on cells by maps
$$c_j(T,T')\colon H_T\times H_{T'}\To H_{T\circ_jT'}$$
where $T$ and $T'$ are binary trees with $p$ and $q$ leaves and $s$ and $t$ corks, respectively.
If %$e_{k}\in I(T\circ_j T')$ is
the identification of the root edge of $T'$ with the $j^{\text{th}}$ leaf edge of $T$ is the $k^{\text{th}}$ inner edge of $T\circ_j T'$
then the map $c_j(T,T')$ is given by
$$
\xymatrix@C=20pt{%c_j(T,T')\colon 
H_T\times H_{T'}=[0,1]^{\alpha}\times[0,1]^{\beta}\ar[r]^-{\sigma_{
%\alpha,\beta,
k}}& [0,1]^{\alpha+\beta}\ar[r]^-{\partial^+_{k}}& [0,1]^{\alpha+\beta+1}=H_{T\circ_jT'}.}
$$
Here $\alpha=2s+p-2$ and $\beta=2t+q-2$ are the numbers of inner edges of $T$ and $T'$, respectively. If $\norm{T}=|$ or $\norm{T'}=|$ then $c_j(T,T')$ is the identity.
\end{prop}

With our graphical notation, the composition law $\circ_j$ can be illustrated as follows:
$$
\begin{array}{c}
\xy/r1.3pt/:
(0,0)*{}="a",
(0,10)*-<3pt>{\bullet}="b",
(-10,20)*-<3pt>{\bullet}="c",
(10,20)*-<3pt>{\bullet}="d",
(-8,12)*{\scriptstyle x_1},
(-17,23)*{\scriptstyle x_2},
(17,23)*{\scriptstyle x_4},
(8,12)*{\scriptstyle x_3},
(4,30)*{}="e",
(16,30)*-<3pt>{\bullet}="g",
(-16,30)*-<3pt>{\bullet}="ee",
(-4,30)*{}="gg",
\ar@{-}"a";"b",
\ar@{-}"b";"c",
\ar@{-}"b";"d",
\ar@{-}"d";"e",
\ar@{-}"d";"g",
\ar@{-}"c";"ee",
\ar@{-}"c";"gg",
\endxy
\end{array}
\circ_1
\begin{array}{c}
\xy/r1.3pt/:
(0,0)*{}="a",
(0,10)*-<3pt>{\bullet}="b",
(-6,20)*{}="c",
(6,20)*-<3pt>{\bullet}="d",
(8,14)*{\scriptstyle y_1},
(8,34)*{\scriptstyle y_3},
(0,30)*-<3pt>{\bullet}="e",
(0,23)*{\scriptstyle y_2},
(12,30)*{}="g",
(-6,40)*{}="h",
(6,40)*-<3pt>{\bullet}="j",
\ar@{-}"a";"b",
\ar@{-}"b";"c",
\ar@{-}"b";"d",
\ar@{-}"d";"e",
\ar@{-}"d";"g",
\ar@{-}"e";"h",
\ar@{-}"e";"j",
\endxy
\end{array}
=
\begin{array}{c}
\xy/r1.3pt/:
(0,0)*{}="a",
(0,10)*-<3pt>{\bullet}="b",
(-10,20)*-<3pt>{\bullet}="c",
(10,20)*-<3pt>{\bullet}="d",
(-8,12)*{\scriptstyle x_1},
(-17,23)*{\scriptstyle x_2},
(17,23)*{\scriptstyle x_4},
(8,12)*{\scriptstyle x_3},
(4,30)*{}="e",
(16,30)*-<3pt>{\bullet}="g",
(-16,30)*-<3pt>{\bullet}="ee",
(-4,30)*{}="gg",
(-4,30)*-<3pt>{\bullet}="bbb",
(-10,40)*{}="ccc",
(2,40)*-<3pt>{\bullet}="ddd",
(4,34)*{\scriptstyle y_1},
(4,54)*{\scriptstyle y_3},
(-4,24)*{\scriptstyle 1},
(-4,50)*-<3pt>{\bullet}="eee",
(-4,43)*{\scriptstyle y_2},
(8,50)*{}="ggg",
(-10,60)*{}="hhh",
(2,60)*-<3pt>{\bullet}="jjj",
\ar@{-}"bbb";"ccc",
\ar@{-}"bbb";"ddd",
\ar@{-}"ddd";"eee",
\ar@{-}"ddd";"ggg",
\ar@{-}"eee";"hhh",
\ar@{-}"eee";"jjj",
\ar@{-}"a";"b",
\ar@{-}"b";"c",
\ar@{-}"b";"d",
\ar@{-}"d";"e",
\ar@{-}"d";"g",
\ar@{-}"c";"ee",
\ar@{-}"c";"gg",
\endxy
\end{array}=(x_1,x_2,1,y_1,y_2,y_3,x_3,x_4).
$$

\begin{proof}[Proof of Proposition \ref{circi}]
We must check the compatibility with the nine relations generating $\sim$. Given $x\in H_{T_1}$ and $x'\in H_{T_2}$ satisfying $x\sim x'$ we must show that $c_j(T_1,T')(x,y)\sim c_j(T_2,T')(x',y)$ for any $y\in H_{T'}$. Similarly, if $y\in H_{T_1'}$ and $y'\in H_{T_2'}$ satisfy $y\sim y'$ we must show that $c_j(T,T'_1)(x,y)\sim c_j(T,T'_2)(x,y')$ for any $x\in H_{T}$. In the first (resp.\ second) case the compatibility condition is obvious when $T'$ (resp.\ $T$) is $|$. We exclude these cases from now on.

If  $x\sim x'$ or $y\sim y'$ by (R1) then the compatibility conditions are illustrated by the following diagrams, which show (R1) relations between faces of $H_{T_1\circ_jT'}$ and $H_{T_2\circ_jT'}$, and $H_{T\circ_jT'_1}$ and $H_{T\circ_jT'_2}$, derived from the former relations.

$$
\xy
(0,20)*+[F]{T'},
(0,0)*+[F]{\begin{array}{c}
\xy/r1.3pt/:
(0,0)*{}="a",
(0,10)*-<3pt>{\bullet}="b",
(10,20)*{}="c",
(10,30)*{}="cc",
(-10,20)*-<3pt>{\bullet}="d",
(-7,13)*{\scriptstyle 0},
(-20,0)*{\scriptstyle T_1},
(0,30)*{}="e",
(0,40)*{}="ee",
(-20,30)*{}="g",
(-20,40)*{}="gg",
\ar@{.}"a";"b",
\ar@{.}"c";"cc",
\ar@{.}"e";"ee",
\ar@{.}"g";"gg",
\ar@{-}"b";"c",
\ar@{-}"b";"d",
\ar@{-}"d";"e",
\ar@{-}"d";"g",
\endxy
\end{array}},
(0,11)*{\bullet},
(0,17.5)*{\bullet},
(2,14.25)*{\scriptstyle 1}
\ar@{-}(0,11);(0,17.5)
 \endxy
\quad\sim\quad
\xy
(0,20)*+[F]{T'},
(0,0)*+[F]{\begin{array}{c}
\xy/r1.3pt/:
(0,0)*{}="a",
(0,10)*-<3pt>{\bullet}="b",
(-10,20)*{}="c",
(-10,30)*{}="cc",
(10,20)*-<3pt>{\bullet}="d",
(7,13)*{\scriptstyle 0},
(20,0)*{\scriptstyle T_2},
(0,30)*{}="e",
(0,40)*{}="ee",
(20,30)*{}="g",
(20,40)*{}="gg",
\ar@{.}"a";"b",
\ar@{.}"c";"cc",
\ar@{.}"e";"ee",
\ar@{.}"g";"gg",
\ar@{-}"b";"c",
\ar@{-}"b";"d",
\ar@{-}"d";"e",
\ar@{-}"d";"g",
\endxy
\end{array}},
(0,11)*{\bullet},
(0,17.5)*{\bullet},
(2,14.25)*{\scriptstyle 1}
\ar@{-}(0,11);(0,17.5)
\endxy
\qquad\qquad
\begin{array}{c}\\[-2.9cm]
\xy
(0,-19.8)*+[F]{T},
(0,0)*+[F]{\begin{array}{c}
\xy/r1.3pt/:
(0,0)*{}="a",
(0,10)*-<3pt>{\bullet}="b",
(10,20)*{}="c",
(10,30)*{}="cc",
(-10,20)*-<3pt>{\bullet}="d",
(-7,13)*{\scriptstyle 0},
(-20,0)*{\scriptstyle T_1'},
(0,30)*{}="e",
(0,40)*{}="ee",
(-20,30)*{}="g",
(-20,40)*{}="gg",
\ar@{.}"a";"b",
\ar@{.}"c";"cc",
\ar@{.}"e";"ee",
\ar@{.}"g";"gg",
\ar@{-}"b";"c",
\ar@{-}"b";"d",
\ar@{-}"d";"e",
\ar@{-}"d";"g",
\endxy
\end{array}},
(0,-11)*{\bullet},
(0,-17.5)*{\bullet},
(2,-14.25)*{\scriptstyle 1}
\ar@{-}(0,-11);(0,-17.5)
 \endxy
\quad\sim\quad
\begin{array}{c}\\[.87cm]
\xy
(0,-19.8)*+[F]{T},
(0,0)*+[F]{\begin{array}{c}
\xy/r1.3pt/:
(0,0)*{}="a",
(0,10)*-<3pt>{\bullet}="b",
(-10,20)*{}="c",
(-10,30)*{}="cc",
(10,20)*-<3pt>{\bullet}="d",
(7,13)*{\scriptstyle 0},
(20,0)*{\scriptstyle T_2'},
(0,30)*{}="e",
(0,40)*{}="ee",
(20,30)*{}="g",
(20,40)*{}="gg",
\ar@{.}"a";"b",
\ar@{.}"c";"cc",
\ar@{.}"e";"ee",
\ar@{.}"g";"gg",
\ar@{-}"b";"c",
\ar@{-}"b";"d",
\ar@{-}"d";"e",
\ar@{-}"d";"g",
\endxy
\end{array}},
(0,-11)*{\bullet},
(0,-17.5)*{\bullet},
(2,-14.25)*{\scriptstyle 1}
\ar@{-}(0,-11);(0,-17.5)
\endxy
\end{array}
\end{array}
$$

\medskip

The same can be done for relations (R2a) and (R2b). Also for (R2c) and (R2d) unless $e_i$ is adjacent to the $j^{\text{th}}$ leaf and $x\sim x'$ because of one of these relations. In these cases the compatibility relations follow from (R2a) and (R2b) respectively, as illustrated by the following diagrams.
$$\begin{array}{c}
\\[-2mm]
\xy/r1.3pt/:
(0,-10)*{}="o",
(0,0)*-{\bullet}="a",
(-4,4)*{\scriptstyle a},
(-8,12)*{\scriptstyle 0},
(8,12)*{\scriptstyle 1},
(0,10)*-<2.8pt>{\bullet}="b",
(-10,20)*-{\bullet}="c",
(10,20)*-<3pt>{\bullet}="d",
(10,26)*+[F]{T'}
\ar@{.}"o";"a",
\ar@{-}"a";"b",
\ar@{-}"b";"c",
\ar@{-}"b";"d",
\endxy\\[-2mm]{}
\end{array}
\quad\st{\text{R2a}}\sim\quad
\begin{array}{c}\\[-2mm]
\xy/r1.3pt/:
(0,-10)*{}="o",
(0,0)*-{\bullet}="a",
(19,10)*{\scriptstyle \max(a,1)=1},
(0,20)*-{\bullet}="d",
(0,26)*+[F]{T'}
\ar@{.}"o";"a",
\ar@{-}"a";"d",
\endxy\\[-2mm]{}
\end{array}
\quad\st{\text{R2b}}\sim\quad
\begin{array}{c}
\\[-2mm]
\xy/r1.3pt/:
(0,-10)*{}="o",
(0,0)*-{\bullet}="a",
(-4,4)*{\scriptstyle a},
(8,12)*{\scriptstyle 0},
(-8,12)*{\scriptstyle 1},
(0,10)*-<2.8pt>{\bullet}="b",
(-10,20)*-{\bullet}="c",
(-10,26)*+[F]{T'},
(10,20)*-<3pt>{\bullet}="d",
\ar@{.}"o";"a",
\ar@{-}"a";"b",
\ar@{-}"b";"c",
\ar@{-}"b";"d",
\endxy\\[-2mm]{}
\end{array}$$

We can also proceed in the simple way if $x\sim x'$ by (R2e) or (R2f). If $y\sim y'$ by (R2e) or (R2f) then the compatibility condition follows from 
$$\begin{array}{c}
\\[-2mm]
\xy/r1.3pt/:
(0,-5)*+[F]{T},
(0,0)*-{\bullet}="a",
(-3,5)*{\scriptstyle 1},
(-8,12)*{\scriptstyle 0},
(8,12)*{\scriptstyle a},
(0,10)*-<2.8pt>{\bullet}="b",
(-10,20)*-{\bullet}="c",
(10,20)*-<3pt>{\bullet}="d",
(10,30)*{}="f",
\ar@{-}"a";"b",
\ar@{-}"b";"c",
\ar@{-}"b";"d",
\ar@{.}"d";"f",
\endxy\\[-2mm]{}
\end{array}
\quad\st{\text{R2a}}\sim\quad
\begin{array}{c}\\[-2mm]
\xy/r1.3pt/:
(0,-5)*+[F]{T},
(0,0)*-{\bullet}="a",
(19,10)*{\scriptstyle \max(a,1)=1},
(0,20)*-<3pt>{\bullet}="d",
(0,30)*{}="f",
\ar@{-}"a";"d",
\ar@{.}"d";"f",
\endxy\\[-2mm]{}
\end{array}
\quad\st{\text{R2b}}\sim\quad
\begin{array}{c}
\\[-2mm]
\xy/r1.3pt/:
(0,-5)*+[F]{T},
(0,0)*-{\bullet}="a",
(-3,5)*{\scriptstyle 1},
(-8,12)*{\scriptstyle a},
(7,12)*{\scriptstyle 0},
(0,10)*-<2.8pt>{\bullet}="b",
(-10,20)*-{\bullet}="c",
(10,20)*-<3pt>{\bullet}="d",
(-10,30)*{}="f",
\ar@{-}"a";"b",
\ar@{-}"b";"c",
\ar@{-}"b";"d",
\ar@{.}"c";"f",
\endxy\\[-2mm]{}
\end{array}$$

The compatibility conditions for (R2g) and (R2h) are illustrated below.
$$\begin{array}{c}
\\[-2mm]
\xy/r1.3pt/:
(0,0)*{}="a",
(-8,12)*{\scriptstyle 0},
(8,12)*{\scriptstyle 1},
(0,10)*-<2.8pt>{\bullet}="b",
(-10,20)*-{\bullet}="c",
(10,20)*-<3pt>{\bullet}="d",
(10,26)*+[F]{T'}
\ar@{-}"a";"b",
\ar@{-}"b";"c",
\ar@{-}"b";"d",
\endxy\\[-2mm]{}
\end{array}
\quad\st{\text{R2e}}\sim\quad
\begin{array}{c}\\[-2mm]
\xy/r1.3pt/:
(0,0)*{}="a",
(0,20)*-{\bullet}="d",
(0,26)*+[F]{T'}
\ar@{-}"a";"d",
\endxy\\[-2mm]{}
\end{array}
\quad\st{\text{R2f}}\sim\quad
\begin{array}{c}
\\[-2mm]
\xy/r1.3pt/:
(0,0)*{}="a",
(8,12)*{\scriptstyle 0},
(-8,12)*{\scriptstyle 1},
(0,10)*-<2.8pt>{\bullet}="b",
(-10,20)*-{\bullet}="c",
(-10,26)*+[F]{T'},
(10,20)*-<3pt>{\bullet}="d",
\ar@{-}"a";"b",
\ar@{-}"b";"c",
\ar@{-}"b";"d",
\endxy\\[-2mm]{}
\end{array}$$

$$\begin{array}{c}
\\[-2mm]
\xy/r1.3pt/:
(0,-5)*+[F]{T},
(0,0)*-{\bullet}="a",
(-3,5)*{\scriptstyle 1},
(-8,12)*{\scriptstyle 0},
(0,10)*-<2.8pt>{\bullet}="b",
(-10,20)*-{\bullet}="c",
(10,20)*{}="d",
\ar@{-}"a";"b",
\ar@{-}"b";"c",
\ar@{-}"b";"d",
\endxy\\[-2mm]{}
\end{array}
\quad\st{\text{R2c}}\sim\quad
\begin{array}{c}\\[-2mm]
\xy/r1.3pt/:
(0,-5)*+[F]{T},
(0,0)*-{\bullet}="a",
(0,20)*{}="d",
\ar@{-}"a";"d",
\endxy\\[-2mm]{}
\end{array}
\quad\st{\text{R2d}}\sim\quad
\begin{array}{c}
\\[-2mm]
\xy/r1.3pt/:
(0,-5)*+[F]{T},
(0,0)*-{\bullet}="a",
(-3,5)*{\scriptstyle 1},
(7,12)*{\scriptstyle 0},
(0,10)*-<2.8pt>{\bullet}="b",
(-10,20)*-{\bullet}="c",
(10,20)*{}="d",
\ar@{-}"a";"b",
\ar@{-}"b";"c",
\ar@{-}"b";"d",
\endxy\\[-2mm]{}
\end{array}$$

\end{proof}

Once we have checked that the composition maps given by grafting of labeled trees are well defined, it is straightforward to see that they define an operad.

\begin{thm}
Unital associahedra form a topological operad, called $\mathtt{uA}_\infty$, with the composition laws in Proposition \ref{circi} and unit $u=\bullet\in H_{|}\subset K_1^u$.
\end{thm}

\begin{rem}
One can actually show that $\mathtt{uA}_\infty$ may be obtained by applying a non-symmetric Boardman--Vogt construction to the operad $\mathtt{Ass}$ for topological monoids, $\mathtt{Ass}(n)=\bullet$, $n\geq 0$. Therefore $\mathtt{uA}_\infty$ is a cofibrant operad in the model category of topological operads \cite{ahto,tbvrommc,htnso}.
\end{rem}

Our goal now is to endow unital associahedra with a new smaller cellular structure generalizing the polytope cellular structure of non-unital associahedra.

Notice that $K_n^u$ admits a filtration by finite-dimensional subcomplexes $$K_n^u=\bigcup_{m=0}^\infty K_{n,m}^u$$
where $K_{n,m}^u$ is the quotient of the disjoint union 
$$\coprod_{\begin{array}{c}\\[-5mm]\text{\scriptsize binary trees $T$}\\[-1.5mm]\text{\scriptsize with $n$ leaves}
\\[-1.5mm]\text{\scriptsize and  $\leq m$ corks}
\end{array}}H_T$$
%indexed by the binary trees with $n$ leaves
by the restriction of the equivalence relation generated by (R1, R2).
For $m=0$ we obtain the associahedra
$$K^u_{n,0}=K_n.$$
The following proposition shows that passing from one stage in the filtration to the next consists of adjoining cells
$$
\coprod\limits_{
\begin{array}{c}\\[-15pt]
\scriptstyle S\subset\ord{n+m}\\[-3pt]
\scriptstyle \abs{S}=m
\end{array}}\hspace{-17pt}K_{n+m}\times[0,1]^m
.
$$
%Here we write $[k]=\{1,\dots,k\}$. 
Given a set $S\subseteq [n+m]$, $|S|=m$, and a binary tree $T$ with $n$ leaves and no corks,
the \emph{shuffle permutation of coordinates}
$$H_T\times[0,1]^{m}\st{\cong}\To H_{T^{\bullet S}}$$
puts the last $m$ coordinates in the positions of the new $m$ inner edges in $T^{\bullet S}$ formed by adding corks to $T$.
We consider also the negative cubical boundary, $m\geq1$,  
$$\xymatrix@C=6em{
\partial^-([0,1]^m)=\displaystyle\bigcup_{i=1}^m\partial^-_i([0,1]^{m-1})
\ar@{ >->}[r]_-\sim^-{\scriptstyle\text{incl.}}
    &
[0,1]^m.}
$$

%Notice that, for $1\leq i<j\leq m$,
%$$\partial^-_i([0,1]^{m-1})\cap
%\partial^-_j([0,1]^{m-1})
%=\partial^-_i\partial^-_{j-1}([0,1]^{m-2})
%=\partial^-_{j}\partial^-_i([0,1]^{m-2}).$$

\begin{prop}\label{puchau}
For  $n\geq 0$ and $m\geq 1$, $(n,m)\neq (0,1)$, there is a push-out diagram
$$\xy
(-3.5,-2)*{\mathop{\displaystyle\coprod}\limits_{
\begin{array}{c}\\[-15pt]
\scriptstyle S\subset\ord{n+m}
\\[-3pt]
\scriptstyle \abs{S}=m
\end{array}}\hspace{-17pt}K_{n+m}\times\partial^-([0,1]^m)},
(53,-2)*{\mathop{\displaystyle\coprod}\limits_{
\begin{array}{c}\\[-15pt]
\scriptstyle S\subset\ord{n+m}\\[-3pt]
\scriptstyle \abs{S}=m
\end{array}}\hspace{-17pt}K_{n+m}\times[0,1]^m},
(0,-20)*{K^u_{n,m-1}},
(60,-20)*{K^u_{n,m}},
(30,-10)*{\scriptstyle\text{push}}
\ar@{ >->}(18,1.5);(40,1.5)_\sim^{\coprod \id{}\!%K_{n+m}
\times\scriptstyle\text{incl.}}
\ar(0,-2);(0,-16)_{(f_{S})_{S}}
\ar@{ >->}(7,-19.5);(54,-19.5)_\sim^{\text{inclusion}}
\ar(60,-2);(60,-16)^{(\bar{f}_{S})_{S}}
\endxy$$
such that, if $S=\{j_1<\cdots<j_m\}$ and $T$ is a binary tree with $n+m$ leaves and no corks:
\begin{enumerate} 
\item The restriction of $\bar f_S$ to the product cell $H_T\times [0,1]^{m}$  maps it to the cell $H_{T^{\bullet S}}$ by the shuffle permutation of coordinates
$$H_T\times[0,1]^{m}\st{\cong}\To H_{T^{\bullet S}}.$$
\item For $1\leq i\leq m$, the restriction of $f_S$ to the product cell $H_T\times\partial^-_i([0,1]^{m-1})$ is given by a shuffle permutation of coordinates followed by a projection,
$$\xymatrix{H_T\times[0,1]^{m-1}\ar[r]^-{\cong}
&H_{T^{\bullet (S\setminus\{j_i\})}}=[0,1]^{2m+n-3}\ar@{->>}[r]^-\eps
&H_{T^{\bullet S}\backslash e}},$$
if the cork specified by $j_i\in S$ forms the inner edge $e$ of $T^{\bullet S}$ and the map $\eps$ is defined in (R2) according the local structure of $T^{\bullet S}$ around $e$.
\end{enumerate}
\end{prop}

\begin{proof}
The shuffle permutation of coordinates gives a homeomorphism
$$
{\coprod\limits_{
\begin{array}{c}\\[-15pt]
\scriptstyle S\subset\ord{n+m}\\[-3pt]
\scriptstyle \abs{S}=m
\end{array}}\hspace{-17pt}K_{n+m}\times[0,1]^m}\qquad
\cong
\left.\coprod_{\begin{array}{c}\\[-5mm]\text{\scriptsize binary trees $T$}\\[-1.3mm]\text{\scriptsize with $n$ leaves}
\\[-1.3mm]\text{\scriptsize and  $m$ corks} 
\end{array}}H_T\right/ \text{(R1)}$$
under which the $i^{\text{th}}$ negative face of $[0,1]^m$ corresponds to the $k^{\text{th}}$  negative face of $H_T$, 
if the edge $e$  defined in (2) is the $k^{\text{th}}$ inner edge of $T^{\bullet S}$.  
Thus the relations imposed by the pushout are exactly the relations (R2).
\end{proof}

\begin{figure}[H]
\scalebox{0.8} % Change this value to rescale the drawing.
{
\begin{pspicture}(0,-2.205)(8.98,2.2191112)
\psline[linewidth=0.02cm](4.01,0.365)(4.49,-0.115)
\psline[linewidth=0.02cm](4.49,-0.595)(4.49,0.365)
\psline[linewidth=0.02cm](4.49,-0.115)(4.97,0.365)
\psdots[dotsize=0.16](4.49,0.365)
\psdots[dotsize=0.16](4.49,-0.115)
\psline[linewidth=0.02cm](4.49,1.805)(4.49,1.485)
\psline[linewidth=0.02cm](4.49,1.805)(4.81,2.125)
\psline[linewidth=0.02cm](4.49,1.805)(4.17,2.125)
\psline[linewidth=0.02cm](4.49,1.805)(4.49,2.125)
\psdots[dotsize=0.16,fillstyle=solid,dotstyle=o](4.49,2.125)
\psdots[dotsize=0.12](4.49,1.805)
\psline[linewidth=0.02cm](5.93,-0.755)(5.93,-0.435)
\psline[linewidth=0.02cm](5.93,-0.435)(6.25,-0.115)
\psline[linewidth=0.02cm](5.93,-0.435)(5.61,-0.115)
\psdots[dotsize=0.12](5.93,-0.435)
\psline[linewidth=0.02cm](6.25,-0.115)(6.57,0.205)
\psline[linewidth=0.02cm](6.25,-0.115)(5.93,0.205)
\psdots[dotsize=0.12](6.25,-0.115)
\psdots[dotsize=0.12](5.93,0.205)
\psline[linewidth=0.02cm](3.05,-0.755)(3.05,-0.435)
\psline[linewidth=0.02cm](3.05,-0.435)(2.73,-0.115)
\psline[linewidth=0.02cm](2.73,-0.115)(2.41,0.205)
\psline[linewidth=0.02cm](2.73,-0.115)(3.05,0.205)
\psline[linewidth=0.02cm](3.05,-0.435)(3.37,-0.115)
\psdots[dotsize=0.12](2.73,-0.115)
\psdots[dotsize=0.12](3.05,-0.435)
\psdots[dotsize=0.12](3.05,0.205)
\psline[linewidth=0.02](1.29,1.165)(1.29,-1.395)(7.69,-1.395)(7.69,1.165)(1.29,1.165)
\psline[linewidth=0.02cm,linestyle=dashed,dash=0.16cm 0.16cm](4.49,1.165)(4.49,0.685)
\psdots[dotsize=0.12](6.09,1.805)
\psdots[dotsize=0.12](2.89,-1.875)
\psline[linewidth=0.02cm](2.89,-2.035)(2.89,-1.875)
\psline[linewidth=0.02cm](2.89,-1.875)(2.73,-1.715)
\psline[linewidth=0.02cm](2.89,-1.875)(3.05,-1.715)
\psdots[dotsize=0.12](6.09,-1.875)
\psline[linewidth=0.02cm](6.09,-2.035)(6.09,-1.875)
\psline[linewidth=0.02cm](6.09,-1.875)(5.93,-1.715)
\psline[linewidth=0.02cm](6.09,-1.875)(6.25,-1.715)
\psdots[dotsize=0.12](4.49,-1.875)
\psline[linewidth=0.02cm](4.49,-2.195)(4.49,-1.875)
\psline[linewidth=0.02cm](4.49,-1.875)(4.17,-1.555)
\psline[linewidth=0.02cm](4.49,-1.875)(4.81,-1.555)
\psdots[dotsize=0.12](1.13,-1.875)
\psline[linewidth=0.02cm](1.13,-2.195)(1.13,-1.875)
\psline[linewidth=0.02cm](1.13,-1.875)(0.81,-1.555)
\psline[linewidth=0.02cm](1.13,-1.875)(1.45,-1.555)
\psdots[dotsize=0.12](7.85,-1.875)
\psline[linewidth=0.02cm](7.85,-2.195)(7.85,-1.875)
\psline[linewidth=0.02cm](7.85,-1.875)(7.53,-1.555)
\psline[linewidth=0.02cm](7.85,-1.875)(8.17,-1.555)
\psline[linewidth=0.02cm,linestyle=dashed,dash=0.16cm 0.16cm](4.49,-0.915)(4.49,-1.395)
\psline[linewidth=0.02cm](3.05,1.325)(3.05,1.645)
\psline[linewidth=0.02cm](3.05,1.645)(2.89,1.805)
\psline[linewidth=0.02cm](2.89,1.805)(2.73,1.965)
\psline[linewidth=0.02cm](3.05,1.645)(3.21,1.805)
\psdots[dotsize=0.12](2.89,1.805)
\psdots[dotsize=0.12](3.05,1.645)
\psline[linewidth=0.02cm](5.93,1.325)(5.93,1.645)
\psline[linewidth=0.02cm](5.93,1.645)(6.09,1.805)
\psline[linewidth=0.02cm](5.93,1.645)(5.77,1.805)
\psdots[dotsize=0.12](5.93,1.645)
\psline[linewidth=0.02cm](6.25,1.965)(6.09,1.805)
\psline[linewidth=0.02cm](2.89,1.805)(3.05,1.965)
\psdots[dotsize=0.16,fillstyle=solid,dotstyle=o](3.05,1.965)
\psline[linewidth=0.02cm](6.09,1.805)(5.93,1.965)
\psdots[dotsize=0.16,fillstyle=solid,dotstyle=o](5.93,1.965)
\psline[linewidth=0.02cm](1.13,1.165)(1.13,1.485)
\psline[linewidth=0.02cm](1.13,1.485)(0.81,1.805)
\psline[linewidth=0.02cm](0.81,1.805)(0.49,2.125)
\psline[linewidth=0.02cm](1.13,1.485)(1.45,1.805)
\psdots[dotsize=0.12](0.81,1.805)
\psdots[dotsize=0.12](1.13,1.485)
\psline[linewidth=0.02cm](0.81,1.805)(1.13,2.125)
\psdots[dotsize=0.16,fillstyle=solid,dotstyle=o](1.13,2.125)
\psdots[dotsize=0.12](8.17,1.805)
\psline[linewidth=0.02cm](7.85,1.165)(7.85,1.485)
\psline[linewidth=0.02cm](7.85,1.485)(8.17,1.805)
\psline[linewidth=0.02cm](7.85,1.485)(7.53,1.805)
\psdots[dotsize=0.12](7.85,1.485)
\psline[linewidth=0.02cm](8.49,2.125)(8.17,1.805)
\psline[linewidth=0.02cm](8.17,1.805)(7.85,2.125)
\psdots[dotsize=0.16,fillstyle=solid,dotstyle=o](7.85,2.125)
\psline[linewidth=0.02cm](0.65,-0.755)(0.65,-0.435)
\psline[linewidth=0.02cm](0.65,-0.435)(0.33,-0.115)
\psline[linewidth=0.02cm](0.33,-0.115)(0.01,0.205)
\psline[linewidth=0.02cm](0.33,-0.115)(0.65,0.205)
\psline[linewidth=0.02cm](0.65,-0.435)(0.97,-0.115)
\psdots[dotsize=0.12](0.33,-0.115)
\psdots[dotsize=0.12](0.65,-0.435)
\psdots[dotsize=0.12](0.65,0.205)
\psline[linewidth=0.02cm](8.33,-0.755)(8.33,-0.435)
\psline[linewidth=0.02cm](8.33,-0.435)(8.65,-0.115)
\psline[linewidth=0.02cm](8.33,-0.435)(8.01,-0.115)
\psdots[dotsize=0.12](8.33,-0.435)
\psline[linewidth=0.02cm](8.65,-0.115)(8.97,0.205)
\psline[linewidth=0.02cm](8.65,-0.115)(8.33,0.205)
\psdots[dotsize=0.12](8.65,-0.115)
\psdots[dotsize=0.12](8.33,0.205)
%\usefont{T1}{ptm}{m}{n}
\rput(0.856875,1.43){$1$}
%\usefont{T1}{ptm}{m}{n}
\rput(0.376875,-0.49){$1$}
%\usefont{T1}{ptm}{m}{n}
\rput(8.536875,-0.49){$1$}
%\usefont{T1}{ptm}{m}{n}
\rput(8.056875,1.43){$1$}
\end{pspicture} 
}
\caption{A $2$-cell of $K^u_{2,1}$ depicting edges identified by (R1) and degenerate by (R2), cf.\ Figure \ref{firstfour}. An edge with a white cork has length $1$, as do the other edges so labeled.}\label{k21}
\end{figure}

Thus we have a cork filtration
$$\xymatrix{
K^u_{n,0}\ar@{ >->}[r]^-\sim&K^u_{n,1}\ar@{ >->}[r]^-\sim&K^u_{n,2}\ar@{ >->}[r]^-\sim&\cdots \ar@{ >->}[r]^-\sim &K^u_n
}$$
formed by attaching cells obtained from products of  associahedra with hypercubes,
 $$K_{n+m}\times [0,1]^m.$$
%If $T$ is a tree with $n$ leaves and $m$ corks and $T'$ the tree with $n+m$ leaves obtained by removing the corks from $T$, then 
%The associahedron $K_n$ can be represented as an $(n-2)$-dimensional convex polytope with $(n-3)$-faces $K_p\circ_i K_q$
%, $p+q-1=n$, $1\leq i\leq p$.
%The poset of faces of $K_n$ under inclusion may be identified with the category of planted planar trees with $n$ leaves and no corks, since inclusion of faces corresponds to collapse maps $T\to T/e$. We may write $K_T$ for the codimension $r$ face of $K_n$ corresponding to a tree $T$ with $n$ leaves and $r$ inner edges. 
%The unique $(n-2)$-cell is the height 2 tree with no internal edges. The codimension 1 faces $K_p\circ_i K_q$ of $K_n$ are the trees with a unique internal edge $e=\{v_1,v_{i+1}\}\in I(T)$, and more generally
%$$
%\partial(K_T)=\bigcup_{T_2\circ_i T_1\to T}
%K_{T}\circ_i K_{T'}=K_{T\circ_i T'}
%\bigcup_{T_2\circ_i T_1\to T} 
%$$  

\begin{cor}\label{arecontractible}
Unital associahedra are contractible.
\end{cor}

After the classical associahedron $K^u_{n,0}=K_n$, the next stage in the cork filtration also has a straightforward description:
\begin{cor}
For $n\geq1$ the space $K^u_{n,1}$ is the mapping cylinder of
$$
(\circ_i)_i\colon\coprod_{i=1}^{n+1}\!\!K_{n+1}\To K_{n},
$$  
where $\circ_i:K_{n+1}\to K_n$ are the associahedral degeneracy maps \eqref{degeneracy}.
\end{cor}

\begin{figure}[H]

\bigskip

$$
\begin{array}{cc}\\[-29pt]&\scalebox{.5} % Change this value to rescale the drawing.
{
\begin{pspicture}(0,-0.45)(0.18,0.46)
\psline[linewidth=0.02cm](0.08,0.36)(0.08,-0.44)  
%%%\psdots[dotsize=0.16](0.08,0.36)                  % CORCHO NEGRO
\psdots[dotsize=0.2,dotstyle=o](0.08,0.36)        % CORCHO BLANCO
\end{pspicture} 
}
\\[-3pt] K^{u}_{0,1}=&{\bullet}\end{array},\qquad
K^{u}_{1,1}=
\begin{array}{c}\\[-30pt]
\begin{array}{ccccc}
\scalebox{.3} % Change this value to rescale the drawing.
{
\begin{pspicture}(0,-0.87)(1.72,0.87)
\psline[linewidth=0.04cm](0.1,0.75)(0.9,-0.05)
\psline[linewidth=0.04cm](0.9,-0.05)(1.7,0.75)
\psline[linewidth=0.04cm](0.9,-0.05)(0.9,-0.85)
\psdots[dotsize=0.2](0.9,-0.05)
\psdots[dotsize=0.2,fillstyle=solid,dotstyle=o](0.1,0.75)
\end{pspicture} 
}

&\scalebox{.3} % Change this value to rescale the drawing.
{
\begin{pspicture}(0,-0.87)(1.72,0.87)
\psline[linewidth=0.04cm](0.1,0.75)(0.9,-0.05)
\psline[linewidth=0.04cm](0.9,-0.05)(1.7,0.75)
\psline[linewidth=0.04cm](0.9,-0.05)(0.9,-0.85)
\psdots[dotsize=0.2](0.9,-0.05)
\psdots[dotsize=0.2](0.1,0.75)
\end{pspicture} 
}
&
\scalebox{.4} % Change this value to rescale the drawing.
{
\begin{pspicture}(0,-0.42)(0.02,0.42)
\psline[linewidth=0.04cm](0.0,0.4)(0.0,-0.4)
\end{pspicture} 
}

&\scalebox{.3} % Change this value to rescale the drawing.
{
\begin{pspicture}(0,-0.87)(1.72,0.87)
\psline[linewidth=0.04cm](0.0,0.75)(0.8,-0.05)
\psline[linewidth=0.04cm](0.8,-0.05)(1.6,0.75)
\psline[linewidth=0.04cm](0.8,-0.05)(0.8,-0.85)
\psdots[dotsize=0.2](0.8,-0.05)
\psdots[dotsize=0.2](1.6,0.75)
\end{pspicture} 
}
&\scalebox{.3} % Change this value to rescale the drawing.
{
\begin{pspicture}(0,-0.87)(1.72,0.87)
\psline[linewidth=0.04cm](0.0,0.75)(0.8,-0.05)
\psline[linewidth=0.04cm](0.8,-0.05)(1.6,0.75)
\psline[linewidth=0.04cm](0.8,-0.05)(0.8,-0.85)
\psdots[dotsize=0.2](0.8,-0.05)
\psdots[dotsize=0.2,fillstyle=solid,dotstyle=o](1.6,0.75)
\end{pspicture} 
}

\end{array}
\\

\scalebox{1} % Change this value to rescale the drawing.
{
\begin{pspicture}(0,-0.1)(3.38,0.1)
\psdots[dotsize=0.16](0.08,0.0)
\psdots[dotsize=0.16](1.68,0.0)
\psdots[dotsize=0.16](3.28,0.0)
\psline[linewidth=0.02cm](0.08,0.0)(3.28,0.0)
\end{pspicture} 
}
\end{array},$$

$$
K^{u}_{2,1}=
\begin{array}{c}
\scalebox{.5} % Change this value to rescale the drawing.
{
\begin{pspicture}(0,-4.45)(9.874111,4.454111)
\definecolor{color13b}{rgb}{1.0,0.5372549019607843,0.5372549019607843}
\definecolor{color14b}{rgb}{0.5372549019607843,0.5372549019607843,1.0}
\definecolor{color15b}{rgb}{0.5372549019607843,1.0,0.5372549019607843}
\pspolygon[linewidth=0.0020,linecolor=white,fillstyle=solid,fillcolor=color13b](4.93,-0.09)(2.53,3.51)(7.33,3.51)
\pspolygon[linewidth=0.0020,linecolor=white,fillstyle=solid,fillcolor=color14b](4.93,-0.09)(0.53,-0.09)(2.53,-3.69)
\pspolygon[linewidth=0.0020,linecolor=white,fillstyle=solid,fillcolor=color15b](4.93,-0.09)(9.33,-0.09)(7.33,-3.69)
\psline[linewidth=0.1cm,linecolor=red](2.53,3.51)(7.33,3.51)
\psline[linewidth=0.1cm,linecolor=blue](0.53,-0.09)(2.53,-3.69)
\psline[linewidth=0.1cm,linecolor=green](9.33,-0.09)(7.33,-3.69)
\psline[linewidth=0.02cm](4.74,0.1)(4.94,-0.1)
\psline[linewidth=0.02cm](4.94,-0.1)(5.14,0.1)
\psline[linewidth=0.02cm](4.94,-0.1)(4.94,-0.3)
\psdots[dotsize=0.12](4.94,-0.1)
\psline[linewidth=0.02cm](1.98,-0.94)(2.58,-1.54)
\psline[linewidth=0.02cm](2.58,-1.54)(2.58,-0.94)
\psline[linewidth=0.02cm](2.58,-1.54)(3.18,-0.94)
\psline[linewidth=0.02cm](2.58,-1.54)(2.58,-2.04)
\psdots[dotsize=0.16](1.98,-0.94)
\psline[linewidth=0.02cm](4.28,2.76)(4.88,2.16)
\psline[linewidth=0.02cm](4.88,2.16)(4.88,2.76)
\psline[linewidth=0.02cm](4.88,2.16)(5.48,2.76)
\psline[linewidth=0.02cm](4.88,2.16)(4.88,1.66)
\psdots[dotsize=0.16](4.88,2.76)
\psline[linewidth=0.02cm](6.68,-0.94)(7.28,-1.54)
\psline[linewidth=0.02cm](7.28,-1.54)(7.28,-0.94)
\psline[linewidth=0.02cm](7.28,-1.54)(7.88,-0.94)
\psline[linewidth=0.02cm](7.28,-1.54)(7.28,-2.04)
\psdots[dotsize=0.16](7.88,-0.94)
\psdots[dotsize=0.16](7.28,-1.54)
\psdots[dotsize=0.16](2.58,-1.54)
\psdots[dotsize=0.16](4.88,2.16)
\psline[linewidth=0.02cm](4.88,4.06)(4.88,3.76)
\psline[linewidth=0.02cm](4.88,4.06)(5.18,4.36)
\psline[linewidth=0.02cm](4.88,4.06)(4.58,4.36)
\psline[linewidth=0.02cm](4.88,4.06)(4.88,4.36)
\psdots[dotsize=0.16,fillstyle=solid,dotstyle=o](4.88,4.36)
\psline[linewidth=0.02cm](8.98,-2.14)(8.98,-2.44)
\psline[linewidth=0.02cm](8.98,-2.14)(9.28,-1.84)
\psline[linewidth=0.02cm](8.98,-2.14)(8.68,-1.84)
\psline[linewidth=0.02cm](8.98,-2.14)(8.98,-1.84)
\psdots[dotsize=0.16,fillstyle=solid,dotstyle=o](9.28,-1.84)
\psline[linewidth=0.02cm](0.78,-2.24)(0.78,-2.54)
\psline[linewidth=0.02cm](0.78,-2.24)(1.08,-1.94)
\psline[linewidth=0.02cm](0.78,-2.24)(0.48,-1.94)
\psline[linewidth=0.02cm](0.78,-2.24)(0.78,-1.94)
\psdots[dotsize=0.16,fillstyle=solid,dotstyle=o](0.48,-1.94)
\psdots[dotsize=0.12](0.78,-2.24)
\psdots[dotsize=0.12](8.98,-2.14)
\psdots[dotsize=0.12](4.88,4.06)
\psline[linewidth=0.02cm](7.88,0.06)(7.88,0.36)
\psline[linewidth=0.02cm](7.88,0.36)(7.58,0.66)
\psline[linewidth=0.02cm](7.88,0.36)(8.18,0.66)
\psline[linewidth=0.02cm](7.58,0.66)(7.88,0.96)
\psline[linewidth=0.02cm](7.58,0.66)(7.28,0.96)
\psdots[dotsize=0.12](7.58,0.66)
\psdots[dotsize=0.12](7.88,0.36)
\psline[linewidth=0.02cm](2.78,1.76)(2.78,2.06)
\psline[linewidth=0.02cm](2.78,2.06)(2.48,2.36)
\psline[linewidth=0.02cm](2.78,2.06)(3.08,2.36)
\psline[linewidth=0.02cm](2.48,2.36)(2.78,2.66)
\psline[linewidth=0.02cm](2.48,2.36)(2.18,2.66)
\psdots[dotsize=0.12](2.48,2.36)
\psdots[dotsize=0.12](2.78,2.06)
\psline[linewidth=0.02cm](4.38,-3.04)(4.38,-2.74)
\psline[linewidth=0.02cm](4.38,-2.74)(4.08,-2.44)
\psline[linewidth=0.02cm](4.38,-2.74)(4.68,-2.44)
\psline[linewidth=0.02cm](4.08,-2.44)(4.38,-2.14)
\psline[linewidth=0.02cm](4.08,-2.44)(3.78,-2.14)
\psdots[dotsize=0.12](4.08,-2.44)
\psdots[dotsize=0.12](4.38,-2.74)
\psdots[dotsize=0.12](3.78,-2.14)
\psdots[dotsize=0.12](2.78,2.66)
\psdots[dotsize=0.12](8.18,0.66)
\psline[linewidth=0.02cm](7.08,1.76)(7.08,2.06)
\psline[linewidth=0.02cm](7.08,2.06)(7.38,2.36)
\psline[linewidth=0.02cm](7.38,2.36)(7.08,2.66)
\psline[linewidth=0.02cm](7.38,2.36)(7.68,2.66)
\psline[linewidth=0.02cm](7.08,2.06)(6.78,2.36)
\psdots[dotsize=0.12](7.08,2.06)
\psdots[dotsize=0.12](7.38,2.36)
\psline[linewidth=0.02cm](1.68,0.16)(1.68,0.46)
\psline[linewidth=0.02cm](1.68,0.46)(1.98,0.76)
\psline[linewidth=0.02cm](1.98,0.76)(1.68,1.06)
\psline[linewidth=0.02cm](1.98,0.76)(2.28,1.06)
\psline[linewidth=0.02cm](1.68,0.46)(1.38,0.76)
\psdots[dotsize=0.12](1.68,0.46)
\psdots[dotsize=0.12](1.98,0.76)
\psline[linewidth=0.02cm](5.48,-3.04)(5.48,-2.74)
\psline[linewidth=0.02cm](5.48,-2.74)(5.78,-2.44)
\psline[linewidth=0.02cm](5.78,-2.44)(5.48,-2.14)
\psline[linewidth=0.02cm](5.78,-2.44)(6.08,-2.14)
\psline[linewidth=0.02cm](5.48,-2.74)(5.18,-2.44)
\psdots[dotsize=0.12](5.48,-2.74)
\psdots[dotsize=0.12](5.78,-2.44)
\psdots[dotsize=0.12](1.38,0.76)
\psdots[dotsize=0.12](7.08,2.66)
\psdots[dotsize=0.12](6.08,-2.14)
\psline[linewidth=0.02cm](2.68,-4.44)(2.68,-4.24)
\psline[linewidth=0.02cm](2.68,-4.24)(2.88,-4.04)
\psline[linewidth=0.02cm](2.68,-4.24)(2.48,-4.04)
\psline[linewidth=0.02cm](2.48,-4.04)(2.68,-3.84)
\psline[linewidth=0.02cm](2.48,-4.04)(2.28,-3.84)
\psdots[dotsize=0.1](2.68,-4.24)
\psdots[dotsize=0.1](2.48,-4.04)
\psline[linewidth=0.02cm](9.58,-0.24)(9.58,-0.04)
\psline[linewidth=0.02cm](9.58,-0.04)(9.78,0.16)
\psline[linewidth=0.02cm](9.58,-0.04)(9.38,0.16)
\psline[linewidth=0.02cm](9.38,0.16)(9.58,0.36)
\psline[linewidth=0.02cm](9.38,0.16)(9.18,0.36)
\psdots[dotsize=0.1](9.58,-0.04)
\psdots[dotsize=0.1](9.38,0.16)
\psline[linewidth=0.02cm](2.28,3.46)(2.28,3.66)
\psline[linewidth=0.02cm](2.28,3.66)(2.48,3.86)
\psline[linewidth=0.02cm](2.28,3.66)(2.08,3.86)
\psline[linewidth=0.02cm](2.08,3.86)(2.28,4.06)
\psline[linewidth=0.02cm](2.08,3.86)(1.88,4.06)
\psdots[dotsize=0.1](2.28,3.66)
\psdots[dotsize=0.1](2.08,3.86)
\psline[linewidth=0.02cm](7.48,3.46)(7.48,3.66)
\psline[linewidth=0.02cm](7.48,3.66)(7.68,3.86)
\psline[linewidth=0.02cm](7.48,3.66)(7.28,3.86)
\psline[linewidth=0.02cm](7.68,3.86)(7.88,4.06)
\psline[linewidth=0.02cm](7.68,3.86)(7.48,4.06)
\psdots[dotsize=0.1](7.48,3.66)
\psdots[dotsize=0.1](7.68,3.86)
\psline[linewidth=0.02cm](7.08,-4.44)(7.08,-4.24)
\psline[linewidth=0.02cm](7.08,-4.24)(7.28,-4.04)
\psline[linewidth=0.02cm](7.08,-4.24)(6.88,-4.04)
\psline[linewidth=0.02cm](7.28,-4.04)(7.48,-3.84)
\psline[linewidth=0.02cm](7.28,-4.04)(7.08,-3.84)
\psdots[dotsize=0.1](7.08,-4.24)
\psdots[dotsize=0.1](7.28,-4.04)
\psline[linewidth=0.02cm](0.28,-0.24)(0.28,-0.04)
\psline[linewidth=0.02cm](0.28,-0.04)(0.48,0.16)
\psline[linewidth=0.02cm](0.28,-0.04)(0.08,0.16)
\psline[linewidth=0.02cm](0.48,0.16)(0.68,0.36)
\psline[linewidth=0.02cm](0.48,0.16)(0.28,0.36)
\psdots[dotsize=0.1](0.28,-0.04)
\psdots[dotsize=0.1](0.48,0.16)
\psdots[dotsize=0.16,fillstyle=solid,dotstyle=o](0.08,0.16)
\psdots[dotsize=0.16,fillstyle=solid,dotstyle=o](2.28,-3.84)
\psdots[dotsize=0.16,fillstyle=solid,dotstyle=o](7.48,-3.84)
\psdots[dotsize=0.16,fillstyle=solid,dotstyle=o](9.78,0.16)
\psdots[dotsize=0.16,fillstyle=solid,dotstyle=o](7.48,4.06)
\psdots[dotsize=0.16,fillstyle=solid,dotstyle=o](1.88,4.06)
\end{pspicture} 
}
\end{array},
K^{u}_{3,1}=\!\!\!\!\!
\begin{array}{c}
\includegraphics[scale=.4]{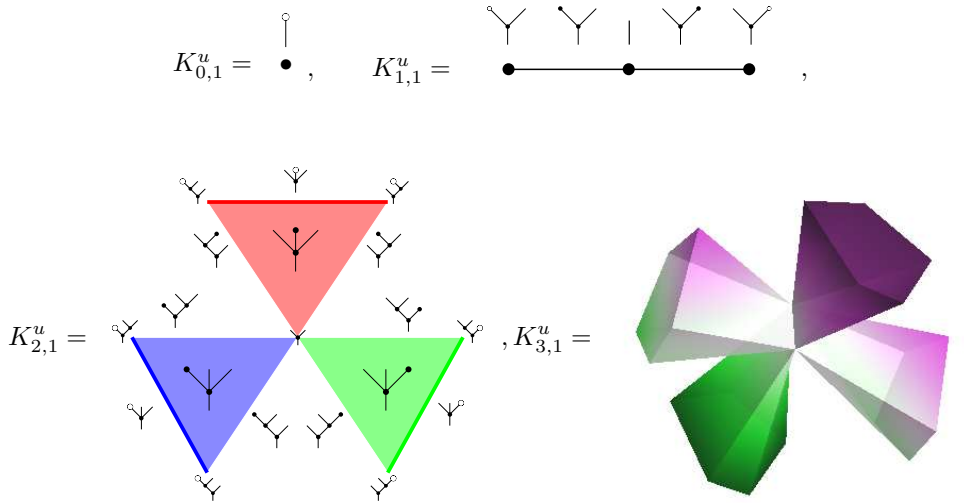}   %para B&W poner ku31grey.eps
\end{array}\!\!\!\!\!\!\!\!\!\!\!\!\!\!\!\!.
$$
\caption{The first four $K^{u}_{n,1}$.}
\label{firstfour}
\end{figure}

We now show that operadic composition maps respect the new cell structure. 

Consider the sets $\mathtt{C}(n)\subset\mathtt{T}(n)$ of trees of height 2 with $n$ leaves.  Then $\mathtt C$ has an operad structure
% tree operad $\mathtt{T}$ has a quotient 
%$\mathtt{C}$  
%consisting of trees of height 2. The quotient map contracts all inner edges containing no corks. The operad structure in $\mathtt C$ is 
given by grafting trees and then contracting the new inner edge.
There is another set operad $\mathtt P$, isomorphic to $\mathtt C$, where 
$$
\mathtt P(n)=\{ (S,n+m):\;m\geq 0,\;S\subseteq [n+m],\;|S|=m\}.
$$ 
A pair $(S,n+m)$ corresponds to the tree obtained from the corolla $C_{n+m}$ by replacing  $m$ leaf vertices by corks, at the positions indicated by the subset $S$. 
Thus the operad structure of $\mathtt P$ has unit $(\varnothing,1)$ and 
if $S_1=\{j_1<\dots<j_s\}\subset\ord{p+s}$,
 $S_2=\{k_1<\dots<k_t\}\subset\ord{q+t}$,  
and $1\leq i\leq p$ then
$$
(S_1,p+s) \circ_i (S_2,q+t)=(S_1\circ_i S_2,p+s+q+t-1).
$$
Here the set $S_1\circ_i S_2$ is
$$
\{j_1,\dots,j_{r-1},\:k_1+i+r-2,\dots,k_t+i+r-2,\:j_r+q+t-1,\dots,j_s+q+t-1\}$$
if the $i^{\text{th}}$ element of the complement of $S_1$ lies between $j_{r-1}$ 
and $j_r$.

\begin{prop}\label{filtri}
The composition structure given in Proposition~\ref{circi} respects the cork filtration 
$$
K^u_{p,s}\circ_i K^u_{q,t}\subset K^u_{p+q-1,s+t}.
$$
Moreover, given
$S_1\subset\ord{p+s}$,
$S_2\subset\ord{q+t}$, and $r$ as in the previous paragraph,
%, % 
the following diagram commutes for $(q,t)\neq (0,1)$
\begin{equation}\label{filtr1}
\xymatrix{
K_{p+s}\times [0,1]^s\times K_{q+t}\times [0,1]^t
\ar[rr]^-{\bar f_{S_1}\times \bar f_{S_2}}
\ar[d]_{\id{}\!\times \text{swap}\times\id{}}&&
K^u_{p,s}\times K^u_{q,t}\ar[dd]^{\circ_i}\\
\ar[d]_{\circ_{i+r-1}\times\sigma_{
%s,t,
r}}
K_{p+s}\times 
K_{q+t}\times [0,1]^s\times [0,1]^t
\\
K_{s+t+p+q-1}\times [0,1]^{s+t}
\ar[rr]^-{\bar f_{S_1\circ_{i}S_{2}}}&& K^u_{p+q-1,s+t}
}\end{equation}
and in the case $(q,t)=(0,1)$ then
\begin{equation}\label{filtr2}
\xymatrix{
K_{p+s}\times [0,1]^s
\ar[rr]^-{(\bar f_{S_1},\bullet)}
\ar[d]_{\id{}\!\times\partial^+_{r}}
%\ar[d]_{\id{}\times \text{swap}\times\id{}}
&&
K^u_{p,s}\times K^u_{0,1}\ar[d]^{\circ_i}\\
%K_{p+s}\times 
%K_{q+t}\times [0,1]^s\times [0,1]^t
%\\
K_{p+s}\times [0,1]^{s+1}
\ar[rr]^-{\bar f_{S_1\circ_{i}\{1\}}
}&& K^u_{p-1,s+1}
}\end{equation}
%where the subset $S_1\circ_{i}S_{2}\subset[s+t+p+q-1]$ is 
\end{prop}
\noindent This result is easily checked by inspection.

We may describe the new cellular structure of unital associahedra in terms of \emph{trees with black and white corks}, as we have seen depicted in Figures~\ref{k21} and \ref{firstfour} above.

Let $T$ be a tree of height $h$, with $n$ leaves, and $m$ corks, of which $m_b$ are black and $m_w$ are white. We exclude trees with any degree two vertices, and also the case $h=m_b=1$. 
\begin{enumerate}\item If $h=2$ and $m_w=0$ then we have a cell
$$
K^u_T=\bar f_S(K_{n+m}\times [0,1]^m)
$$
where $S\subseteq\ord{n+m}$ indicates which level 2 vertices are corks.
\item If $h=1$ and $m_w=1$, i.e.\ $T=\raisebox{-0.4ex}[0ex][-0.4ex]{\begin{sideways}$\multimap$\end{sideways}}$,
$$
K^u_T=K^u_{0,1}=\{\bullet\}.
$$ 
\item If $T=T'\circ_i T''$ we define inductively
$$
K^u_T=K^u_{T'}\circ_i K^u_{T''}.
$$ 
\end{enumerate}

We may give a non-inductive definition as follows. Consider the tree $T'''$ obtained removing all corks from $T$, replacing them by leaves. 
Then the characteristic map of $K^u_T$ is
$$\xymatrix{
K_{T'''}\times [0,1]^{m_b}\ar[rr]^-{\text{incl.}\times\partial^+_w}&&
K_{n+m}\times [0,1]^{m}\ar[r]^-{\bar f_S} &K_{n,m}^u\subset K_n^u%\subset \mathtt uA_\infty
}$$
where $S$ indicates which leaves of $T'''$ are corks in $T$ and 
$$
\partial^+_w=\partial^+_{j_{m_w}}\cdots\partial^+_{j_2}\partial^+_{j_1}
$$
where $\{j_1<j_2<\dots <j_{m_w}\}\subset[m]$ indicates the positions of the white corks in the set of corks.

%If $T$ has no corks then $K^u_T$ is a cell $K_T$ of the classical associahedron. If $T$ has one black and no white corks then $K_T^u$ is ...

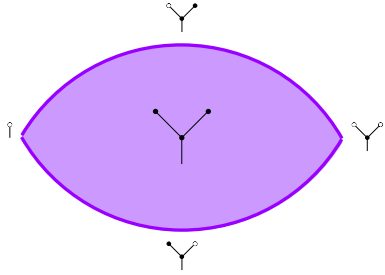
\begin{figure}[H]
$$\scalebox{.44} % Change this value to rescale the drawing.
{
\begin{pspicture}(0,-4.01)(11.854953,4.08)
\definecolor{color181}{rgb}{0.8,0.6,1.0}
\definecolor{color505}{rgb}{0.6,0.0,1.0}
\rput{88.798164}(8.510092,-3.0917294){\psarc[linewidth=0.04,linecolor=color181,fillstyle=solid,fillcolor=color181](5.8336816,2.799384){5.6}{120.25644}{239.74356}
\psline[linewidth=0.04,linecolor=color181](3.0120034,7.6365457)(3.0120034,-2.0377777)}
\rput{268.79816}(8.634888,2.8568678){\psarc[linewidth=0.04,linecolor=color181,fillstyle=solid,fillcolor=color181](5.7162247,-2.799384){5.6}{120.25644}{239.74356}
\psline[linewidth=0.04,linecolor=color181](2.8945467,2.0377777)(2.8945467,-7.6365457)}
\psline[linewidth=0.1cm,linecolor=color181](0.976009,0.10067715)(10.573897,-0.10067715)
\rput{89.613464}(8.554754,-2.9940288){\psarc[linewidth=0.1,linecolor=color505](5.784525,2.809317){5.6}{120.25644}{239.74356}}
\rput{269.42117}(8.595214,2.9277062){\psarc[linewidth=0.1,linecolor=color505](5.7467456,-2.7905555){5.6}{120.25644}{239.74356}}
\psline[linewidth=0.02cm](5.774953,-4.0)(5.774953,-3.6)
\psline[linewidth=0.02cm](5.774953,-3.6)(6.174953,-3.2)
\psline[linewidth=0.02cm](5.774953,-3.6)(5.3749533,-3.2)
\psline[linewidth=0.02cm](0.5749531,0.0)(0.5749531,0.4)
\psline[linewidth=0.02cm](5.774953,3.2)(5.774953,3.6)
\psline[linewidth=0.02cm](5.774953,3.6)(6.174953,4.0)
\psline[linewidth=0.02cm](5.774953,3.6)(5.3749533,4.0)
\psline[linewidth=0.02cm](11.374953,-0.4)(11.374953,0.0)
\psline[linewidth=0.02cm](11.374953,0.0)(11.774953,0.4)
\psline[linewidth=0.02cm](11.374953,0.0)(10.974953,0.4)
\psline[linewidth=0.02cm](5.769978,-0.7939801)(5.769978,0.0030099489)
\psline[linewidth=0.02cm](5.769978,0.0030099489)(6.565003,0.8)
\psline[linewidth=0.02cm](5.769978,0.0030099489)(4.974953,0.8)
\psdots[dotsize=0.16](5.774953,0.0)
\psdots[dotsize=0.16](4.974953,0.8)
\psdots[dotsize=0.16](6.574953,0.8)
\psdots[dotsize=0.14,dotstyle=o](0.5749531,0.4) % CORCHO BLANCO
%%%%%%%% \psdots[dotsize=0.14](0.5749531,0.4)  %%CORCHO NEGRO
\psdots[dotsize=0.14](5.774953,3.6)
\psdots[dotsize=0.14](11.374953,0.0)
\psdots[dotsize=0.14](5.774953,-3.6)
\psdots[dotsize=0.14](6.174953,4.0)
\psdots[dotsize=0.14](5.3749533,-3.2)
\psdots[dotsize=0.14,fillstyle=solid,dotstyle=o](10.974953,0.4)
\psdots[dotsize=0.14,fillstyle=solid,dotstyle=o](11.774953,0.4)
\psdots[dotsize=0.14,fillstyle=solid,dotstyle=o](5.3749533,4.0)
\psdots[dotsize=0.14,fillstyle=solid,dotstyle=o](6.174953,-3.2)
\end{pspicture} 
}
$$
\caption{$K^{u}_{0,2}$ with the new cellular structure. It has two vertices, two edges, and one $2$-cell.}\label{discochino}
\end{figure}

\begin{rem}
The CW-complex $K^{u}_{1,2}$ is a cellular decomposition of three $3$-balls such that the first and  second $3$-balls share a common edge, as well as the second and third $3$-balls. These two edges have a common vertex which is the intersection of the first and  third $3$-balls, as in the following picture
$$\includegraphics[scale=.4]{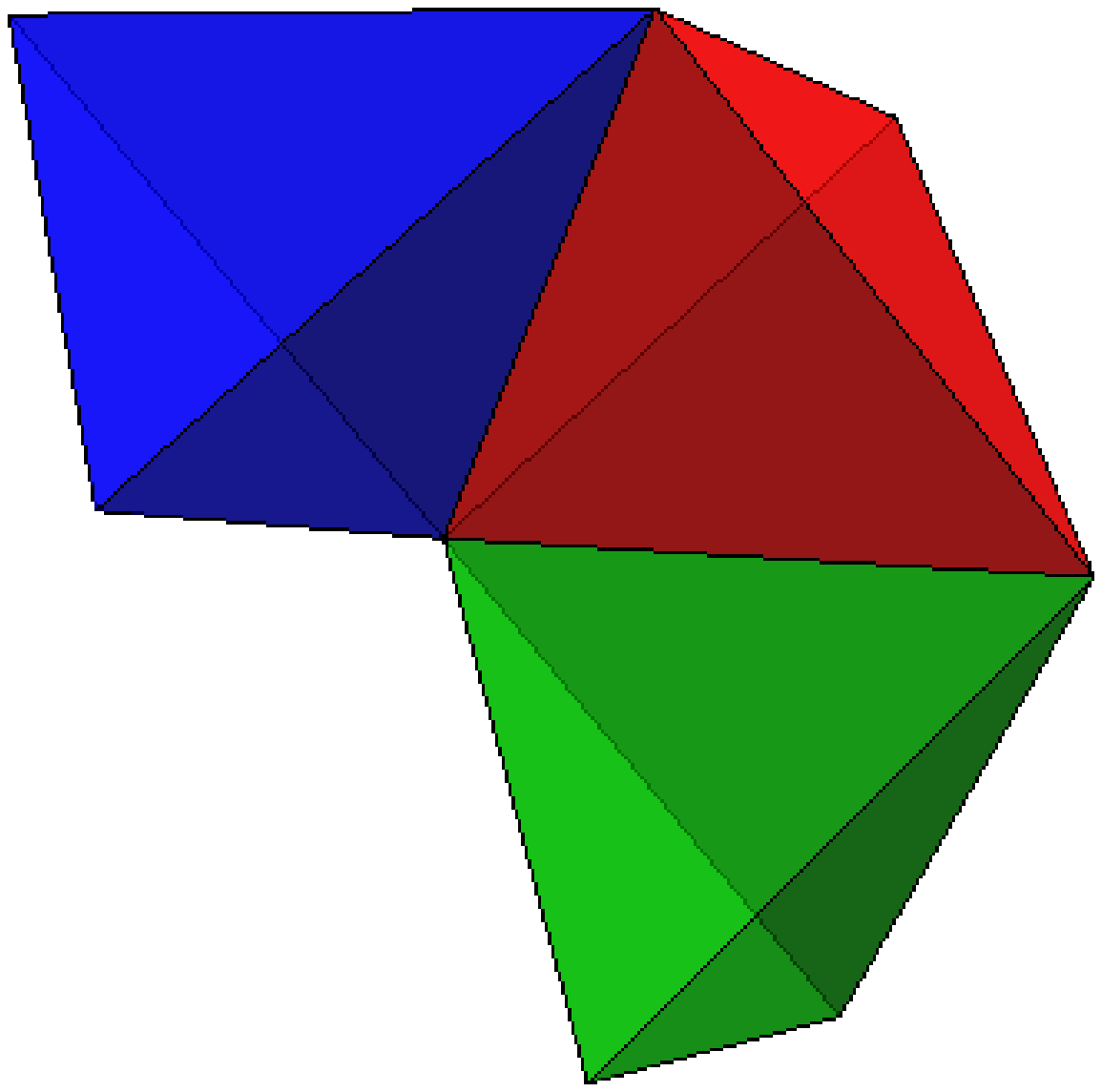}$$ %para B&W usar 3tetrahedra2grey.eps
The common edges and common vertex form the subcomplex $K^u_{1,1}\into K^u_{1,2}$, see Figure \ref{firstfour}.
For the sake of simplicity, we have drawn the $3$-balls as tetrahedra. 
Each of these $3$-balls in  $K^{u}_{1,2}$ is filled in with a $3$-cell, which correspond to the following trees, respectively:
$$\scalebox{.6} % Change this value to rescale the drawing.
{
\begin{pspicture}(0,-0.87)(1.72,0.87)
\psline[linewidth=0.04cm](0.9,-0.05)(0.9,-0.85)
\psline[linewidth=0.04cm](0.9,-0.05)(1.7,0.75)
\psline[linewidth=0.04cm](0.9,-0.05)(0.1,0.75)
\psline[linewidth=0.04cm](0.9,-0.05)(0.9,0.75)
\psdots[dotsize=0.2](0.1,0.75)
\psdots[dotsize=0.2](0.9,0.75)
\psdots[dotsize=0.2](0.9,-0.05)
\end{pspicture} 
}\qquad 
\scalebox{.6} % Change this value to rescale the drawing.
{
\begin{pspicture}(0,-0.87)(1.82,0.87)
\psline[linewidth=0.04cm](0.9,-0.05)(0.9,-0.85)
\psline[linewidth=0.04cm](0.9,-0.05)(1.7,0.75)
\psline[linewidth=0.04cm](0.9,-0.05)(0.1,0.75)
\psline[linewidth=0.04cm](0.9,-0.05)(0.9,0.75)
\psdots[dotsize=0.2](0.1,0.75)
\psdots[dotsize=0.2](1.7,0.75)
\psdots[dotsize=0.2](0.9,-0.05)
\end{pspicture} 
}
\qquad
\scalebox{.6} % Change this value to rescale the drawing.
{
\begin{pspicture}(0,-0.87)(1.72,0.87)
\psline[linewidth=0.04cm](0.8,-0.05)(0.8,-0.85)
\psline[linewidth=0.04cm](0.8,-0.05)(1.6,0.75)
\psline[linewidth=0.04cm](0.8,-0.05)(0.0,0.75)
\psline[linewidth=0.04cm](0.8,-0.05)(0.8,0.75)
\psdots[dotsize=0.2](0.8,0.75)
\psdots[dotsize=0.2](1.6,0.75)
\psdots[dotsize=0.2](0.8,-0.05)
\end{pspicture} 
}.$$
Let us examine more explicitly all the cells of $K^u_{1,2}$, corresponding to trees with one leaf and at most two black or white corks.

The cellular structure of the first $3$-ball is
$$\includegraphics[scale=.4]{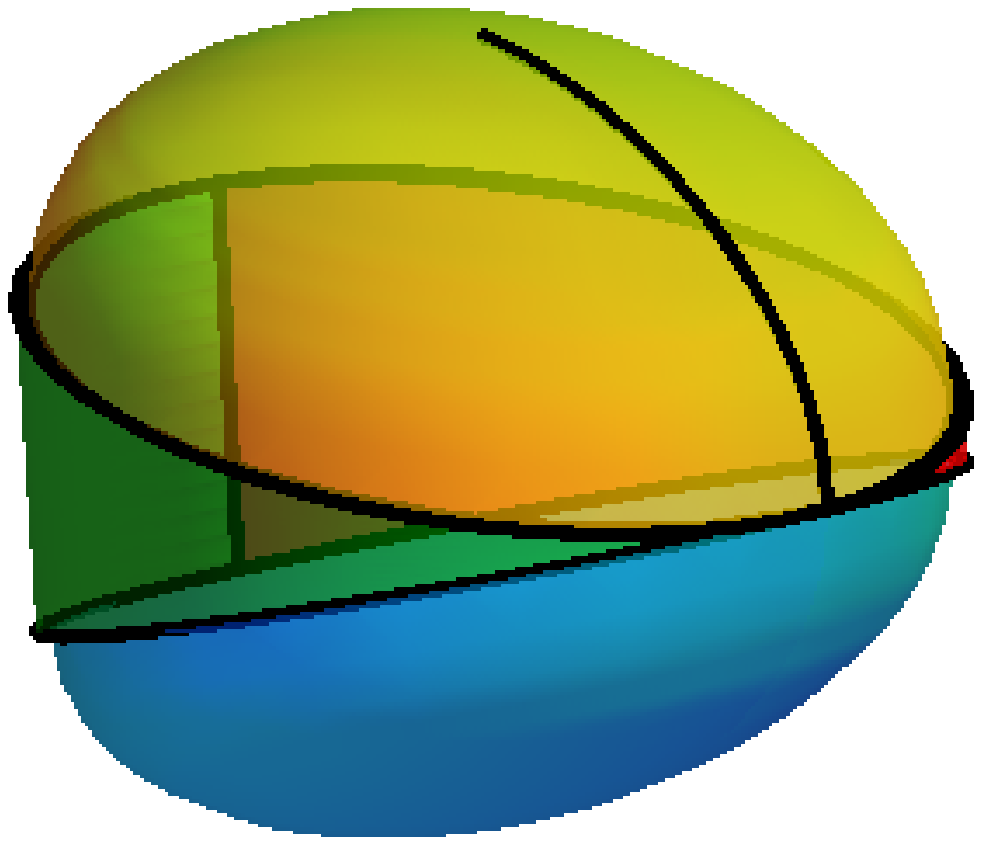}\vspace{-35pt}$$ %para B&W usar ku12grey.eps
\smallskip\par\noindent
It has four vertices, six edges,  four $2$-cells, and one $3$-cell. 
The four vertices correspond to the following trees
\smallskip

$$\scalebox{.6} % Change this value to rescale the drawing.
{
\begin{pspicture}(0,-1.27)(9.62,1.27)
\psline[linewidth=0.04cm](1.6,0.35)(2.4,-0.45)
\psline[linewidth=0.04cm](2.4,-0.45)(3.2,0.35)
\psline[linewidth=0.04cm](2.4,-0.45)(2.4,-1.25)
\psline[linewidth=0.04cm](4.0,1.15)(4.8,0.35)
\psline[linewidth=0.04cm](4.8,0.35)(5.6,-0.45)
\psline[linewidth=0.04cm](4.8,0.35)(5.6,1.15)
\psline[linewidth=0.04cm](5.6,-0.45)(6.4,0.35)
\psline[linewidth=0.04cm](5.6,-0.45)(5.6,-1.25)
\psline[linewidth=0.04cm](8.0,1.15)(8.8,0.35)
\psline[linewidth=0.04cm](7.2,0.35)(8.0,-0.45)
\psline[linewidth=0.04cm](8.8,0.35)(9.6,1.15)
\psline[linewidth=0.04cm](8.0,-0.45)(8.8,0.35)
\psline[linewidth=0.04cm](8.0,-0.45)(8.0,-1.25)
\psdots[dotsize=0.2](2.4,-0.45)
\psdots[dotsize=0.2](4.8,0.35)
\psdots[dotsize=0.2](5.6,-0.45)
\psdots[dotsize=0.2](8.0,-0.45)
\psdots[dotsize=0.2](8.8,0.35)
\psdots[dotsize=0.2,fillstyle=solid,dotstyle=o](1.6,0.35)
\psdots[dotsize=0.2,fillstyle=solid,dotstyle=o](4.0,1.15)
\psdots[dotsize=0.2,fillstyle=solid,dotstyle=o](8.0,1.15)
\psdots[dotsize=0.2,fillstyle=solid,dotstyle=o](5.6,1.15)
\psdots[dotsize=0.2,fillstyle=solid,dotstyle=o](7.2,0.35)
\psline[linewidth=0.04cm](0.0,0.35)(0.0,-1.25)
\end{pspicture} 
}.$$
\smallskip\par\noindent
The top $2$-cell, the two middle triangles, and the bottom disk, together with their six edges, can be represented as follows,  
\smallskip\par\noindent
$$
\scalebox{0.48} % Change this value to rescale the drawing.
{
\begin{pspicture}(0,-3.6058578)(24.336386,3.6058578)
\definecolor{color1762}{rgb}{0.0,0.0,0.6}
\psline[linewidth=0.04cm](12.791716,1.2141422)(12.791716,0.41414216)
\psline[linewidth=0.04cm](12.791716,1.2141422)(13.591716,2.014142)
\psline[linewidth=0.04cm](12.791716,1.2141422)(11.991715,2.014142)
\psline[linewidth=0.04cm](12.791716,1.2141422)(12.791716,2.014142)
\psdots[dotsize=0.2,fillstyle=solid,dotstyle=o](12.791716,2.014142)
\psdots[dotsize=0.2](11.991715,2.014142)
\psdots[dotsize=0.2](12.791716,1.2141422)
\psline[linewidth=0.04cm](12.791716,-1.5858579)(12.791716,-2.3858578)
\psline[linewidth=0.04cm](12.791716,-1.5858579)(13.591716,-0.78585786)
\psline[linewidth=0.04cm](12.791716,-1.5858579)(11.991715,-0.78585786)
\psline[linewidth=0.04cm](12.791716,-1.5858579)(12.791716,-0.78585786)
\psdots[dotsize=0.2,fillstyle=solid,dotstyle=o](11.991715,-0.78585786)
\psdots[dotsize=0.2](12.791716,-0.78585786)
\psdots[dotsize=0.2](12.791716,-1.5858579)
\psline[linewidth=0.04cm,linecolor=color1762](14.391716,3.214142)(14.391716,3.214142)
\psline[linewidth=0.04cm,linecolor=color1762](14.391716,3.214142)(15.191716,2.4141421)
\psline[linewidth=0.04cm,linecolor=color1762](14.791716,2.8141422)(15.191716,3.214142)
\psline[linewidth=0.04cm,linecolor=color1762](15.191716,2.4141421)(15.591716,2.8141422)
\psline[linewidth=0.04cm,linecolor=color1762](15.191716,2.4141421)(15.191716,2.014142)
\psdots[dotsize=0.2,linecolor=color1762](14.791716,2.8141422)
\psdots[dotsize=0.2,linecolor=color1762](15.191716,2.4141421)
\psdots[dotsize=0.2,linecolor=color1762](14.391716,3.214142)
\psdots[dotsize=0.2,linecolor=color1762,fillstyle=solid,dotstyle=o](15.191716,3.214142)
\psline[linewidth=0.04cm,linecolor=color1762](10.391716,3.214142)(10.791716,2.8141422)
\psline[linewidth=0.04cm,linecolor=color1762](9.991715,2.8141422)(10.391716,2.4141421)
\psline[linewidth=0.04cm,linecolor=color1762](10.391716,2.4141421)(11.191716,3.214142)
\psline[linewidth=0.04cm,linecolor=color1762](10.391716,2.4141421)(10.391716,2.014142)
\psdots[dotsize=0.2,linecolor=color1762](10.391716,2.4141421)
\psdots[dotsize=0.2,linecolor=color1762](10.791716,2.8141422)
\psdots[dotsize=0.2,linecolor=color1762](9.991715,2.8141422)
\psdots[dotsize=0.2,linecolor=color1762,fillstyle=solid,dotstyle=o](10.391716,3.214142)
\psline[linewidth=0.04cm,linecolor=color1762](10.391716,-2.3858578)(10.791716,-2.785858)
\psline[linewidth=0.04cm,linecolor=color1762](9.991715,-2.785858)(10.391716,-3.1858578)
\psline[linewidth=0.04cm,linecolor=color1762](10.391716,-3.1858578)(11.191716,-2.3858578)
\psline[linewidth=0.04cm,linecolor=color1762](10.391716,-3.1858578)(10.391716,-3.5858579)
\psdots[dotsize=0.2,linecolor=color1762](10.391716,-3.1858578)
\psdots[dotsize=0.2,linecolor=color1762](10.791716,-2.785858)
\psdots[dotsize=0.2,linecolor=color1762](10.391716,-2.3858578)
\psdots[dotsize=0.2,linecolor=color1762,fillstyle=solid,dotstyle=o](9.991715,-2.785858)
\psline[linewidth=0.04cm,linecolor=color1762](14.391716,-2.3858578)(15.191716,-3.1858578)
\psline[linewidth=0.04cm,linecolor=color1762](14.791716,-2.785858)(15.191716,-2.3858578)
\psline[linewidth=0.04cm,linecolor=color1762](15.191716,-3.1858578)(15.591716,-2.785858)
\psline[linewidth=0.04cm,linecolor=color1762](15.191716,-3.1858578)(15.191716,-3.5858579)
\psdots[dotsize=0.2,linecolor=color1762](14.791716,-2.785858)
\psdots[dotsize=0.2,linecolor=color1762](15.191716,-3.1858578)
\psdots[dotsize=0.2,linecolor=color1762](15.191716,-2.3858578)
\psdots[dotsize=0.2,linecolor=color1762,fillstyle=solid,dotstyle=o](14.391716,-2.3858578)
\psdiamond[linewidth=0.04,dimen=outer](12.791716,0.014142151)(3.6,3.6)
\psline[linewidth=0.04cm](9.231716,0.014142151)(16.351715,0.014142151)
\psline[linewidth=0.04cm,linecolor=color1762](13.591716,0.6141422)(13.991715,0.21414214)
\psline[linewidth=0.04cm,linecolor=color1762](13.991715,0.21414214)(13.991715,0.6141422)
\psline[linewidth=0.04cm,linecolor=color1762](13.991715,0.21414214)(14.391716,0.6141422)
\psline[linewidth=0.04cm,linecolor=color1762](13.991715,0.21414214)(13.991715,-0.18585785)
\psdots[dotsize=0.2,linecolor=color1762](13.991715,0.21414214)
\psdots[dotsize=0.2,linecolor=color1762,fillstyle=solid,dotstyle=o](13.591716,0.6141422)
\psdots[dotsize=0.2,linecolor=color1762,fillstyle=solid,dotstyle=o](13.991715,0.6141422)
\psline[linewidth=0.04cm](21.591715,-0.78585786)(21.591715,-1.5858579)
\psline[linewidth=0.04cm](21.591715,-0.78585786)(20.791716,0.014142151)
\psline[linewidth=0.04cm](21.591715,-0.78585786)(22.391716,0.014142151)
\psdots[dotsize=0.2](21.591715,-0.78585786)
\psline[linewidth=0.04cm](20.791716,0.014142151)(21.591715,0.81414217)
\psline[linewidth=0.04cm](20.791716,0.014142151)(19.991716,0.81414217)
\psdots[dotsize=0.2](20.791716,0.014142151)
\psdots[dotsize=0.2](19.991716,0.81414217)
\psdots[dotsize=0.2](21.591715,0.81414217)
\psarc[linewidth=0.04](21.319939,-1.77408){3.46}{30.0}{150.0}
\psarc[linewidth=0.04](21.319939,1.70592){3.46}{210.0}{330.0}
\psline[linewidth=0.04cm,linecolor=color1762](20.791716,-1.9858578)(21.591715,-2.785858)
\psline[linewidth=0.04cm,linecolor=color1762](21.191715,-2.3858578)(21.591715,-1.9858578)
\psline[linewidth=0.04cm,linecolor=color1762](21.591715,-2.785858)(21.991716,-2.3858578)
\psline[linewidth=0.04cm,linecolor=color1762](21.591715,-2.785858)(21.591715,-3.1858578)
\psdots[dotsize=0.2,linecolor=color1762](21.191715,-2.3858578)
\psdots[dotsize=0.2,linecolor=color1762](21.591715,-2.785858)
\psdots[dotsize=0.2,linecolor=color1762](21.591715,-1.9858578)
\psdots[dotsize=0.2,linecolor=color1762,fillstyle=solid,dotstyle=o](20.791716,-1.9858578)
\psline[linewidth=0.04cm,linecolor=color1762](20.791716,3.214142)(20.791716,3.214142)
\psline[linewidth=0.04cm,linecolor=color1762](20.791716,3.214142)(21.591715,2.4141421)
\psline[linewidth=0.04cm,linecolor=color1762](21.191715,2.8141422)(21.591715,3.214142)
\psline[linewidth=0.04cm,linecolor=color1762](21.591715,2.4141421)(21.991716,2.8141422)
\psline[linewidth=0.04cm,linecolor=color1762](21.591715,2.4141421)(21.591715,2.014142)
\psdots[dotsize=0.2,linecolor=color1762](21.191715,2.8141422)
\psdots[dotsize=0.2,linecolor=color1762](21.591715,2.4141421)
\psdots[dotsize=0.2,linecolor=color1762](20.791716,3.214142)
\psdots[dotsize=0.2,linecolor=color1762,fillstyle=solid,dotstyle=o](21.591715,3.214142)
\psline[linewidth=0.04cm](3.1917157,-0.38585785)(3.1917157,-1.1858579)
\psline[linewidth=0.04cm](3.1917157,-0.38585785)(3.9917157,0.41414216)
\psline[linewidth=0.04cm](3.1917157,-0.38585785)(2.3917158,0.41414216)
\psdots[dotsize=0.2](2.3917158,0.41414216)
\psdots[dotsize=0.2](3.1917157,-0.38585785)
\psline[linewidth=0.04cm](3.9917157,0.41414216)(3.1917157,1.2141422)
\psline[linewidth=0.04cm](3.9917157,0.41414216)(4.7917156,1.2141422)
\psdots[dotsize=0.2](3.1917157,1.2141422)
\psdots[dotsize=0.2](3.9917157,0.41414216)
\psline[linewidth=0.04cm,linecolor=color1762](1.1917157,2.8141422)(1.5917157,2.4141421)
\psline[linewidth=0.04cm,linecolor=color1762](1.5917157,2.4141421)(1.9917158,2.8141422)
\psline[linewidth=0.04cm,linecolor=color1762](1.5917157,2.4141421)(1.5917157,2.014142)
\psdots[dotsize=0.2,linecolor=color1762](1.5917157,2.4141421)
\psdots[dotsize=0.2,linecolor=color1762](1.1917157,2.8141422)
\psline[linewidth=0.04cm,linecolor=color1762](1.1917157,-2.3858578)(1.5917157,-2.785858)
\psline[linewidth=0.04cm,linecolor=color1762](1.5917157,-2.785858)(1.9917158,-2.3858578)
\psline[linewidth=0.04cm,linecolor=color1762](1.5917157,-2.785858)(1.5917157,-3.1858578)
\psdots[dotsize=0.2,linecolor=color1762](1.5917157,-2.785858)
\psdots[dotsize=0.2,linecolor=color1762](1.1917157,-2.3858578)
\psline[linewidth=0.04cm,linecolor=color1762](5.591716,3.214142)(5.991716,2.8141422)
\psline[linewidth=0.04cm,linecolor=color1762](5.1917157,2.8141422)(5.591716,2.4141421)
\psline[linewidth=0.04cm,linecolor=color1762](5.591716,2.4141421)(6.3917155,3.214142)
\psline[linewidth=0.04cm,linecolor=color1762](5.591716,2.4141421)(5.591716,2.014142)
\psdots[dotsize=0.2,linecolor=color1762](5.591716,2.4141421)
\psdots[dotsize=0.2,linecolor=color1762](5.991716,2.8141422)
\psdots[dotsize=0.2,linecolor=color1762](5.1917157,2.8141422)
\psdots[dotsize=0.2,linecolor=color1762,fillstyle=solid,dotstyle=o](5.591716,3.214142)
\psdiamond[linewidth=0.04,dimen=outer](3.5917158,0.014142151)(3.6,3.6)
\psline[linewidth=0.04cm,linecolor=color1762](5.591716,-2.3858578)(5.991716,-2.785858)
\psline[linewidth=0.04cm,linecolor=color1762](5.1917157,-2.785858)(5.591716,-3.1858578)
\psline[linewidth=0.04cm,linecolor=color1762](5.591716,-3.1858578)(6.3917155,-2.3858578)
\psline[linewidth=0.04cm,linecolor=color1762](5.591716,-3.1858578)(5.591716,-3.5858579)
\psdots[dotsize=0.2,linecolor=color1762](5.591716,-3.1858578)
\psdots[dotsize=0.2,linecolor=color1762](5.991716,-2.785858)
\psdots[dotsize=0.2,linecolor=color1762](5.591716,-2.3858578)
\psdots[dotsize=0.2,linecolor=color1762,fillstyle=solid,dotstyle=o](5.1917157,-2.785858)
\end{pspicture} 
}
$$
\smallskip\par\noindent
Observe that the top $2$-cell is doubly-incident with one of its edges, which contains the vertex corresponding to the trivial tree $\|T\|=|$.
It is also along this edge that the first and second $3$-balls are identified.

The cellular structure of the third $3$-ball is the same, and its cells correspond to symmetric versions of the trees above.

The second $3$-ball looks different. We give two different perspectives:\vspace{-5pt}
$$\includegraphics[scale=.4]{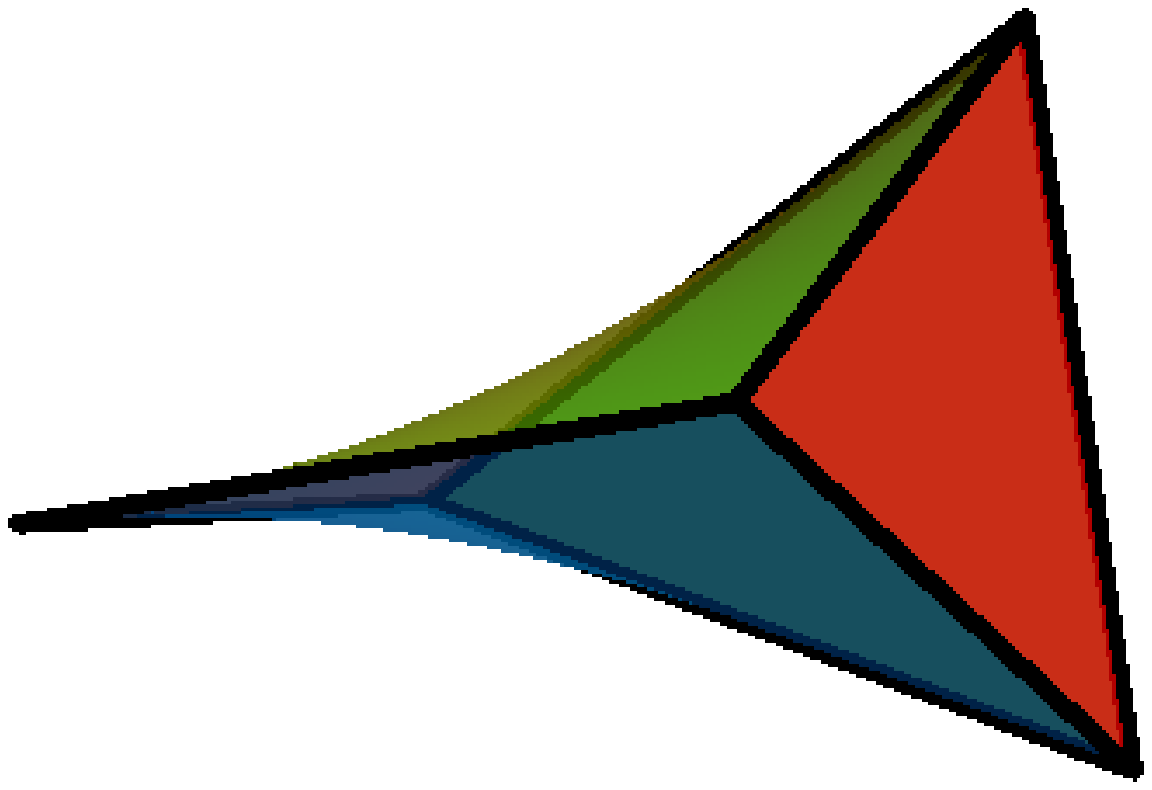}\includegraphics[scale=.3]{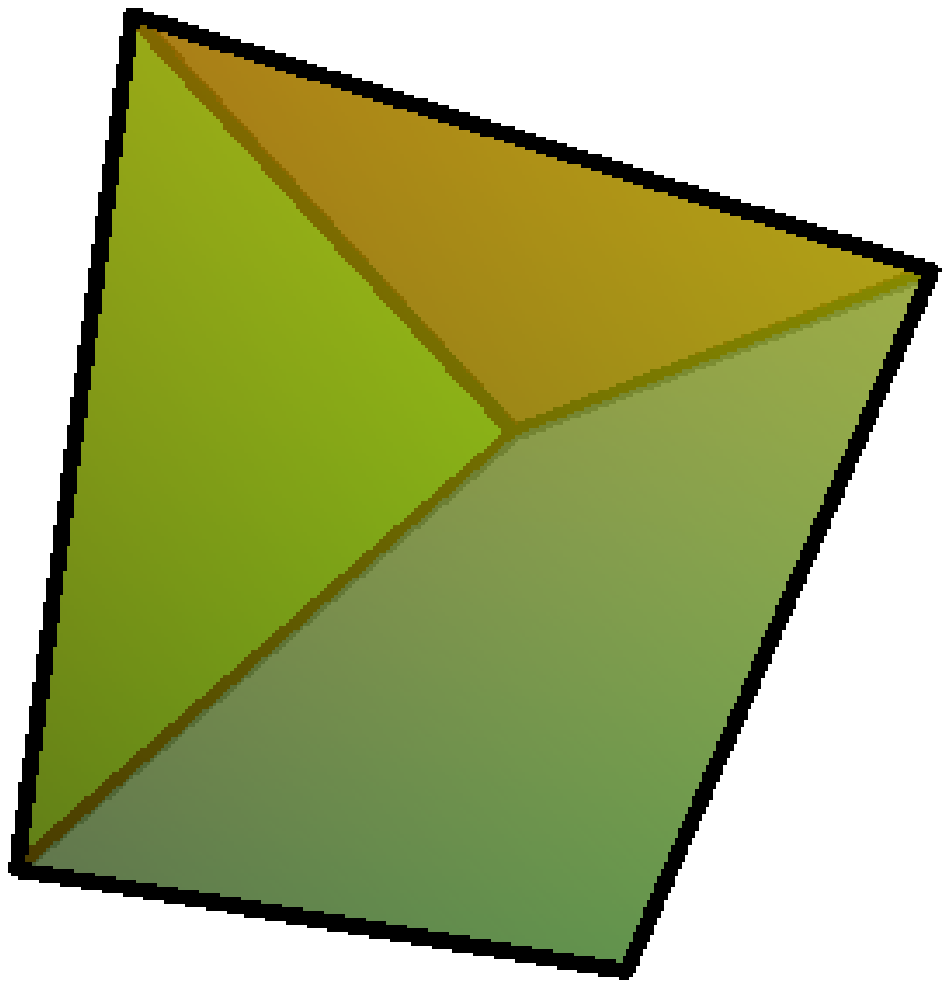}\vspace{-5pt}$$ %para B&W usar ku122grey.eps y ku1222grey.eps
\smallskip\par\noindent
It has five vertices, seven edges,  four $2$-cells, and one $3$-cell.  
The vertices correspond to the following trees
\smallskip\par\noindent
$$
\scalebox{.6} % Change this value to rescale the drawing.
{
\begin{pspicture}(0,-1.27)(12.12,1.27)
\psline[linewidth=0.04cm](9.6,0.35)(10.4,-0.45)
\psline[linewidth=0.04cm](10.4,-0.45)(11.2,0.35)
\psline[linewidth=0.04cm](10.4,-0.45)(10.4,-1.25)
\psdots[dotsize=0.2](10.4,-0.45)
\psdots[dotsize=0.2](11.2,0.35)
\psline[linewidth=0.04cm](11.2,0.35)(10.4,1.15)
\psline[linewidth=0.04cm](7.2,0.35)(8.0,-0.45)
\psline[linewidth=0.04cm](8.0,-0.45)(8.8,0.35)
\psline[linewidth=0.04cm](8.0,-0.45)(8.0,-1.25)
\psdots[dotsize=0.2](8.0,-0.45)
\psdots[dotsize=0.2](7.2,0.35)
\psline[linewidth=0.04cm](7.2,0.35)(8.0,1.15)
\psline[linewidth=0.04cm](1.6,0.35)(2.4,-0.45)
\psline[linewidth=0.04cm](2.4,-0.45)(3.2,0.35)
\psline[linewidth=0.04cm](2.4,-0.45)(2.4,-1.25)
\psdots[dotsize=0.2](2.4,-0.45)
\psline[linewidth=0.04cm](4.0,0.35)(4.8,-0.45)
\psline[linewidth=0.04cm](4.8,-0.45)(5.6,0.35)
\psline[linewidth=0.04cm](4.8,-0.45)(4.8,-1.25)
\psdots[dotsize=0.2](4.8,-0.45)
\psdots[dotsize=0.2,fillstyle=solid,dotstyle=o](8.8,0.35)
\psline[linewidth=0.04cm](11.2,0.35)(12.0,1.15)
\psdots[dotsize=0.2,fillstyle=solid,dotstyle=o](9.6,0.35)
\psdots[dotsize=0.2,fillstyle=solid,dotstyle=o](1.6,0.35)
\psdots[dotsize=0.2,fillstyle=solid,dotstyle=o](5.6,0.35)
\psline[linewidth=0.04cm](7.2,0.35)(6.4,1.15)
\psdots[dotsize=0.2,fillstyle=solid,dotstyle=o](6.4,1.15)
\psdots[dotsize=0.2,fillstyle=solid,dotstyle=o](12.0,1.15)
\psline[linewidth=0.04cm](0.0,0.35)(0.0,-1.25)
\end{pspicture} 
}
.$$
\smallskip\par\noindent
The four $2$-cells occur in 2 symmetric pairs of (curved) squares and triangles. We can represent them, and their seven edges, as follows:

\bigskip\par\noindent
$$
\scalebox{0.4} % Change this value to rescale the drawing.
{
\begin{pspicture}(0,-3.61)(25.583431,3.6017158)
\definecolor{color1762}{rgb}{0.0,0.0,0.6}
\psline[linewidth=0.04cm](12.791716,1.21)(12.791716,0.41)
\psline[linewidth=0.04cm](12.791716,1.21)(13.591716,2.01)
\psline[linewidth=0.04cm](12.791716,1.21)(11.991715,2.01)
\psline[linewidth=0.04cm](12.791716,1.21)(12.791716,2.01)
\psdots[dotsize=0.2,fillstyle=solid,dotstyle=o](13.591716,2.01)
\psdots[dotsize=0.2](11.991715,2.01)
\psdots[dotsize=0.2](12.791716,1.21)
\psline[linewidth=0.04cm](12.791716,-1.59)(12.791716,-2.39)
\psline[linewidth=0.04cm](12.791716,-1.59)(13.591716,-0.79)
\psline[linewidth=0.04cm](12.791716,-1.59)(11.991715,-0.79)
\psline[linewidth=0.04cm](12.791716,-1.59)(12.791716,-0.79)
\psdots[dotsize=0.2,fillstyle=solid,dotstyle=o](11.991715,-0.79)
\psdots[dotsize=0.2](13.591716,-0.79)
\psdots[dotsize=0.2](12.791716,-1.59)
\psline[linewidth=0.04cm,linecolor=color1762](14.391716,3.21)(14.391716,3.21)
\psline[linewidth=0.04cm,linecolor=color1762](14.391716,3.21)(15.191716,2.41)
\psline[linewidth=0.04cm,linecolor=color1762](14.791716,2.81)(15.191716,3.21)
\psline[linewidth=0.04cm,linecolor=color1762](15.191716,2.41)(15.591716,2.81)
\psline[linewidth=0.04cm,linecolor=color1762](15.191716,2.41)(15.191716,2.01)
\psdots[dotsize=0.2,linecolor=color1762](14.791716,2.81)
\psdots[dotsize=0.2,linecolor=color1762](15.191716,2.41)
\psdots[dotsize=0.2,linecolor=color1762](14.391716,3.21)
\psdots[dotsize=0.2,linecolor=color1762,fillstyle=solid,dotstyle=o](15.591716,2.81)
\psline[linewidth=0.04cm,linecolor=color1762](10.391716,3.21)(10.791716,2.81)
\psline[linewidth=0.04cm,linecolor=color1762](9.991715,2.81)(10.391716,2.41)
\psline[linewidth=0.04cm,linecolor=color1762](10.391716,2.41)(11.191716,3.21)
\psline[linewidth=0.04cm,linecolor=color1762](10.391716,2.41)(10.391716,2.01)
\psdots[dotsize=0.2,linecolor=color1762](10.391716,2.41)
\psdots[dotsize=0.2,linecolor=color1762](10.791716,2.81)
\psdots[dotsize=0.2,linecolor=color1762](9.991715,2.81)
\psdots[dotsize=0.2,linecolor=color1762,fillstyle=solid,dotstyle=o](11.191716,3.21)
\psline[linewidth=0.04cm,linecolor=color1762](10.391716,-2.39)(10.791716,-2.79)
\psline[linewidth=0.04cm,linecolor=color1762](9.991715,-2.79)(10.391716,-3.19)
\psline[linewidth=0.04cm,linecolor=color1762](10.391716,-3.19)(11.191716,-2.39)
\psline[linewidth=0.04cm,linecolor=color1762](10.391716,-3.19)(10.391716,-3.59)
\psdots[dotsize=0.2,linecolor=color1762](10.391716,-3.19)
\psdots[dotsize=0.2,linecolor=color1762](10.791716,-2.79)
\psdots[dotsize=0.2,linecolor=color1762](11.191716,-2.39)
\psdots[dotsize=0.2,linecolor=color1762,fillstyle=solid,dotstyle=o](9.991715,-2.79)
\psline[linewidth=0.04cm,linecolor=color1762](14.391716,-2.39)(15.191716,-3.19)
\psline[linewidth=0.04cm,linecolor=color1762](14.791716,-2.79)(15.191716,-2.39)
\psline[linewidth=0.04cm,linecolor=color1762](15.191716,-3.19)(15.591716,-2.79)
\psline[linewidth=0.04cm,linecolor=color1762](15.191716,-3.19)(15.191716,-3.59)
\psdots[dotsize=0.2,linecolor=color1762](14.791716,-2.79)
\psdots[dotsize=0.2,linecolor=color1762](15.191716,-3.19)
\psdots[dotsize=0.2,linecolor=color1762](15.591716,-2.79)
\psdots[dotsize=0.2,linecolor=color1762,fillstyle=solid,dotstyle=o](14.391716,-2.39)
\psdiamond[linewidth=0.04,dimen=outer](12.791716,0.01)(3.6,3.6)
\psline[linewidth=0.04cm](9.231716,0.01)(16.351715,0.01)
\psline[linewidth=0.04cm,linecolor=color1762](13.591716,0.61)(13.991715,0.21)
\psline[linewidth=0.04cm,linecolor=color1762](13.991715,0.21)(13.991715,0.61)
\psline[linewidth=0.04cm,linecolor=color1762](13.991715,0.21)(14.391716,0.61)
\psline[linewidth=0.04cm,linecolor=color1762](13.991715,0.21)(13.991715,-0.19)
\psdots[dotsize=0.2,linecolor=color1762](13.991715,0.21)
\psdots[dotsize=0.2,linecolor=color1762,fillstyle=solid,dotstyle=o](13.591716,0.61)
\psdots[dotsize=0.2,linecolor=color1762,fillstyle=solid,dotstyle=o](14.391716,0.61)
\psline[linewidth=0.04cm](3.1917157,-0.39)(3.1917157,-1.19)
\psline[linewidth=0.04cm](3.1917157,-0.39)(3.9917157,0.41)
\psline[linewidth=0.04cm](3.1917157,-0.39)(2.3917158,0.41)
\psdots[dotsize=0.2](2.3917158,0.41)
\psdots[dotsize=0.2](3.1917157,-0.39)
\psline[linewidth=0.04cm](3.9917157,0.41)(3.1917157,1.21)
\psline[linewidth=0.04cm](3.9917157,0.41)(4.7917156,1.21)
\psdots[dotsize=0.2](4.7917156,1.21)
\psdots[dotsize=0.2](3.9917157,0.41)
\psline[linewidth=0.04cm,linecolor=color1762](1.1917157,2.81)(1.5917157,2.41)
\psline[linewidth=0.04cm,linecolor=color1762](1.5917157,2.41)(1.9917158,2.81)
\psline[linewidth=0.04cm,linecolor=color1762](1.5917157,2.41)(1.5917157,2.01)
\psdots[dotsize=0.2,linecolor=color1762](1.5917157,2.41)
\psdots[dotsize=0.2,linecolor=color1762](1.9917158,2.81)
\psline[linewidth=0.04cm,linecolor=color1762](1.1917157,-2.79)(1.5917157,-3.19)
\psline[linewidth=0.04cm,linecolor=color1762](1.5917157,-3.19)(1.9917158,-2.79)
\psline[linewidth=0.04cm,linecolor=color1762](1.5917157,-3.19)(1.5917157,-3.59)
\psdots[dotsize=0.2,linecolor=color1762](1.5917157,-3.19)
\psdots[dotsize=0.2,linecolor=color1762](1.1917157,-2.79)
\psline[linewidth=0.04cm,linecolor=color1762](5.591716,3.21)(5.991716,2.81)
\psline[linewidth=0.04cm,linecolor=color1762](5.1917157,2.81)(5.591716,2.41)
\psline[linewidth=0.04cm,linecolor=color1762](5.591716,2.41)(6.3917155,3.21)
\psline[linewidth=0.04cm,linecolor=color1762](5.591716,2.41)(5.591716,2.01)
\psdots[dotsize=0.2,linecolor=color1762](5.591716,2.41)
\psdots[dotsize=0.2,linecolor=color1762](5.991716,2.81)
\psdots[dotsize=0.2,linecolor=color1762](5.1917157,2.81)
\psdots[dotsize=0.2,linecolor=color1762,fillstyle=solid,dotstyle=o](6.3917155,3.21)
\psdiamond[linewidth=0.04,dimen=outer](3.5917158,0.01)(3.6,3.6)
\psline[linewidth=0.04cm,linecolor=color1762](5.591716,-2.39)(5.991716,-2.79)
\psline[linewidth=0.04cm,linecolor=color1762](5.1917157,-2.79)(5.591716,-3.19)
\psline[linewidth=0.04cm,linecolor=color1762](5.591716,-3.19)(6.3917155,-2.39)
\psline[linewidth=0.04cm,linecolor=color1762](5.591716,-3.19)(5.591716,-3.59)
\psdots[dotsize=0.2,linecolor=color1762](5.591716,-3.19)
\psdots[dotsize=0.2,linecolor=color1762](5.991716,-2.79)
\psdots[dotsize=0.2,linecolor=color1762](6.3917155,-2.39)
\psdots[dotsize=0.2,linecolor=color1762,fillstyle=solid,dotstyle=o](5.1917157,-2.79)
\psline[linewidth=0.04cm](22.391716,-0.39)(22.391716,-1.19)
\psline[linewidth=0.04cm](22.391716,-0.39)(21.591715,0.41)
\psline[linewidth=0.04cm](22.391716,-0.39)(23.191715,0.41)
\psdots[dotsize=0.2](23.191715,0.41)
\psdots[dotsize=0.2](22.391716,-0.39)
\psline[linewidth=0.04cm](21.591715,0.41)(22.391716,1.21)
\psline[linewidth=0.04cm](21.591715,0.41)(20.791716,1.21)
\psdots[dotsize=0.2](20.791716,1.21)
\psdots[dotsize=0.2](21.591715,0.41)
\psline[linewidth=0.04cm,linecolor=color1762](24.391716,-2.79)(23.991716,-3.19)
\psline[linewidth=0.04cm,linecolor=color1762](23.991716,-3.19)(23.591715,-2.79)
\psline[linewidth=0.04cm,linecolor=color1762](23.991716,-3.19)(23.991716,-3.59)
\psdots[dotsize=0.2,linecolor=color1762](23.991716,-3.19)
\psdots[dotsize=0.2,linecolor=color1762](23.591715,-2.79)
\psline[linewidth=0.04cm,linecolor=color1762](24.391716,2.81)(23.991716,2.41)
\psline[linewidth=0.04cm,linecolor=color1762](23.991716,2.41)(23.591715,2.81)
\psline[linewidth=0.04cm,linecolor=color1762](23.991716,2.41)(23.991716,2.01)
\psdots[dotsize=0.2,linecolor=color1762](23.991716,2.41)
\psdots[dotsize=0.2,linecolor=color1762](24.391716,2.81)
\psline[linewidth=0.04cm,linecolor=color1762](19.591715,-2.39)(19.191715,-2.79)
\psline[linewidth=0.04cm,linecolor=color1762](19.991716,-2.79)(19.591715,-3.19)
\psline[linewidth=0.04cm,linecolor=color1762](19.591715,-3.19)(18.791716,-2.39)
\psline[linewidth=0.04cm,linecolor=color1762](19.591715,-3.19)(19.591715,-3.59)
\psdots[dotsize=0.2,linecolor=color1762](19.591715,-3.19)
\psdots[dotsize=0.2,linecolor=color1762](19.191715,-2.79)
\psdots[dotsize=0.2,linecolor=color1762](19.991716,-2.79)
\psdots[dotsize=0.2,linecolor=color1762,fillstyle=solid,dotstyle=o](18.791716,-2.39)
\psdiamond[linewidth=0.04,dimen=outer](21.991716,0.01)(3.6,3.6)
\psline[linewidth=0.04cm,linecolor=color1762](19.991716,3.21)(19.591715,2.81)
\psline[linewidth=0.04cm,linecolor=color1762](20.391716,2.81)(19.991716,2.41)
\psline[linewidth=0.04cm,linecolor=color1762](19.991716,2.41)(19.191715,3.21)
\psline[linewidth=0.04cm,linecolor=color1762](19.991716,2.41)(19.991716,2.01)
\psdots[dotsize=0.2,linecolor=color1762](19.991716,2.41)
\psdots[dotsize=0.2,linecolor=color1762](19.591715,2.81)
\psdots[dotsize=0.2,linecolor=color1762](19.191715,3.21)
\psdots[dotsize=0.2,linecolor=color1762,fillstyle=solid,dotstyle=o](20.391716,2.81)
\end{pspicture} 
}
$$
Again we see the extreme left (or right) edges and their common vertex form the subcomplex $K^u_{1,1}$. One of these edges is attached to the first $3$-ball and the other to the third $3$-ball. The common vertex is the subcomplex $K^u_{1,0}$; it is the only cell shared by all three $3$-balls.

\end{rem}

\section{Cellular chains of unital associahedra}

Let us fix  a commutative ring $\Bbbk$ for the whole section. In this section we always consider associahedra with the polytope cell structure, and unital associahedra with the new cell structure with cells obtained from products of associahedra and hypercubes as described in Proposition~\ref{puchau}.

Associahedra admit orientations such that the DG-operad $C_*(\mathtt{A}_\infty,\Bbbk)$ obtained by taking cellular chains on $\mathtt{A}_\infty$ is freely generated as a graded operad by the fundamental classes $\mu_{n}=[K_{n}]\in C_{n-2}(K_{n},\Bbbk)$, $n\geq 2$, and the differential is
\begin{equation}\label{dm}d(\mu_{n})=\sum_{
\begin{array}{c}\\[-15pt]
\scriptstyle p+q-1=n\\[-1.5mm]
\scriptstyle 1\leq i\leq p
\end{array}
} 
(-1)^{
q%(
p%+s)
+(q-1)(i-1)
%qp-(q-1)(i+r-1)+t(r-1)
}
\mu_{p}\circ_{i}\mu_{q}.
\end{equation}
This is the operad for $A_{\infty}$-algebras, cf.\ \cite[Proposition 12.3]{fcplt} and \cite[\S 3]{acwcpainfa}. Our slightly different sign conventions only indicate that we may be choosing different orientations on some associahedra.

Denote $\mathtt{O}$ the DG-operad whose underlying graded operad is freely generated by the operations
$$\mu_{n+m}^{S}\in\mathtt{O}(n),\quad n,m\geq 0,\quad S\subset\ord{n+m},\quad |S|=m,\quad (n,m)\neq (0,0), (1,0),$$
of degree $$|\mu_{n+m}^{S}|=2m+n-2.$$
The differential is defined as follows, $(n,m)\neq (1,1)$:
\begin{align}\label{dmu}d(\mu_{n+m}^{S})&=\sum_{
\begin{array}{c}\\[-15pt]
\scriptstyle p+q-1=n\\[-1.5mm]
\scriptstyle s+t=m\\[-1mm]
\scriptstyle 1\leq i\leq p\\[-1mm]
\scriptstyle S_{1}\circ_{i}S_{2}=S
\end{array}
} 
(-1)^{
(q+t)%(
p%+s)
+(q+t-1)(i+r-1)+t(r-1)
%qp-(q-1)(i+r-1)+t(r-1)
}
\mu_{p+s}^{S_{1}}\circ_{i}\mu_{q+t}^{S_{2}},\\
\nonumber d(\mu_{1+1}^{\{i\}})&=\mu_{2+0}^{\varnothing}\circ_i\mu_{0+1}^{\{1\}}-u,\qquad i\in\{1,2\}.
%\nonumber d(\mu_{1+1}^{\{2\}})&=\mu_{2+0}^{\varnothing}\circ_2\mu_{0+1}^{\{1\}}-\id.
\end{align}
%where $t=\card S_2$ and
Here $r$ is such that the $i^{\text{th}}$ element of the complement of $S_1$ lies between the
$(r-1)^{\text{st}}$ and $r^{\text{th}}$ elements of $S_1$.

We may also describe a $\Bbbk$-linear basis $\{\mu_T\}_{T}$ of $\mathtt O$ indexed by trees $T$ with black and white corks. 
Let $T$ be a tree of height $h$, with $n$ leaves, and $m$ corks, of which $m_b$ are black and $m_w$ are white. We exclude trees with any degree two vertices, and also the case $h=m_b=1$. 
\begin{enumerate}\item If $h=2$ and $m_w=0$ then we have
$$
\mu_T=\mu_{n+m}^S
$$
where $S\subseteq\ord{n+m}$ indicates which level 2 vertices are corks.
\item If $h=1$ and $m_w=1$,
 i.e.\ $T=\raisebox{-0.4ex}[0ex][-0.4ex]{\begin{sideways}$\multimap$\end{sideways}}$,
$$
\mu_T=\mu_{0+1}^{\{1\}}.
$$ 
\item If $T=T'\circ_i T''$ we define inductively
$$
\mu_T=\mu_{T'}\circ_i \mu_{T''}.
$$ 
\end{enumerate}
We observe that $\mu_T$ is an element of degree $2m+n-2-\card{I(T)}+m_b$ 
in $\mathtt O(n)$.

Over a field of characteristic zero, this operad $\mathtt O$, as pointed out in \cite{ckdt}, is the operad for homotopy unital $A_{\infty}$-algebras in the sense of Fukaya--Oh--Ohta--Ono   \cite{fooo1,fooo2}. In the next theorem we show that $\mathtt O$ is just the chains on the topological operad $\mathtt{uA}_{\infty}$ formed by the unital associahedra; this construction works over any commutative ring.

\begin{thm}
There is an isomorphism of differential graded operads 
\begin{align*}
\mathtt O&\st{\cong}\To C_*(\mathtt{uA}_\infty,\Bbbk),\\
\mu_{0+1}^{\{1\}}&\;\mapsto\; [K^u_{0,1}],
\\\mu_{n+m}^S&\;\mapsto\; (-1)^{m(m-1)/2}[\bar f_S],\qquad(n,m)\neq(0,1).
\end{align*}
\end{thm}

\begin{proof}
Since $\mathtt O$ is free as a graded operad, the underlying graded operad morphism is well defined. The descriptions above for the $\Bbbk$-basis of $\mathtt O$ and for the cellular structure of unital associahedra show that it is a graded isomorphism. In order to show that it also commutes with the differentials we choose the orientations $I\in C_{1}([0,1],\Bbbk)$ with $d(I)=[1]-[0]$, and
$$[K_{n+m}\times [0,1]^m]=[K_{n+m}]\otimes I^{\otimes m};$$
and 
 observe that
\begin{align*}
%{}[K_{n+m}\times [0,1]^m]={}&[K_{n+m}]\otimes I^{\otimes m},
%\\
d([K_{n+m}\times [0,1]^m])={}&
\sum_{\begin{array}{c}\\[-5mm]\scriptstyle\alpha+\beta-1=n+m\\[-1.5mm]\scriptstyle\alpha,\beta\geq 2\\[-1.5mm] \scriptstyle1\leq\kappa\leq \alpha\end{array}}
(-1)^{\beta\alpha+(\beta-1)\kappa}
([K_\alpha]\circ_\kappa[K_\beta])\otimes I^{\otimes m}
\\
&\hspace{-10pt}{}+\sum_{\lambda=1}^m(-1)^{n+m+\lambda-1}[K_{n+m}]
\otimes I^{\otimes(\lambda-1)}\otimes([1]-[0])\otimes I^{\otimes(m-\lambda)}.
\end{align*}
The image under 
$({\bar f_S})_*$ (we denote  morphisms induced by cellular maps on cellular chains in this way) of
the first summation may be identified with
$$\sum_{
\begin{array}{c}\\[-5mm]
\scriptstyle p+q-1=n\\[-1.5mm]
\scriptstyle s+t=m\\[-1mm]
\scriptstyle (q,t)\neq(0,1)\\[-1mm]
\scriptstyle 1\leq i\leq p\\[-1mm]
\scriptstyle S_{1}\circ_{i}S_{2}=S
\end{array}
} 
(-1)^{
(q+t)(p+s)+(q+t-1)(i+r-1)+t(s-r+1)+s(q+t)
}
[\bar f_{S_{1}}]\circ_{i}[\bar f_{S_{2}}]$$
%$$
%\sum_{\begin{array}{c}\scriptstyle\alpha+\beta-1=n+m\\ \scriptstyle1\leq\kappa\leq \alpha\end{array}}
%(-1)^{\alpha\beta+\kappa(\beta-1)}
%([K_\alpha]\circ_\kappa[K_\beta])\otimes I^{\otimes m}
%$$
by the commutativity of \eqref{filtr1} in Proposition \ref{filtri}.
Here $\alpha=p+s$, $\beta=q+t$, $\kappa=i+r-1$, and furthermore  $(-1)^{t(s-r+1)}$ and  $(-1)^{s(q+t)}$ are the signs associated to the maps $\sigma_{
%s,t,
r}$ \eqref{sigma} and the swap in \eqref{filtr1}. This sign differs from that in \eqref{dmu} by precisely 
$$
(-1)^{st}\;=\;(-1)^{m(m-1)/2}\cdot(-1)^{s(s-1)/2}\cdot(-1)^{t(t-1)/2}.
$$

In the second summation, if $(n,m)\neq (1,1)$ the terms containing $[0]$ vanish on applying $({\bar f_S})_*$ by 
Proposition \ref{puchau}, since the pushout identifies $K_{n+m}\times\partial_\lambda^-([0,1]^{m-1})$ with lower dimensional cells. 
If $(n,m)=(1,1)$ the image of the term containing $[0]$ is 
$-u$, see $K^{u}_{1,1}$ in Figure \ref{firstfour}. The image of the terms containing $[1]$ are identified with 
$$\sum_{
\begin{array}{c}
%\scriptstyle p-1=n\\[-1.5mm]
%\scriptstyle s+1=m\\[-1mm]
%\scriptstyle (t,q)\neq(1,0)\\[-1mm]
\\[-5mm]
\scriptstyle 1\leq i\leq p\\[-.5mm]
\scriptstyle S_{1}\circ_{i}\{1\}=S
\end{array}
} 
(-1)^{
p+s+r-1
}
[\bar f_{S_{1}}]\circ_{i}[K^u_{0,1}]$$
%$$
%\sum_{\begin{array}{c}\scriptstyle\alpha+\beta-1=n+m\\ \scriptstyle1\leq\kappa\leq \alpha\end{array}}
%(-1)^{\alpha\beta+\kappa(\beta-1)}
%([K_\alpha]\circ_\kappa[K_\beta])\otimes I^{\otimes m}
%$$
using the diagram \eqref{filtr2} in Proposition \ref{filtri}.
Here $p-1=n$, $s+1=m$ and $r=\lambda$.
 This sign differs from that in \eqref{dmu} for $(q,t)=(0,1)$ by precisely 
$$
(-1)^{s}\;=\; (-1)^{s^2}\;
=\;(-1)^{m(m-1)/2}\cdot(-1)^{s(s-1)/2}\cdot(-1)^{t(t-1)/2}.
$$
This finishes the proof that the isomorphism commutes with the differential structure.
\end{proof}

The following result is a consequence of this theorem and Corollary \ref{arecontractible}, since DG-operads with free underlying graded operad are cofibrant.

\begin{cor}
For any commutative ring $\Bbbk$, the operad $\mathtt O\cong C_*(\mathtt{uA}_\infty,\Bbbk)$ is a cofibrant resolution of the DG-operad for unital DG-algebras.
\end{cor}

Lyubashenko sketches a purely algebraic proof of this result in \cite[1.11]{huainfa}.

% ----------------------------------------------------------------
%\bibliographystyle{amsalpha}
%\bibliography{../Fernando}

\begin{thebibliography}{FOOO09b}

\bibitem[BM03]{ahto}
C.~Berger and I.~Moerdijk, \emph{Axiomatic homotopy theory for operads},
  Comment. Math. Helv. \textbf{78} (2003), no.~4, 805--831.

\bibitem[BM06]{tbvrommc}
\bysame, \emph{The {B}oardman-{V}ogt resolution of operads in monoidal model
  categories}, Topology \textbf{45} (2006), no.~5, 807--849. \MR{2248514
  (2008e:18016)}

\bibitem[BV73]{hiasts}
J.~M. Boardman and R.~M. Vogt, \emph{Homotopy invariant algebraic structures on
  topological spaces}, vol. 347, Springer-Verlag, Berlin, 1973, Lecture Notes
  in Mathematics.

\bibitem[FOOO09a]{fooo1}
K.~Fukaya, Y.-G. Oh, H.~Ohta, and K.~Ono, \emph{Lagrangian intersection {F}loer
  theory: anomaly and obstruction. {P}art {I}}, AMS/IP Studies in Advanced
  Mathematics, vol.~46, American Mathematical Society, Providence, RI, 2009.

\bibitem[FOOO09b]{fooo2}
\bysame, \emph{Lagrangian intersection {F}loer theory: anomaly and obstruction.
  {P}art {II}}, AMS/IP Studies in Advanced Mathematics, vol.~46, American
  Mathematical Society, Providence, RI, 2009.

\bibitem[HM10]{ckdt}
J.~Hirsh and J.~Mill{\`e}s, \emph{Curved {K}oszul duality theory},
  \texttt{arXiv:1008.5368v1 [math.KT]}, 2010.

\bibitem[Lyu11]{huainfa}
V.~Lyubashenko, \emph{Homotopy unital {$A_\infty$}-algebras}, J. Algebra
  \textbf{329} (2011), 190--212.

\bibitem[May72]{tgils}
J.~P. May, \emph{The geometry of iterated loop spaces}, Springer-Verlag,
  Berlin, 1972, Lecture Notes in Mathematics, Vol. 271.

\bibitem[MS06]{acwcpainfa}
M.~Markl and S.~Shnider, \emph{Associahedra, cellular {$W$}-construction and
  products of {$A_\infty$}-algebras}, Trans. Amer. Math. Soc. \textbf{358}
  (2006), no.~6, 2353--2372 (electronic).

\bibitem[Mur11]{htnso}
F.~Muro, \emph{Homotopy theory of non-symmetric operads}, Algebr. Geom. Topol.
  \textbf{11} (2011), 1541--1599.

\bibitem[Sei08]{fcplt}
P.~Seidel, \emph{Fukaya categories and {P}icard-{L}efschetz theory}, Zurich
  Lectures in Advanced Mathematics, European Mathematical Society (EMS),
  Z{\"u}rich, 2008.

\bibitem[Sta63]{hahs}
J.~D. Stasheff, \emph{Homotopy associativity of {$H$}-spaces. {I}, {II}},
  Trans. Amer. Math. Soc. 108 (1963), 275-292; ibid. \textbf{108} (1963),
  293--312.

\end{thebibliography}
\providecommand{\bysame}{\leavevmode\hbox to3em{\hrulefill}\thinspace}
\providecommand{\MR}{\relax\ifhmode\unskip\space\fi MR }
\providecommand{\MRhref}[2]{%
  \href{http://www.ams.org/mathscinet-getitem?mr=#1}{#2}
}
\providecommand{\href}[2]{#2}

\end{document}